\documentclass[a4]{amsart}

\input xypic 
\input xy 
\xyoption{all} 
\usepackage{amssymb} 
\usepackage{bbm}
\usepackage{tikz-cd}

\usepackage{tikz}
\usetikzlibrary{positioning}

\usepackage{float}
\usepackage{caption}
\usepackage{graphicx}
\usepackage{accents}
\usepackage{mathtools}

\usepackage[bookmarksnumbered, bookmarksopen, backref]{hyperref}
\newtheorem{theorem}{Theorem}[section]
\newtheorem{proposition}[theorem]{Proposition}
\newtheorem{lemma}[theorem]{Lemma}
\newtheorem{corollary}[theorem]{Corollary}

\theoremstyle{definition}
\newtheorem{definition}[theorem]{Definition}
\newtheorem{remark}[theorem]{Remark}
\newtheorem{example}[theorem]{Example}

\newtheorem{problem}[theorem]{Problem}

\newcounter{RomanNumber}

\newcommand{\tbcap}{\ensuremath{\cap\kern-0.54em|\kern0.36em}} %right
\newcommand{\tbccap}{\ensuremath{\cap\kern-0.59em|\kern0.36em}}
\newcommand{\tbcccap}{\ensuremath{\cap\kern-0.639em|\kern0.36em}}
\newcommand{\tccap}{\ensuremath{\cap\kern-0.67em|\kern0.36em}}
\newcommand{\tcap}{\ensuremath{\cap\kern-0.7em|\kern0.36em}}
\newcommand{\tdcap}{\ensuremath{\cap\kern-0.72em|\kern0.36em}}
\newcommand{\ttcap}{\ensuremath{\cap\kern-0.74em|\kern0.36em}}
\newcommand{\tucap}{\ensuremath{\cap\kern-0.79em|\kern0.36em}} %left

\newcommand{\testcap}{\mathrel{\vcenter{\offinterlineskip
\hbox{$\cap$}\vskip-1.8ex\hbox{$\kern0.2em | \kern0.15em$}}}}

\newcounter{bean}

\newcommand{\namedright}[3]{\ensuremath{#1\stackrel{#2}
 {\longrightarrow}#3}}
\newcommand{\nameddright}[5]{\ensuremath{#1\stackrel{#2}
 {\longrightarrow}#3\stackrel{#4}{\longrightarrow}#5}}

\newcommand{\larrow}{\relbar\!\!\relbar\!\!\rightarrow}
\newcommand{\llarrow}{\relbar\!\!\relbar\!\!\larrow}

\newcommand{\llnameddright}[5]{\ensuremath{#1\stackrel{#2}
 {\llarrow}#3\stackrel{#4}{\llarrow}#5}}

\newcommand{\qqed}{\hfill\Box}

\begin{document}
\begin{sloppypar}
%%% Title

\title{Comparison techniques on inert top cell attachments} 
%author two information

\author{Ruizhi Huang} 
\address{Institute of Mathematics, Academy of Mathematics and Systems Science, 
 Chinese Academy of Sciences, Beijing 100190, China} 
\email{huangrz@amss.ac.cn} 
   \urladdr{https://sites.google.com/site/hrzsea/}
   \thanks{}
   
%\author{Stephen Theriault}
%\address{School of Mathematics, University of Southampton, Southampton 
 %  SO17 1BJ, United Kingdom}
%\email{S.D.Theriault@soton.ac.uk}

%    \subjclass is required.
\date{}
%\thanks{}

%%% Abstract

\begin{abstract} 
We establish various criteria for the inertness of the top cell attachments of Poincar\'{e} duality complexes through nonzero degree maps, algebraic intersection theory and various types of homotopy fibrations. Many examples are provided, including specific surgery, homogeneous spaces and low dimensional manifolds. Additionally, we propose eight open problems. 
\end{abstract}

\maketitle

\setcounter{tocdepth}{1}
\tableofcontents
\newpage

%%%%%%%%%%%%%%%%%%%%%%%%%%%%%%%%%%%%%%%%%%%%%%%%%%%%%%%%%%%%%%%%%%%%%%%

\section{Introduction}
In 1982, F\'{e}lix-Halperin-Thomas \cite{FHT82} introduced the classical notion of rational inertness. A map $h: S^{n-1}\stackrel{}{\longrightarrow} X$ is called {\it rationally inert} if the inclusion map $X\stackrel{i}{\longrightarrow} X\cup_h CS^{n-1}$ of $X$ into the mapping cone of $h$ induces an epimorphism in rational homotopy groups, or equivalently, if the loop map $\Omega i$ has a right homotopy inverse after rationalization. Rationally inert maps have been widely studied in rational homotopy theory \cite{FT89, HL87, HL95, HeL96, Bub05} and have important applications, such as in the study of the growth of rational homotopy groups \cite{FHT07} and in solving a higher dimensional Whitehead's asphericity problem \cite{Ani86}.

Recently, Theriault \cite{The24a} introduced an integral generalization of rational inertness, which has stimulated fruitful progress in unstable homotopy theory. The map $h: S^{n-1}\stackrel{}{\longrightarrow} X$ is called {\it inert} if the loop map $\Omega X\stackrel{\Omega i}{\longrightarrow} \Omega(X\cup_h CS^{n-1})$ has a right homotopy inverse. In this situation, an important result by Beben-Theriault \cite{BT22} identifies the homotopy fibre of $i$ and proves a loop space decomposition of $X$ in terms of $S^{n-1}$ and $X\cup_h CS^{n-1}$. The result provides an explicit way to manage the homotopy information of an inert attaching map. Many interesting applications have arisen from this deep decomposition, such as in the study of the unstable homotopy of Poincar\'{e} duality complexes \cite{Hua22, HT22, Hua23a, Hua24, The24a, The24b}, and surprisingly, in addressing problems in rational homotopy theory by unstable homotopy techniques \cite{Che23a, Che23b, Hua23b, HT24b}. 

The inertness for the top cell of a Poincar\'{e} duality complex is of special interest. A fundamental result of Halperin and Lemaire \cite{HL87} shows that the attaching map for the top cell of any Poincar\'{e} duality complex is rationally inert unless its rational cohomology algebra is generated by a single element. The corresponding problem in integral and local contexts has recently received growing attention. Beben-Theriault \cite{BT14, BT22} demonstrated that the attaching maps for the top cells of certain highly connected Poincar\'{e} duality complexes are inert. For more systematic treatments, Theriault in \cite{The24a} showed that the connected sum operation preserves the inertness property, and in \cite{The24b} proved an elegant criterion of the inertness property for a large family of Poincar\'{e} duality complexes.

In this paper, we study the inertness property for the top cell of a Poincar\'{e} duality complex by comparison techniques. The results provide various criteria for inertness by comparing Poincar\'{e} duality complexes through nonzero degree maps, algebraic intersection theory and various kinds of homotopy fibrations. Many examples are given including a type of surgery, Stiefel manifolds, complete flag manifolds and low dimensional manifolds. Based on the results, we propose eight open problems. 

In the following we start with some conventions of notations and terminologies and then proceed to detail our main results. Occasionally, we state simplified versions of results for clarity, with complete versions available in the main text.

$\, $

\noindent{\bf Conventions.}
\begin{itemize}
\item Following \cite{The24a}, A map $h: A\stackrel{}{\longrightarrow} X$ of $CW$-complexes is called {\it inert} if the inclusion map $X\stackrel{i}{\longrightarrow} X\cup_h CA$ of $X$ into the mapping cone has a right homotopy inverse after looping. 
Let $A\stackrel{h}{\longrightarrow} X\stackrel{\varphi}{\longrightarrow} Y$ be a homotopy cofibration. The map $h$ is inert if and only if $\Omega \varphi$ has a right homotopy inverse. 
\item In this paper, all spaces are assumed to be connected, pointed, and have the homotopy type of a $CW$-complex. All Poincar\'{e} duality complexes are assumed to have a $CW$-structure with a single top cell. 
\item When applying localization, a space is supposed to be {\it nilpotent}, that is, its fundamental group is nilpotent and acts nilpotently on the higher homotopy groups. 
\item Denote by $1_X$ the identity map $X\stackrel{=}{\longrightarrow} X$.
\item Denote by $[M]$ the {\it fundamental class} of a Poincar\'{e} duality complex $M$, and by ${\rm dim}(M)$ the dimension of $M$. 
\item Let $M$ be an $n$-dimensional Poincar\'{e} duality complex with a single top cell. Denote by $M_0$ its $(n-1)$-skeleton. When the dimension of $M$ is implicit, we may refer to $M_0$ as the {\it lower skeleton} of $M$.  
\item Let $f: M\stackrel{}{\longrightarrow} N$ be a map between Poincar\'{e} duality complexes of the same dimension. Denote by $f_0: M_0\stackrel{}{\longrightarrow} N_0$ the restriction map of $f$ on the lower skeletons through the $CW$-approximation. 
\end{itemize}

$\, $
%-------------------------------------------

\noindent{\bf Inertness via nonzero degree maps.} 
The study of nonzero degree maps has been a classical topic in algebraic and geometric topology since the work of Brouwer \cite{Bro11} in 1911. It is closely connected to the geometry of Riemannian manifolds through various degree theorems \cite{Gro82}. Furthermore, extensive efforts \cite{Wan02, DW03, DW04} have been made to compute mapping degrees between manifolds. 

In our study, we establish fundamental criteria for inertness by comparing two Poincar\'{e} duality complexes of the same dimension through a nonzero degree map. This exploration highlights a novel connection between nonzero degree maps and homotopy theory. 

\begin{theorem}\label{inertdeg1thmintro}
Let $M$ and $N$ be two Poincar\'{e} duality complexes of the same dimension such that both $M$ and $N$ have a single top cell. Let $h_M$ and $h_N$ be the attaching maps of the top cells of $M$ and $N$, respectively. 

If there exists a degree one map $f: M\stackrel{}{\longrightarrow} N$, then the following hold:
\begin{itemize}
\item[(1).] suppose that the restriction map $f_0: M_0\stackrel{}{\longrightarrow} N_0$ on the lower skeletons is inert. Then $h_M$ is inert if and only if $h_N$ is inert; 
\item[(2).] suppose that $\Omega f$ has a right homotopy inverse. Then if $h_M$ is inert so is $h_N$;
\item[(3).] suppose that there is a homotopy cofibration $A\stackrel{a_0}{\longrightarrow} M_0 \stackrel{f_0}{\longrightarrow} N_0$ for some space $A$ and some inert map $a_0$. Then if $h_N$ is inert so is $h_M$.
\end{itemize}

Furthermore, suppose that $M$ and $N$ are nilpotent and at least one of them is simply connected. 
If there exists a degree $k$ map $f: M\stackrel{}{\longrightarrow} N$ with $k\neq 0$, then the three conclusions hold after localization away from all primes $p$ that divide $k$.
\end{theorem}

In Theorem \ref{inertdeg1thmintro}, the proof of statement (2) is easy as pointed out in Remark \ref{pushoutthm3}. We include it in the theorem as it is fundamental and serves as a useful criterion. For instance, it implies that if $h_M$ is inert and $N$ is a homotopy retraction of $M$, then $h_N$ is inert. This indicates that the inertness of the top cell attachments for Poincar\'{e} duality complexes of a fix dimension is preserved under homotopy retraction. Furthermore, statement (2) can be viewed as a dual statement of statement (3). 

The proofs of statements (1) and (3) of Theorem \ref{inertdeg1thmintro} are more substantial. A nonzero degree map corresponds to a homotopy pushout in the integral or local category, and we may investigate the inertness property within the broader context of homotopy pushouts. A critical step involves applying Beben-Theriault's decomposition \cite{BT22} with its naturality, which allows us to reduce the study of inertness around homotopy pushouts to the study of inertness around homotopy pullbacks. The latter is more manageable because homotopy pullbacks are preserved by the loop functor.

Theorem \ref{inertdeg1thmintro} can be strengthened when $f$ is a finite cover. In fact, it can be shown that in this case $h_M$ is inert if and only if $h_N$ is inert; see Proposition \ref{coverprop} for more details. 

Before introducing criteria for inertness in other contexts, we discuss two applications of Theorem \ref{inertdeg1thmintro}: studying inertness via algebraic intersection numbers and via comparison with a twisted product of spheres.

$\, $

%-------------------------------------------
\noindent{\bf Inertness via intersection numbers.} 
Motivated by the notion of intersection number in geometric topology and algebraic geometry, we can define an algebraic version of the intersection number for general Poincar\'{e} duality complexes. A relevant case is as follows. 
Let $f: A\times B\longrightarrow M$ be a map between Poincar\'{e} duality complexes such that ${\rm dim}(A)+{\rm dim}(B)= {\rm dim}(M)$. 
Denote by $[A]^\ast:=(f_{\ast}([A]))^\ast$ the {\it Poincar\'{e} dual} of the image of the fundamental class of $A$ in $M$, called {\it the Poincar\'{e} dual of $A$ in $M$} for simplicity. Similarly for $B$. 
The {\it (algebraic) intersection number} of $A$ and $B$ in $M$ is the integer
\[\label{int=1eqintro}
A \testcap B:= \langle [A]^\ast \cup [B]^\ast, [M]\rangle\in \mathbb{Z}.
\]
When $A \testcap B\neq 0$ we say that $A$ and $B$ {\it essentially intersect} in $M$. 

Denote by $A\stackrel{f_A}{\longrightarrow} M$ and $B\stackrel{f_B}{\longrightarrow} M$ the restriction maps of $f$ on the two factors. When $f_A$ and $f_B$ are embeddings of smooth manifolds, the two embeddings are transversal to each other up to isotopy. A fundamental result in the intersection theory implies that $A \testcap B$ is equal to the {\it geometric intersection number} of the two embeddings. This justifies the terminology. Indeed, the algebraic intersection number can be defined for any $f_A$ and $f_B$ by the same formula. 

In the algebraic setting, it turns out that the algebraic intersection number is equal to the degree of the map $f$ under mild conditions. This correspondence allows us to apply Theorem \ref{inertdeg1thmintro} (1) to study the inertness of the top cell attachment of $M$ through the map $f$. 

To state the result, we call a cohomology class $z\in H^\ast(M;\mathbb{Z})$ {\it $f$-primitive} if $f^\ast(z)$ has no terms in $H^{+}(A;\mathbb{Z})\otimes H^{+}(B;\mathbb{Z})\subseteq H^{\ast}(A\times B;\mathbb{Z})$. 

 \begin{theorem}\label{pdtintthmintro}
Let 
\[
f: A\times B\stackrel{}{\longrightarrow} M
\]
be a map between simply connected Poincar\'{e} duality complexes such that ${\rm dim}(B)>{\rm dim}(A)>0$. 
Suppose that $A\testcap B\neq 0$, and the Poincar\'{e} dual $[A]^\ast$ is $f$-primitive. 

If the restriction map 
$
f_0: (A\times B)_0\stackrel{}{\longrightarrow} M_0
$ 
is inert, 
then the attaching map for the top cell of $M$ is inert after localization away from all primes $p$ that divide $A\testcap B$.

In particular, when $A\testcap B=\pm 1$ the attaching map for the top cell of $M$ is inert. 
\end{theorem}

Theorem \ref{pdtintthmintro}, along with its general form Theorem \ref{pdtintthm}, shows how the inertness property can be inferred from an essential intersection. As mentioned earlier, the algebraic intersection is inspired by its geometric counterpart. In the smooth case, Theorem \ref{pdtintthmintro} suggests that we can investigate the inertness of the top cell attachment for a manifold based on the inertness of the top cell attachments for two submanifolds that essentially intersect. This approach is particularly applicable to manifolds with free actions and homogeneous spaces; see Section \ref{sec: int} for details. 
The underlying idea is consistent with the general philosophy of studying a manifold through its submanifolds and their interactions. An important related topic is Schubert calculus for intersection theory of flag manifolds; see the wonderful survey \cite{DZ22} by Duan-Zhao. 
Here, we can specifically prove the inertness property for Stiefel manifolds as concrete examples; a summary is provided in Theorem \ref{homthmintro} below. 

$\, $

%-------------------------------------------
\noindent{\bf Inertness via comparison with a twisted product of spheres.} 
The simplest examples of Poincar\'{e} duality complexes with inert top cell attachments could be products of spheres, which is a consequence of the classical Hilton-Milnor theorem. To study the inertness of the top cell attachment for a Poincar\'{e} duality complex, it is natural to compare it with a product of spheres. With Theorem \ref{inertdeg1thmintro} this suggests looking for a nonzero degree map between them.

To address a slightly general context, we can consider twisted product of spheres. 
Let $S^{n-m}\stackrel{}{\longrightarrow} D\stackrel{}{\longrightarrow} S^{m}$ be a homotopy fibration with a homotopy section. It can be shown that $D$ is a Poincar\'{e} duality complex such that $D_0\simeq S^{m} \vee S^{n-m}$. For such a homotopy fibration, the total space $D$ is called a {\it twisted product of spheres} and denoted by $S^{m}\widetilde{\times} S^{n-m}$. The product $S^{m}\times S^{n-m}$ is obviously a twisted product of spheres. It can also be shown that the top cell attachment for a twisted product of spheres is inert. Therefore, we can study the inertness property from a nonzero degree map onto a twisted product of spheres.

The following theorem, as a special case of Theorem \ref{detthm}, provides an example along this idea. It can be proved by Theorem \ref{inertdeg1thmintro} (3) combined with a cubic method developed by Theriault in \cite{The24b}.

\begin{theorem}\label{detthmintro}
Let 
\[
f: M\stackrel{}{\longrightarrow} S^{m}\widetilde{\times} S^{n-m}
\]
be a degree one map between simply connected Poincar\'{e} duality complexes such that
\[
f_0\simeq f_Y\vee 1_{S^{n-m}}: M_0\simeq Y\vee S^{n-m}\stackrel{}{\longrightarrow} S^{m} \vee S^{n-m}.
\]
If $\Omega f_Y$ has a right homotopy inverse, then the attaching map for the top cell of $M$ is inert.

Additionally, if the map $f$ is of degree $k$ with $k\neq 0$, then the attaching map for the top cell of $M$ is inert after localization away from all primes $p$ that divide $k$.
\end{theorem}

Let us discuss two special cases of Theorem \ref{detthmintro}. The first, stated in Corollary \ref{detcor}, can be summarized as follows: 
Suppose that $n>m+1> 2$ and $\pi_{n-1}(S^m)$ is a torsion group. 
Let $M$ be an $n$-dimensional simply connected Poincar\'{e} duality complex with $M_0\simeq Y\vee S^{n-m}$ for some complex $Y$. 
Assume there is a map $f_Y: Y\stackrel{}{\longrightarrow} S^{m}$ that {\it detects} the Poincar\'{e} dual of $S^{n-m}$ in $M$; meaning that the pullback of a generator of $H^m(S^m;\mathbb{Z})$ in $H^m(Y;\mathbb{Z})\subseteq H^{m}(M_0;\mathbb{Z})=H^m(M;\mathbb{Z})$ is the Poincar\'{e} dual of $[S^{n-m}]\in H_{n-m}(M_0;\mathbb{Z})=H_{n-m}(M;\mathbb{Z})$. 
If $\Omega f_Y$ has a right homotopy inverse, then the attaching map for the top cell of $M$ is inert after localization away from all primes $p$ that divide the order of $\pi_{n-1}(S^m)$.

For another special case, Theorem \ref{detthmintro} can be strengthened by the remarkable work \cite{BT14} of Beben-Theriault. 
Suppose that $M_0\simeq Y\vee S^{n-m}\simeq S^m\vee Z\vee S^{n-m}$ for some complex $Z$. If the intersection number of $S^m$ and $S^{n-m}$ in $M$ is $\pm 1$, one can construct a degree one map $M\stackrel{f}{\longrightarrow} D$ onto a Poincar\'{e} duality complex $D$ such that $H^\ast(D;\mathbb{Z})\cong H^\ast(S^m\times S^{n-m};\mathbb{Z})$. Note that $D$ does not necessarily fiber over a sphere. 
In this case, Beben-Theriault \cite{BT14} showed that the attaching map for the top cell of $M$ is inert; see Theorem \ref{exJthm} for an alternative proof using the comparison idea. In particular, this special case can be applied to show that the top cell attachments for most $(n-1)$-connected $2n$-dimensional and $(n-1)$-connected $(2n+1)$-dimensional Poincar\'{e} duality complexes are inert; see Example \ref{Huaex} below for illustrations. 

Theorem \ref{detthmintro} is different from the main result of \cite{The24b}. 
In the latter work, Theriault showed that if $M$ is $(m-1)$-connected with $n-m\geq m$ and $f_Y$ itself has a right homotopy inverse, then the attaching map for the top cell of $M$ is inert. In our context, we impose no restrictions on the connectivity of $M$ or the value of $m$, and we assume that $\Omega f_Y$ has a right homotopy inverse.   

The idea of comparing a Poincar\'{e} duality complex with a twisted product of spheres can be generalized. It is possible to compare a Poincar\'{e} duality complex with other candidates. For instance, consider a connected sum $M\#N$ such that the attaching map for the top cell of $N$ is inert. By comparing the connected sum $M\#N$ with $N$ through the canonical degree one projection $M\#N\stackrel{}{\longrightarrow} N$, we can prove that the attaching map for the top cell of $M\#N$ is inert. This reproduces a result of Theriault in \cite{The24a} by the comparing idea. As concrete applications, we show that many low dimensional manifolds satisfy the inertness property. See Section \ref{sec: ex1} for detailed proofs and Example \ref{Huaex} below for a summary of examples.

$\, $
%-------------------------------------------

\noindent{\bf Inertness via strict fibrations.} 
The comparison idea can be extended further. Besides comparing Poincar\'{e} duality complexes of the same dimension through nonzero degree maps, we can study the inertness property by comparing Poincar\'{e} duality complexes of different dimensions through fibrations. In particular, from the geometric structure of a strict fibration of  Poincar\'{e} duality complexes we are able to show that the inertness of the top cell attachment for the base space determines the inertness of the top cell attachment for the total space. 
\begin{theorem}[Theorem \ref{FEBthm} and Proposition \ref{FEBprop}]\label{FEBthmintro}
Let 
\[\label{FEBeqintro}
F\stackrel{}{\longrightarrow} E\stackrel{}{\longrightarrow} B
\]
 be a strict fibration of connected Poincar\'{e} duality complexes with a single top cell. Then the following hold:
\begin{itemize}
\item[(1).] if the attaching map for the top cell of $B$ is inert, then the attaching map for the top cell of $E$ is inert;
\item[(2).] if the fibration has a homotopy section, then the attaching map for the top cell of $E$ is inert.
\end{itemize}
Additionally, the assertions hold after localization at any set of primes.
\end{theorem}

Theorem \ref{FEBthmintro} provides a powerful criterion for inertness. For instance, according to the work of Beben-Theriault \cite{BT14}, it is known that the top cell attachments for most $(n-1)$-connected $2n$-dimensional and $(n-1)$-connected $(2n+1)$-dimensional Poincar\'{e} duality complexes are inert. 
Therefore, for any strict fibration over such a highly connected Poincar\'{e} duality complex with the fibre being a connected Poincar\'{e} duality complex, the top cell attachment for its total space is inert by Theorem \ref{FEBthmintro} (1). See Example \ref{Huaex} for more examples. 

The converse statement of Theorem \ref{FEBthmintro} (1) can be proved for specific fibrations. 
Indeed, we study certain spherical fibre bundles of manifolds and show that the inertness of top cell attachments for the total manifold and the base manifold are equivalent. 

\begin{theorem}[Theorem \ref{geoinertthm} and Proposition \ref{geoinertprop}]\label{geoinertthmintro}
Let $k=2$, $4$, or $8$. Let 
\[
   \diagram 
       S^{k-1}\rto^-{}\ddouble & E\rto^-{}\dto & M'\# M\dto^{p} \\ 
       S^{k-1}\rto^-{j} & N\rto^-{} & M
   \enddiagram 
\]
be a morphism of fibre bundles of connected oriented closed smooth manifolds, where $p$ is the canonical  pinch map. Suppose that the fibre inclusion $j$ is null homotopic and ${\rm dim}(M)\geq k+2$. Then the following hold:
\begin{itemize}
\item[(1).]
the attaching map for the top cell of $N$ is inert if and only if the attaching map for the top cell of $M$ is inert;
\item[(2).]
the attaching map for the top cell of $E$ is inert if and only if the attaching map for the top cell of $M'\# M$ is inert;
\item[(3).]
if the attaching map for the top cell of $M$ or $M'$ is inert, then the attaching map for the top cell of $E$ is inert.
\end{itemize}
Additionally, the assertions hold after localization at any set of primes.
\end{theorem}
Note that if $M'$ is a sphere, Theorem \ref{geoinertthmintro} (2) reduces to Theorem \ref{geoinertthmintro} (1). Additionally, under the condition that $j$ is null homotopic, the fibre sphere has to be an $H$-space and can only be $S^1$, $S^3$ or $S^7$ by Adams' solution to the famous Hopf invariant one problem. 

Theorem \ref{geoinertthmintro} (1) generalizes the geometric version of \cite[Theorem 1.2]{The24b}. In the latter work, Theriault considered a homotopy principal fibration $S^1\stackrel{j}{\longrightarrow} N\stackrel{}{\longrightarrow} M$, and showed that if $j$ is null homotopic and the attaching map for the top cell of $N$ is inert, then the attaching map for the top cell of $M$ is inert. In contrast, we consider a fibre bundle with a spherical fibre, which may not be principal, and show that the inertnesses for the attaching maps of the top cells of $M$ and $N$ are equivalent. 

In the sequel we will present two applications of Theorem \ref{geoinertthmintro}: Theorem \ref{gyration-inert-thmintro} for specific surgery and Theorem \ref{homthmintro} for complete flag manifolds. 

$\, $

%-------------------------------------------
\noindent{\bf Inertness via a surgery.} 
Surgery is a fundamental operation in geometric topology. It involves a cut-and-paste procedure to produce a new manifold from a given one. In general, the inertness property is not preserved by surgery. For example, $S^m\times S^{n-m}$ is cobordant to the sphere $S^{n}$. However, the attaching map is inert for the top cell of $S^m\times S^{n-m}$ but not inert for the top cell of $S^n$. Therefore, inertness is not a cobordism invariant property. Determining which surgeries preserve inertness is an interesting and challenging problem. 

Inspired by the work \cite{HT23} of Huang-Theriault and \cite{Dua22} of Duan, we can apply Theorem \ref{geoinertthmintro} to show that a certain type of surgery preserves the inertness of top cell attachments. 

Let $N$ be an $n$-dimensional connected oriented closed smooth manifold. Let $k\geq 2$. A $(k-1)$-surgery on the product manifold $N\times S^{k-1}$ along the sphere factor $S^{k-1}$ with a {\it framing} 
\[
\tau: S^{k-1}\stackrel{}{\longrightarrow} SO(n)
\]
produces a new orientable closed manifold, denoted by $\mathcal{G}^{\tau}(N)$. For a more precise definition of $\mathcal{G}^{\tau}(N)$, refer to Section \ref{sec: surgery}. 

This surgery is of interest in both algebraic and geometric topology. When $k=2$ and $\tau$ is trivial, Gonz\'{a}lez-Acu\~{n}a~\cite{GA75} referred  to $\mathcal{G}^{0}(N)$ as a \emph{gyration}, an object further studied in the context of toric topology by Gitler and L\'{o}pez de Medrano~\cite{GLdM13}. Additionally, when $k=2$, Duan \cite{Dua22} referred to $\mathcal{G}^{\tau}(N)$ as a \emph{suspension} and used it to study regular circle actions on manifolds, building on the work of Goldstein and Lininger \cite{GL71} on $6$-manifolds. Furthermore, Huang and Theriault \cite{HT23} proved a loop space decomposition for general $\mathcal{G}^{\tau}(N)$ and applied it to study the homotopy theory of manifolds after stabilizations.

Here, we study inertness around this surgery when $k=2$ or $4$ and prove the following theorem by combining Theorem \ref{geoinertthmintro} and results in \cite{HT23}.

\begin{theorem}\label{gyration-inert-thmintro} 
Let $N$ be an $n$-dimensional connected oriented closed smooth manifold. 
Suppose that the attaching map for the top cell of $N$ is inert. 
The following hold:
\begin{itemize}
\item[(1).]
if $k=2$, $n\equiv 0~{\rm mod}~2$ and $n\geq 4$, then the attaching map for the top cell of $\mathcal{G}^{\tau}(N)$ is inert after localization away from $2$;
\item[(2).]
if $k=2$, $n\equiv 0~{\rm mod}~4$ and $n\geq 4$, then the attaching map for the top cell of $\mathcal{G}^{0}(N)$ is inert;
\item[(3).] 
if $k=4$, $n\equiv 0~{\rm mod}~4$ and $n\geq 8$, then the attaching map for the top cell of $\mathcal{G}^{\tau}(N)$ is inert after localization away from $2$ and $3$.
\end{itemize}
\end{theorem}

$\, $
%-------------------------------------------

\noindent{\bf Inertness for Stiefel manifolds and complete flag manifolds, and other examples collected.} 
Let $G$ be a simply connected compact Lie group with a maximal torus $T\cong T^\ell=S^1\times \cdots \times S^1$, where $\ell$ is the {\it rank} of $G$. The homogeneous manifold $G/T$ is known as the {\it complete flag manifold} of $G$. 
Theorem \ref{geoinertthmintro} can be applied to show that the top cell attachments of complete flag manifolds are inert at large primes. Here, we state the results specifically for simple Lie groups, while the conclusion for general compact Lie groups follows easily and is summarized in Theorem \ref{G/Tthm}.

To collect more concrete examples of homogeneous spaces, we include the results for Stiefel manifolds in the following theorem, which serve as examples of Theorem \ref{pdtintthmintro}. 
Let $V_{n,k}(\mathbb{R})=SO(n)/SO(n-k)$, $V_{n,k}(\mathbb{C})=U(n)/U(n-k)$ and $V_{n,k}(\mathbb{H})=Sp(n)/Sp(n-k)$ be the {\it real, complex and quaternionic Stiefel manifolds}, respectively. 

\begin{theorem}[Theorems \ref{VCHthm}, \ref{VRthm} and \ref{G/Tthm}]\label{homthmintro}
Let $k\geq 2$. Let $M$ be any homogeneous space in the following list 
\[
\begin{array}{llll} 
V_{2n+2, 2k-2}(\mathbb{R})& \quad\mbox{at}\ \ p\geq 3    \\
     V_{2n+2, 2k-1}(\mathbb{R}) & \quad\mbox{at}\ \ p\geq 3   & \quad V_{n,k}(\mathbb{C})   & \quad\mbox{at}\ \ p>n-1          \\
  V_{2n+1,2k}(\mathbb{R}) & \quad\mbox{at}\ \ p>2n-1    &\quad V_{n,k}(\mathbb{H})  & \quad\mbox{at}\ \ p>2n-1 \\   
  ~~~\\
   SU(n)/T^{n-1} & \quad\mbox{at}\  \ p\geq n\geq 3        &\quad F_4/T^4& \quad\mbox{at}\ \ p\geq 12 \\  
   Sp(n)/T^{n} & \quad\mbox{at}\ \ p\geq 2n\geq 4          & \quad E_6/T^6 & \quad\mbox{at}\ \ p\geq 12 \\ 
   Spin(n)/T^{\lfloor \frac{n}{2} \rfloor} & \quad\mbox{at}\ \ p\geq n-1\geq 4           & \quad E_7/T^7 & \quad\mbox{at}\ \ p\geq 18 \\
   G_2/T^2 & \quad\mbox{at}\ \ p\geq 6                         & \quad E_8/T^8 & \quad\mbox{at}\ \ p\geq 30.
\end{array} 
\]
Then the attaching map for the top cell of $M$ is inert after localization at any allowable prime $p$ in the list. 
\end{theorem}
It is evident that one can apply Theorems \ref{FEBthmintro} and \ref{geoinertthmintro} to produce from Theorems \ref{gyration-inert-thmintro} and \ref{homthmintro} more examples of manifolds satisfying the inertness property. 

In addition to Theorems \ref{gyration-inert-thmintro} and \ref{homthmintro}, we summarize other concrete and abstract examples of Poincar\'{e} duality complexes satisfying the inertness property in this paper. It should be noted that there are other known examples not covered here; for instance, interesting examples can be found in the influential works \cite{The24a, The24b} of Theriault.

\begin{example}\label{Huaex}
Let $M$ be a $CW$-complex with a single top cell having the homotopy type of any Poincar\'{e} duality complex in the following list: 
\begin{itemize}
\item[(1).] a twisted product $S^{m}\widetilde{\times} S^{n-m}$ of connected spheres; 
\item[(2).] an $(n-1)$-connected $2n$-dimensional Poincar\'{e} duality complex $N$ with $H^n(N;\mathbb{Z})=\mathbb{Z}^{\oplus d}$ and $n$, $d\geq 2$;
\item[(3).] an $(n-1)$-connected $(2n+1)$-dimensional Poincar\'{e} duality complex with $H^n(N;\mathbb{Z})=\mathbb{Z}^{\oplus d}$, $d\geq 1$ and $n\geq 2$;
\item[(4).] an aspherical orientable closed manifold of dimension $\geq 2$;
\item[(5).] an orientable closed $3$-manifold such that its fundamental group is not a free product of finite groups;
\item[(6).] a simply connected closed $6$-manifold $N$ such that $H^3(N;\mathbb{Q})\neq 0$;
\item[(7).] an orbit manifold $N/S^1$ for a simply connected closed smooth manifold $N$ with a free $S^1$-action and an inert top cell attachment;
\item[(8).] an orbit manifold $N/S^3$ for a $3$-connected closed smooth manifold $N$ with a free $S^3$-action and an inert top cell attachment;
\item[(9).] a connected sum $A\#B$ of two connected Poincar\'{e} duality complexes with $A$ or $B$ having an inert top cell attachment; 
\item[(10).] the total space of a fibration $F\stackrel{}{\longrightarrow} N\stackrel{}{\longrightarrow} B$ of connected Poincar\'{e} duality complexes with $B$ having an inert top cell attachment; 
\item[(11).] the total space of a homotopy fibration $F\stackrel{}{\longrightarrow} N\stackrel{}{\longrightarrow} B$ of connected Poincar\'{e} duality complexes of positive dimensions with a homotopy section; 
\item[(12).] a finite connected cover of a connected $n$-dimensional Poincar\'{e} duality complex with an inert top cell attachment and $n\geq 2$. 
\end{itemize}
Then the attaching map for the top cell of $M$ is inert. 

Here, (1) is derived from Lemma \ref{Dlemma}, (2-3) are illustrated in Example \ref{BTex1} and \ref{BTex2}, (4-6) are discussed in Subsection \ref{subsec: lowfolds}, (7-8) follow directly from Theorem \ref{geoinertthmintro}, and (9-12) are derived from Theorems \ref{exsumthm} and \ref{FEBthmintro} and Propositions \ref{FEBprop} and \ref{coverprop}, respectively.  
\end{example}

$\, $
%-------------------------------------------

\noindent{\bf Organization of the paper.} 
The main body of the paper consists of four parts. The flowchart on the next page illustrates the logical flow of the sections, and the accompanying diagram indicates where the theorems in the introduction are proved.

The first part, from Sections \ref{sec: pullback} to \ref{sec: pushout}, concerns the existence of right homotopy inverses, which is closely related to inertness. 
Section \ref{sec: pullback} explores the existence of right homotopy inverses around homotopy fibrations and homotopy pullbacks. 
Section \ref{sec: BT} reviews two loop space decompositions along with Mather's Cube Lemma for later use. 
Section \ref{sec: halfsmash} investigates the existence of right homotopy inverses around half-smash products. 
With these preparations, Section \ref{sec: pushout} addresses the existence of right homotopy inverses around homotopy cofibrations and homotopy pushouts. 

The second part, from Sections \ref{sec: deg1} to \ref{sec: ex1}, studies inertness by comparing two Poincar\'{e} duality complexes of the same dimension. 
Section \ref{sec: deg1} examines the inertness property via nonzero degree maps between Poincre\'{e} duality complexes and proves Theorem \ref{inertdeg1thmintro}. 
As an application of Theorem \ref{inertdeg1thmintro}, Section \ref{sec: int} investigates inertness through intersection theory and proves Theorem \ref{pdtintthmintro}. 
Section \ref{sec: degen} proves Part (2) of Theorem \ref{FEBthmintro}, focuses on Stiefel manifolds, and shows the corresponding results in Theorem \ref{homthmintro}. 
Section \ref{sec: det} applies Theorem \ref{inertdeg1thmintro} to study inertness by comparison with a twisted product of spheres and proves Theorem \ref{detthmintro}. 
A further comparison is discussed in Section \ref{sec: ex1}, with examples including Beben-Theriault's complexes, connected sums, and low dimensional manifolds. 

The third part, from Sections \ref{sec: det2} to \ref{sec: flag}, studies inertness by comparing two Poincar\'{e} duality complexes of different dimensions. 
Section \ref{sec: det2} examines the inert property around two types of fibrations and proves Part (1) of Theorems \ref{FEBthmintro} and \ref{geoinertthmintro}. 
For one application, Section \ref{sec: surgery} investigates a specific surgery and completes the proofs of Theorems \ref{geoinertthmintro} and \ref{gyration-inert-thmintro}.
For another application, Section \ref{sec: flag} explores complete flag manifolds and completes the proof of Theorem \ref{homthmintro}.

The fourth part, from Sections \ref{sec: emb} to \ref{sec: prob}, is written to inspire further research. 
Section \ref{sec: emb} discusses inertness around manifold embeddings. 
Section \ref{sec: prob} proposes eight open problems. 

$\, $

\noindent{\bf Acknowledgements.} The author would like to thank Prof. Stephen Theriault for sharing his work \cite{The24b} and providing related insights during the course of this project. 
He is also grateful to Prof. Jianfeng Lin for helpful conversations on low dimensional topology.
 
The author was supported in part by the National Natural Science Foundation of China (Grant nos. 12331003 and 12288201), the National Key R\&D Program of China (No. 2021YFA1002300) and the Youth Innovation Promotion Association of Chinese Academy Sciences.

\begin{center}
\begin{figure}[H]
\vspace*{0.05\textheight}
\begin{tikzpicture}[>=stealth,every node/.style={shape=rectangle,draw,rounded corners},]
    % create the nodes
    \node (c3) {Section \ref{sec: BT}};
    \node (c2) [right=of c3] {Section \ref{sec: pullback}};
    \node (c4) [below=of c2] {Section \ref{sec: halfsmash}};
    \node (c5) [below=of c4] {Section \ref{sec: pushout}};
    \node (c6) [below left=of c5] {Section \ref{sec: deg1}};
    \node (c9) [below=of c6] {Section \ref{sec: det}};
    \node (c7) [left=of c9] {Section \ref{sec: int}};
    \node (c8) [below=of c7] {Section \ref{sec: degen}}; 
    \node (c10) [left=of c7] {Section \ref{sec: ex1}};
    \node (c11) [below right=of c5] {Section \ref{sec: det2}};
    \node (c12) [below=of c11] {Section \ref{sec: surgery}};
    \node (c13) [right=of c12] {Section \ref{sec: flag}};
    \node (c14) [below left=of c12] {Section \ref{sec: emb}};
    \node (c15) [below=of c14] {Section \ref{sec: prob}};
    
    % create invisible nodes for placing the squares
    \node (part1start) [draw=none,above left=0.25cm of c3] {};
    \node (part1end) [draw=none,below right=0.25cm of c5] {};
    \node (part2start) [draw=none,above right=0.25cm of c6] {};
    \node (part2end) [draw=none,below left=4.05cm of c7] {};
    \node (part3start) [draw=none,above left=0.25cm of c11] {};
    \node (part3end) [draw=none,below right=0.25cm of c13] {};
    \node (part4start) [draw=none,below right=0.7cm of c9] {};
    \node (part4end) [draw=none,below right=0.48cm of c15] {};
    
    % draw dashed squares around sections
    \draw[dashed, rounded corners] (part1start) rectangle (part1end);
    \node[draw=none, yshift=-2.5cm, xshift=1cm, text=black] at (part1start) {Part I};
    \draw[dashed, rounded corners] (part2start) rectangle (part2end);
    \node[draw=none, yshift=-4.3cm, xshift=-6cm, text=black] at (part2start) {Part II};
    \draw[dashed, rounded corners] (part3start) rectangle (part3end);
    \node[draw=none, yshift=-0.6cm, xshift=4.2cm, text=black] at (part3start) {Part III};
    \draw[dashed, rounded corners] (part4start) rectangle (part4end);
    \node[draw=none, yshift=-2.6cm, xshift=1.1cm, text=black] at (part4start) {Part IV};
    
    % connect the nodes
    \draw[->] (c2) -- (c4);
    \draw[->] (c2) to[out=0,in=90] (c11);
    \draw[->] (c3) to[out=270,in=180]  (c4);
    \draw[->] (c4) -- (c5);
    \draw[->] (c5) to[out=180,in=90] (c10);
    \draw[->] (c5) to[out=180,in=90] (c6);
    \draw[->] (c5) -- (c14);
    \draw[->] (c6) to[out=180,in=90] (c7);
    \draw[->] (c6) -- (c9);
    \draw[->] (c7) -- (c8);
    \draw[->] (c11) -- (c12);
    \draw[->] (c11) to[out=0,in=90] (c13);
    \draw[->] (c12) to[out=270,in=0](c15);
    \draw[->] (c13) to[out=270,in=0] (c15);
    \draw[->] (c14) -- (c15);
\end{tikzpicture}
\vspace*{\fill}
\caption*{The logical flow of the main body of the paper}
\end{figure}
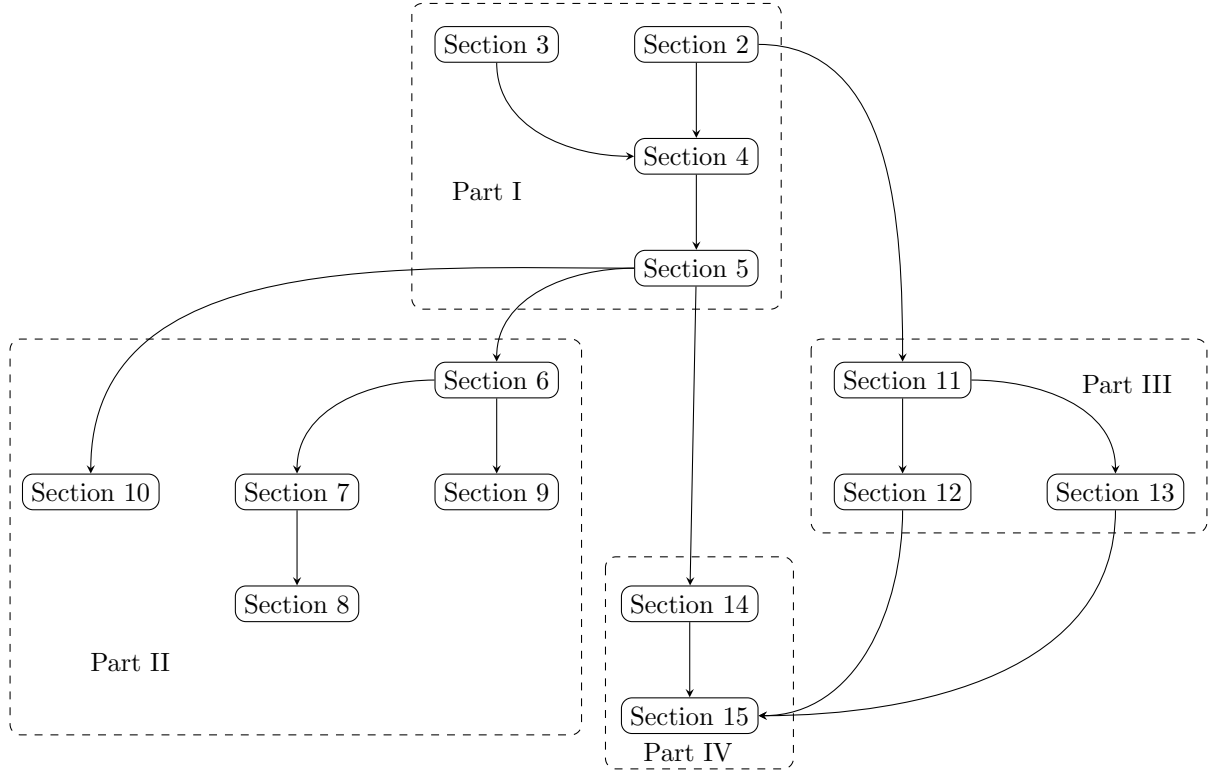
\end{center}

\begin{center}
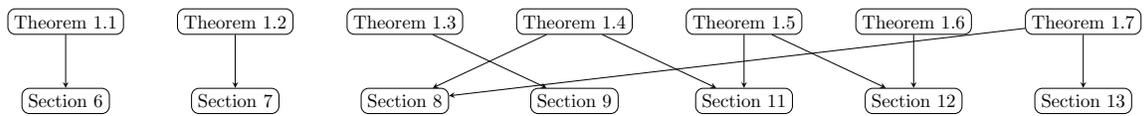
\begin{figure}[H]
\vspace*{0.05\textheight}
\resizebox{\textwidth}{!}{
\begin{tikzpicture}[>=stealth,every node/.style={shape=rectangle,draw,rounded corners},]
    % create the nodes
    \node (t1) {Theorem \ref{inertdeg1thmintro}};
    \node (t2) [right=of t1] {Theorem \ref{pdtintthmintro}};
    \node (t3) [right=of t2] {Theorem \ref{detthmintro}};
    \node (t4) [right=of t3] {Theorem \ref{FEBthmintro}};
    \node (t5) [right=of t4] {Theorem \ref{geoinertthmintro}};
    \node (t6) [right=of t5] {Theorem \ref{gyration-inert-thmintro}};
    \node (t7) [right=of t6] {Theorem \ref{homthmintro}};
    \node (S6) [below=of t1]{Section \ref{sec: deg1}};
    \node (S7) [below=of t2] {Section \ref{sec: int}};
    \node (S8) [below=of t3] {Section \ref{sec: degen}}; 
    \node (S9) [below=of t4] {Section \ref{sec: det}};
    \node (S11) [below=of t5] {Section \ref{sec: det2}};
    \node (S12) [below=of t6] {Section \ref{sec: surgery}};
    \node (S13) [below=of t7] {Section \ref{sec: flag}};
        % connect the nodes
    \draw[->] (t1) -- (S6);
    \draw[->] (t2) -- (S7);
    \draw[->] (t3) -- (S9);
    \draw[->] (t4) -- (S8);
    \draw[->] (t4) -- (S11);
    \draw[->] (t5) -- (S11);
    \draw[->] (t5) -- (S12);
    \draw[->] (t6) -- (S12);
    \draw[->] (t7) -- (S8);
    \draw[->] (t7) -- (S13);
\end{tikzpicture}
}
\vspace*{\fill}
\caption*{The diagram indicating the sections where the theorems are proved}
\end{figure}
\end{center}

\newpage
%%%%%%%%%%%%%%%%%%%%%%%%%%%%%%%%%%%%%%%%%%%%%%%%%%%%%%%%%%%%%%%%%%%%%%

%\part{Right homotopy inverses around homotopy constructions}
%\label{part1}

%----------------------------------------------------------------------------------------------------------------------------------------------------------------------------------------------------------%

\section{Homotopy fibrations and homotopy pullbacks}
\label{sec: pullback}
Since the concept of inertness is defined through the existence of right homotopy inverses, the first several sections are devoted to studying the existence of right homotopy inverses in various contexts. 

In this section, we explore the existence of right homotopy inverses around homotopy fibrations and homotopy pullbacks. As the loop functor behaves well with respect to homotopy fibrations, the arguments in this section are straightforward, and the results are standard. We include them here for the sake of clarity and completeness in the general study on the existence of right homotopy inverses. 

Since we will encounter strict fibrations and fibre bundles in the sequel, it is worth noting that a homotopy fibration refers to a map $p: X\stackrel{}{\longrightarrow} B$ of $CW$-complexes with its homotopy fibre $F$. This is typically organized into a sequence
\[
F\stackrel{}{\longrightarrow} X\stackrel{p}{\longrightarrow} B
\]
similar to strict fibrations. The terminology is reasonable and conventional, based on the classical fact that every homotopy fibration is homotopy equivalent to a strict fibration.  

In the following, Propositions \ref{RHIfibprop} and \ref{LHIfibprop} study the existence of right homotopy inverses without looping under compatible conditions between splittable homotopy fibrations. For their proofs, a basic fact will be used.

\begin{lemma}\label{eqnaturallemma}
Suppose that there is a homotopy commutative diagram
\[
\diagram
X' \dto^{\lambda_X}  \rto^{\varphi'}_{\simeq}   & Y'  \dto^{\lambda_Y} \\
X                              \rto^{\varphi}_{\simeq}    & Y,
\enddiagram
\]
where $\varphi'$ and $\varphi$ are homotopy equivalences. Then for any homotopy inverses $\varphi'^{-1}$ and $\varphi^{-1}$ of $\varphi'$ and $\varphi$, respectively, the diagram
\[
\diagram
Y' \dto^{\lambda_Y}  \rto^{\varphi'^{-1}}_{\simeq}   & X'  \dto^{\lambda_X} \\
Y                              \rto^{\varphi^{-1}}_{\simeq}    & X,
\enddiagram
\] 
homotopy commutes.
\end{lemma}
\begin{proof}
The lemma follows from the following series of homotopies
\[
\varphi^{-1}\circ \lambda_Y\simeq  \varphi^{-1}\circ \lambda_Y \circ \varphi' \circ \varphi'^{-1}\simeq \varphi^{-1}\circ \varphi \circ  \lambda_X  \circ \varphi'^{-1} \simeq  \lambda_X  \circ \varphi'^{-1}. 
\]
\end{proof}

\begin{proposition}\label{RHIfibprop} 
Suppose that there is a homotopy fibration diagram 
\begin{equation}\label{F'Ffibdiag}
\diagram 
      F'\rto^-{j'}\dto^{\lambda_F} & X'\rto^-{p'}\dto^{\lambda_X} & B'\dto^{\lambda_B}  \\ 
      F\rto^-{j} & X\rto^-{p} & B 
  \enddiagram
  \end{equation}
and both $p'$ and $p$ have right homotopy inverses $s'$ and $s$, 
respectively, such that there is a homotopy commutative diagram 
\begin{equation}\label{s'sdiag0}
\diagram 
      B'\rto^-{s'}\dto^{\lambda_B} &X'\dto^{\lambda_X} \\ 
      B\rto^-{s} &X.
  \enddiagram
  \end{equation} 
If $X'$ and $X$ are $H$-spaces and $\lambda_X$ is an $H$-map, then the following hold:
\begin{itemize}
\item[(1).] there are compatible splittings
\[
\xymatrix{
B'\times F'\ar[r]^<<<{\chi'}_{\simeq} \ar[d]^{\lambda_B\times \lambda_F} &  X' \ar[d]^{\lambda_X} \\
B\times F\ar[r]^{\chi}_{\simeq}   & X;
}
\]
\item[(2).] if both $\lambda_B$ and $\lambda_F$ have right homotopy inverses, then so does $\lambda_X$;
\item[(3).] if $\lambda_X$ has a right homotopy inverse, then so do $\lambda_B$ and $\lambda_F$.
\end{itemize}
\end{proposition}
\begin{proof}
(1). Consider the diagram 
\[
\xymatrix{
B'\times F'  \ar[d]^{\lambda_B\times \lambda_F}  \ar[r]^{s'\times j'} &  X'\times X' \ar[r]^<<<{\mu} \ar[d]^{\lambda_X\times \lambda_X} &  X'\ar[d]^{\lambda_X}\\
B\times F \ar[r]^{s\times j} &  X\times X \ar[r]^{\mu} &  X,
}
\]
where the maps $\mu$ are the $H$-multiplications of $X'$ and $X$, respectively. The $H$-map $\lambda_X$ implies that the right square of the diagram homotopy commutes. As the left square of the diagram is the product of (\ref{s'sdiag0}) and the left square of (\ref{F'Ffibdiag}), it also homotopy commutes. It follows that the outer rectangle of the above diagram homotopy commutes. Let $\chi'=\mu\circ (s'\times j')$ and $\chi=\mu\circ (s\times j)$. Then both 
$\chi'$ and $\chi$ are weak homotopy equivalences, and hence are homotopy equivalences by Whitehead's theorem. 

(2). Let $t_B: B\stackrel{}{\longrightarrow} B'$ and $t_F: F\stackrel{}{\longrightarrow} F'$ be right homotopy inverses of $\lambda_B$ and $\lambda_F$, respectively. Consider the homotopy commutative diagram
\[
\xymatrix{
X\ar[r]^{\chi^{-1}} & B\times F \ar[r]^{t_B\times t_F} \ar@{=}[dr]  & B'\times F' \ar[r]^{\chi'} \ar[d]^{\lambda_B\times \lambda_F} &  X' \ar[d]^{\lambda_X}\\
&& B\times F \ar[r]^{\chi}  & X, 
}
\]
where the right square is from (1) and $\chi^{-1}$ is a homotopy inverse of $\chi$. Let $t_X=\chi'\circ (t_B\times t_F)\circ \chi^{-1}$. The diagram implies that $\lambda_X\circ t_X\simeq \chi\circ (\lambda_B\times \lambda_F)\circ (t_B\times t_F)\circ \chi^{-1} \simeq \chi\circ\chi^{-1}\simeq 1_X$, that is, $t_X$ is a right homotopy inverse of $\lambda_X$.

(3). Let $t_X: X\stackrel{}{\longrightarrow} X'$ be a right homotopy inverse of $\lambda_X$. There is the homotopy commutative diagram
\[
\xymatrix{
B\ar[r]^{s} & X \ar[r]^{t_X} \ar@{=}[dr] & X'\ar[r]^{p'} \ar[d]^{\lambda_X} & B\ar[d]^{\lambda_B} \\
&&X \ar[r]^{p} & B,
}
\]
where the right square is the right square of (\ref{F'Ffibdiag}). Let $t_B=p'\circ t_X\circ s$. The diagram implies that $\lambda_B\circ t_B\simeq p\circ \lambda_X\circ t_X\circ s\simeq p\circ s\simeq 1_{B}$, that is, $t_B$ is a right homotopy inverse of $\lambda_B$.

To prove the conclusion for $\lambda_F$, consider the diagram
\[
\xymatrix{
F\ar[r]^{j} & X \ar@{=}[dr] \ar[r]^{t_X} & X' \ar[d]^{\lambda_X} \ar[r]^{\chi'^{-1}} & B'\times F' \ar[r]^<<<{\pi_2}\ar[d]^{\lambda_B\times \lambda_F} & F'\ar[d]^{\lambda_F}\\
&& X\ar[r]^{\chi^{-1}} & B\times F \ar[r]^<<<<{\pi_2}& F,
}
\]
where $\pi_2$ are the canonical projections to the second factors and $\chi'^{-1}$ and $\chi^{-1}$ are homotopy inverses of $\chi'$ and $\chi$, respectively. 
By Lemma \ref{eqnaturallemma}, the homotopy commutative square in (1) implies that the left square of the diagram homotopy commutes. Also, the naturality of $\pi_2$ implies that the right square commutes. Therefore, the above diagram homotopy commutes. Since the composition in the lower direction around the diagram is a homotopy equivalence, so is the composition in the upper direction. This implies that $\lambda_F$ has a right homotopy inverse.
\end{proof}

\begin{proposition}\label{LHIfibprop} 
Suppose that there is a homotopy fibration diagram of the form (\ref{F'Ffibdiag}) 
and both $j'$ and $j$ have left homotopy inverses $r'$ and $r$, 
respectively, such that there is a homotopy commutative diagram 
\begin{equation}\label{r'rdiag0}
\diagram 
      X'\rto^-{r'}\dto^{\lambda_X} &F'\dto^{\lambda_F} \\ 
      X\rto^-{r} &F.
  \enddiagram
  \end{equation} 
Then the following hold:
\begin{itemize}
\item[(1).] there are compatible splittings
\[
\xymatrix{
X' \ar[d]^{\lambda_X} \ar[r]^{\vartheta' \ \ \ }_{\simeq \ \ \ } & B'\times F' \ar[d]^{\lambda_B\times \lambda_F}   \\
X\ar[r]^{\vartheta\ \ \ }_{\simeq \ \ \ } & B\times F;
}
\]
\item[(2).] if both $\lambda_B$ and $\lambda_F$ have right homotopy inverses, then so does $\lambda_X$;
\item[(3).] if $\lambda_X$ has a right homotopy inverse, then so do $\lambda_B$ and $\lambda_F$.
\end{itemize}
\end{proposition}
\begin{proof}
(1). Consider the diagram 
\[
\xymatrix{
X' \ar[r]^<<<<{\Delta} \ar[d]^{\lambda_X} & X'\times X' \ar[r]^{p'\times r'} \ar[d]^{\lambda_X\times \lambda_X} & B'\times F' \ar[d]^{\lambda_B\times \lambda_F} \\
X \ar[r]^<<<<{\Delta} & X\times X \ar[r]^{p\times r}  & B\times F, 
}
\]
where $\Delta$ are the diagonal maps. The naturality of $\Delta$ implies that the left square commutes, while the right square is the product of the right square of (\ref{F'Ffibdiag}) and (\ref{r'rdiag0}), and hence homotopy commutes. It follows that the outer rectangle of the above diagram homotopy commutes. Let $\vartheta'=(p'\times r')\circ \Delta$ and $\vartheta=(p\times r)\circ \Delta$. Then both $\vartheta'$ and $\vartheta$ are weak homotopy equivalences, and hence are homotopy equivalences by Whitehead's theorem. 

(2) \& (3). Let $\chi'$ and $\chi$ be homotopy inverses of $\vartheta'$ and $\vartheta$, respectively. Consider the diagram
\[
\xymatrix{
B'\ar[r]^<<<<{i_1}   \ar[d]^{\lambda_B}  & B'\times F'\ar[r]^<<<{\chi'}_<<<{\simeq} \ar[d]^{\lambda_B\times \lambda_F} &  X' \ar[d]^{\lambda_X} \\
B  \ar[r]^<<<<{i_1}   & B\times F\ar[r]^{\chi}_{\simeq}   & X,
}
\]
where $i_1$ are the canonical inclusions into the first factors. By Lemma \ref{eqnaturallemma}, the homotopy commutative square in statement (1) implies that the right square of the diagram homotopy commutes. Also, the naturality of $i_1$ implies that the left square commutes. Therefore, the above diagram homotopy commutes. Furthermore, the sequence of homotopies 
\[
p'\circ (\chi'\circ i_1)= (\pi_1\circ (p'\times r') \circ \Delta) \circ (\chi'\circ i_1)\simeq \pi_1\circ \vartheta'\circ \chi'\circ i_1\simeq \pi_1\circ i_1\simeq 1_{B'}
\]
implies that $\chi'\circ i_1$ is a right homotopy inverse of $p'$. Similarly, $\chi\circ i_1$ is a right homotopy inverse of $p$. Accordingly, the conditions of Proposition \ref{RHIfibprop} are satisfied, and the statements (2) and (3) follows. 
\end{proof}

\begin{remark}\label{HIfibrmk}
Lemma \ref{eqnaturallemma} implies that the statements (1) of Proposition \ref{RHIfibprop} and \ref{LHIfibprop} are equivalent. Furthermore, if either of these two holds both Diagrams (\ref{s'sdiag0}) and (\ref{r'rdiag0}) can be constructed. 
\end{remark}

As the loop functor behaves well with respect to homotopy fibrations, Proposition \ref{RHIfibprop} and \ref{LHIfibprop} can be used to study the existence of right homotopy inverses after looping around homotopy fibrations. 

\begin{corollary}\label{RHIfibcor} 
Suppose that there is a homotopy fibration diagram 
\begin{equation}\label{F'FfibdiagE}
\diagram 
      F'\rto^-{j'}\dto^{\lambda_F} & E'\rto^-{p'}\dto^{\lambda_E} & B'\dto^{\lambda_B}  \\ 
      F\rto^-{j} & E\rto^-{p} & B 
  \enddiagram
  \end{equation}
and both $\Omega p'$ and $\Omega p$ have right homotopy inverses $s'$ and $s$, 
respectively, such that there is a homotopy commutative diagram 
\begin{equation}\label{loops'sdiag0}
\diagram 
     \Omega  B'\rto^-{s'}\dto^{\Omega \lambda_B} &\Omega E'\dto^{\Omega \lambda_E} \\ 
     \Omega  B\rto^-{s} &\Omega E.
  \enddiagram
  \end{equation} 
Then the following hold:
\begin{itemize}
\item[(1).] there are compatible splittings
\[
\xymatrix{
\Omega B'\times \Omega F'\ar[r]^{\ \ \ \chi' }_{\ \ \ \simeq} \ar[d]^{\Omega \lambda_B\times \Omega \lambda_F} &  \Omega E' \ar[d]^{\Omega \lambda_E} \\
\Omega B\times \Omega F\ar[r]^{\ \ \chi }_{\ \ \simeq}   & \Omega E;
}
\]
\item[(2).] if both $\Omega \lambda_B$ and $\Omega \lambda_F$ have right homotopy inverses, then so does $\Omega \lambda_E$;
\item[(3).] if $\Omega \lambda_E$ has a right homotopy inverse, then so do $\Omega \lambda_B$ and $\Omega \lambda_F$.
\end{itemize}
\end{corollary}
\begin{proof}
Apply Proposition \ref{RHIfibprop} to the loop of Diagram (\ref{F'FfibdiagE}), and note that $\Omega \lambda_E$ is an $H$-map.
\end{proof}

\begin{remark}\label{LHIfibcorrmk}
As in Remark \ref{HIfibrmk}, 
for the homotopy fibration diagram (\ref{F'FfibdiagE}), the condition (\ref{loops'sdiag0}) is equivalent to that both $\Omega j'$ and $\Omega j$ have left homotopy inverses $r'$ and $r$, 
respectively, such that there is a homotopy commutative diagram 
\[\label{loopr'rdiag0}
\diagram 
     \Omega E'\rto^-{r'}\dto^{\Omega \lambda_E} &\Omega F'\dto^{\Omega \lambda_F} \\ 
      \Omega E\rto^-{r} &\Omega F.
  \enddiagram
  \]
Hence, in this case the three statements of Corollary \ref{RHIfibcor} hold immediately.  
\end{remark}

We can apply Proposition \ref{RHIfibprop} and \ref{LHIfibprop} and Corollary \ref{RHIfibcor} for homotopy fibrations to the following theorem for homotopy pullback.

\begin{theorem}\label{pullbackthm}
Let
\begin{equation}\label{pullbackdiag}
\diagram 
      A'\rto^-{\psi'}\dto^-{a} & B'\dto^-{b} \\ 
      A\rto^-{\psi} & B
  \enddiagram
  \end{equation}
be a homotopy pullback. Then the following hold:
\begin{itemize}
\item[(1).] if $b$ has a right homotopy inverse then so does $a$;
\item[(2).] suppose that $\psi$ has a right homotopy inverse, and $A$ is an $H$-space. Then $b$ has a right homotopy inverse if and only if $a$ has a right homotopy inverse; 
\item[(3).] suppose that $\Omega \psi'$ has a left homotopy inverse. Then $\Omega b$ has a right homotopy inverse if and only if $\Omega a$ has a right homotopy inverse;
\item[(4).] suppose that $\Omega \psi$ has a right homotopy inverse. Then $\Omega b$ has a right homotopy inverse if and only if $\Omega a$ has a right homotopy inverse.
\end{itemize}
\end{theorem}
\begin{proof}
(1). Let $c$ be a right homotopy inverse of $b$. Consider the diagram 
\[\xymatrix{ 
A         \ar@{.>}[dr]^(0.6){c'}       \ar@/^0.7pc/[drr]^{c\circ \psi}    \ar@/_0.7pc/[ddr]_{1_A}  &&\\
&A'\rto^-{\psi'}\dto^-{a} & B'\dto^-{b} \\ 
&      A\rto^-{\psi} & B, 
}\]
where the inner square is a homotopy pullback by assumption, and the outer square homotopy commutes as $\psi\simeq b\circ c\circ \psi$.  Hence the universal property of homotopy pullback implies that there is a map $c'$ that makes the two triangular regions homotopy commute. In particular $a$ has a right homotopy inverse $c'$.

(2). Let $F$ be the homotopy fibre of $\psi'$. By assumption it is also the homotopy fibre of $\psi$, and there is a diagram of homotopy fibrations
\[
\diagram 
      F \ddouble  \rto^-{ j'} &  A'\rto^-{ \psi'}\dto^-{ a} &  B'\dto^-{ b} \\ 
        F  \rto^-{ j} &  A\rto^-{ \psi} &  B 
  \enddiagram
\]
which defines the maps $j'$ and $j$. Let $s$ be a right homotopy inverse of $\psi$ and $\mu$ be the $H$-multiplication of $A$. The composite 
\[
\chi: B\times F\stackrel{s\times j}{\longrightarrow} A\times A\stackrel{\mu}{\longrightarrow} A
\]
is a weak homotopy equivalence, and hence is a homotopy equivalence by the Whitehead theorem. In particular, the composite
\[
r': A\stackrel{\chi^{-1}}{\longrightarrow} B\times F\stackrel{\pi_2}{\longrightarrow} F
\]
satisfies that $r'\circ j$ is a homotopy equivalence, where $\chi^{-1}$ is a homotopy inverse of $\chi$ and $\pi_2$ is the projection to the second summand. It follows that $j$ has a left homotopy inverse, say $r$. Then $r\circ  a\circ  j'\simeq r\circ  j\simeq 1_{ F}$, that is, $r\circ  a$ is a left homotopy inverse of $ j'$. Hence, we can apply Proposition \ref{LHIfibprop} to the above homotopy fibration diagram and show that $b$ has a right homotopy inverse if and only if $a$ has a right homotopy inverse.

(3). Extending the homotopy fibration diagram in (2) two steps to the left, we have the homotopy fibration diagram
\[
\diagram 
      \Omega A'\rto^-{\Omega \psi'}\dto^-{\Omega a} & \Omega B'\dto^-{\Omega b} \rto^-{\delta'}&  F\ddouble \\ 
      \Omega A\rto^-{\Omega \psi} & \Omega B \rto^-{\delta} & F
  \enddiagram
\]
with the connecting maps $\delta'$ and $\delta$. 
Since $\Omega \psi'$ has a left homotopy inverse, it follows that the homotopy fibration in the first row splits and in particular $\delta'$ has a right homotopy inverse, say $t'$. Then $\delta\circ \Omega b\circ t'\simeq \delta'\circ t'\simeq 1_{F}$, that is, $\Omega b\circ t'$ is a right homotopy inverse of $\delta$. Hence, we can apply Proposition \ref{RHIfibprop} to the above homotopy fibration diagram, and show that $\Omega b$ has a right homotopy inverse if and only if $\Omega a$ has a right homotopy inverse.

(4). Consider the loop of Diagram (\ref{pullbackdiag}), which is also a homotopy pullback. Since $\Omega A$ is an $H$-space and $\Omega \psi$ has a right homotopy inverse, statement (4) follows from (2) immediately. 
\end{proof}

\begin{remark}\label{couexpullbackthm}
The converse of Theorem \ref{pullbackthm} (1) does not hold. For instance, a homotopy fibration $F\stackrel{j}{\longrightarrow} E\stackrel{p}{\longrightarrow} B$ is equivalent to a homotopy pullback 
\[
\diagram 
      F\rto^-{j}\dto^-{} & E\dto^-{p} \\ 
      \ast\rto^-{} & B.
  \enddiagram
  \]
Then the map $F\stackrel{}{\longrightarrow} \ast$ has a right homotopy inverse, while $E\stackrel{p}{\longrightarrow} B$ does not have a right homotopy inverse in general. 
\end{remark}

\newpage
%----------------------------------------------------------------------------------------------------------------------------------------------------------------------------------------------------------%

\section{Two decompositions and Mather's Cube Lemma}
\label{sec: BT}
This section presents preliminary results that will be used to study the existence of right homotopy inverses around homotopy cofibrations and homotopy pushouts.

The first result is the classical Ganea theorem \cite{Gan65}.
As notation, let $I=[0,1]$ be the unit interval with basepoint at $0$. For 
pointed spaces $X$ and $Y$ the (reduced) \emph{join} $X\ast Y$ is defined as the quotient space 
\[X\ast Y=(X\times I\times Y)/\sim\] 
where $(x,0,y)\sim (x',0,y)$, $(x,1,y)\sim (x,1,y')$ and $(\ast,t,\ast)\sim (\ast,0,\ast)$ for all 
$x$, $x'\in X$, $y$, $y'\in Y$ and $t\in I$. Equivalently, if $CX$ and $CY$ are the reduced cones on $X$ 
and $Y$, respectively, then there is a pushout 
\[\diagram 
       X\times Y\rto\dto & X\times CY\dto \\ 
       CX\times Y\rto & X\ast Y. 
  \enddiagram\] 
As $CX$ and $CY$ are contractible, this implies that there is a homotopy pushout 
\[ 
  \label{joinpo} 
  \diagram 
        X\times Y\rto^-{\pi_{1}}\dto^{\pi_{2}} & X\dto \\ 
        Y\rto & X\ast Y 
  \enddiagram 
\]
where $\pi_{1}$ and $\pi_{2}$ are the projections. It is well-known that there is a homotopy equivalence 
$X\ast Y\simeq\Sigma X\wedge Y$. 

\begin{theorem}\label{Ganeathm}
 Let 
   \(\nameddright{F}{i}{E}{p}{B}\) 
   be a homotopy fibration and let $Y$ be the homotopy cofibre of $i$. Then 
   there is a homotopy pullback 
   \[\diagram 
         F\rto^-{i}\dto^{\xi} & E\rto^-{p}\dto & B\ddouble \\ 
         F\ast \Omega B\rto & Y\rto^{p'} & B 
     \enddiagram\] 
   where $\xi$ is null homotopic. Furthermore, the homotopy fibration 
   \(\nameddright{F\ast \Omega B}{}{Y}{p'}{B}\) 
   splits after looping to give a homotopy equivalence    
   \[\hspace{5.6cm}\Omega Y\simeq\Omega B\times \Omega (F\ast \Omega B).\hspace{5.6cm}\Box\] 
\end{theorem}

The second result is another elegant decomposition of Ganea type, established by Beben and Theriault \cite{BT22}. 
Let $X\ltimes Y$ be the {\it left half-smash} of $X$ and~$Y$, defined as the quotient space $(X\times Y)/(X\times \ast)$.

Dual to a homotopy fibration, a homotopy cofibration refers to a map $f: A\stackrel{}{\longrightarrow} Y$ of $CW$-complexes with its homotopy cofibre $Z$. This is typically organized into a sequence
\[
A\stackrel{f}{\longrightarrow} Y\stackrel{}{\longrightarrow} Z
\]
similar to strict cofibrations. The terminology is reasonable and conventional, based on the classical fact that every homotopy cofibration is homotopy equivalent to a strict cofibration.

\begin{theorem} 
   \label{GTcofib} 
   Let 
   \(\nameddright{ A}{f}{Y}{h}{Z}\) 
   be a homotopy cofibration. Suppose that $f$ is inert, that is, the map $\Omega h$ has a right 
   homotopy inverse. 
   %\(s\colon\namedright{\Omega Z}{}{\Omega Y}\). 
   Then there is a homotopy fibration 
   \[\llnameddright{\Omega Z\ltimes A}{\widetilde{\Gamma}} 
           {Y}{h}{Z}\] 
           for some map $\widetilde{\Gamma}$ whose restriction to $A$ is $f$. Moreover, the homotopy fibration splits after looping to give a homotopy equivalence 
   \[\hspace{5.4cm}\Omega Z\times\Omega(\Omega Z\ltimes A)\stackrel{\simeq}{\longrightarrow}\Omega Y. 
         \hspace{5.4cm}\Box\] 
\end{theorem} 

We may also need a naturality property of Theorem~\ref{GTcofib}.
\begin{proposition}\label{GTcofibnat} 
If there is a homotopy cofibration diagram 
\[\diagram 
      A\rto^-{f}\dto^{\lambda_A} & Y\rto^-{h}\dto^{\lambda_Y} & Z\dto^{\lambda_Z}  \\ 
      A'\rto^-{f'} & Y'\rto^-{h'} & Z' 
  \enddiagram\] 
and both $\Omega h$ and $\Omega h'$ have right homotopy inverses $s$ and $s'$, 
respectively, such that there is a homotopy commutative diagram 
\begin{equation}\label{ss'diag}
\diagram 
      \Omega Z\rto^-{s}\dto^{\Omega\lambda_Z} & \Omega Y\dto^{\Omega \lambda_Y} \\ 
      \Omega Z'\rto^-{s'} & \Omega Y', 
  \enddiagram
  \end{equation}
then there is a homotopy fibration diagram
\[\diagram 
      \Omega Z\ltimes A \rto^-{\widetilde{\Gamma}} \dto^{\Omega \lambda_Z\ltimes \lambda_A} & Y\rto^-{h}\dto^{\lambda_Y} & Z\dto^{\lambda_Z}  \\ 
      \Omega Z'\ltimes A' \rto^-{\widetilde{\Gamma}'}  & Y'\rto^-{h'} & Z', 
  \enddiagram\] 
in which the rows splits after looping to give compatible homotopy equivalences
\[
\xymatrix{
\Omega Z\times \Omega (\Omega Z\ltimes A)\ar[r]^<<<<{\simeq} \ar[d]^{\Omega \lambda_Z\times \Omega (\Omega \lambda_Z\ltimes \lambda_A)} &  \Omega Y \ar[d]^{\Omega \lambda_E} \\
\Omega Z'\times \Omega (\Omega Z'\ltimes A')\ar[r]^<<<{\simeq} &  \Omega Y'. 
}
\]
\end{proposition} 
\begin{proof}
By \cite[Remark 2.2]{HT22} or \cite[Remark 2.7]{The24a}, the homotopy fibration in Theorem~\ref{GTcofib} is natural. Then by Corollary \ref{RHIfibcor} (1) the splitting in Theorem~\ref{GTcofib} is also natural.  
\end{proof}

The following classical result can be proved from Beben-Theriault's decomposition. 
\begin{lemma}\label{wedgeqlemma}
Let $Z$ and $A$ be path connected spaces. Then there is a natural homotopy fibration
\[
\Omega Z\ltimes A \stackrel{}{\longrightarrow} Z\vee A\stackrel{q_1}{\longrightarrow} Z,
\]
where $q_1$ is the projection onto the first wedge summand. Moreover, the homotopy fibration splits after looping to give a natural homotopy equivalence
\[
\Omega (Z\vee A) \simeq \Omega Z\times \Omega (\Omega Z\ltimes A).
\]
\end{lemma}
\begin{proof}
There is the natural homotopy cofibration
\[
A\stackrel{i_2}{\longrightarrow}Z\vee A\stackrel{q_1}{\longrightarrow} Z,
\]
such that $q_1\circ i_1=1_{Z}$, where $i_k$ is the natural inclusion of the $k$-the wedge summand with $k=1$, $2$. Therefore, the lemma follows from Theorem \ref{GTcofib} and Proposition \ref{GTcofibnat} immediately. 
\end{proof}

The two decompositions, Theorems \ref{Ganeathm} and \ref{GTcofib}, were proved by different methods originally. However, they both can be proved by Mather's Cube Lemma \cite{Mat76} as illustrated in \cite[Chapter 5]{HT24c}. The latter is also a crucial tool in the sequel.

\begin{theorem}
   \label{cubethm} 
  Suppose that there is a homotopy pushout 
\[\diagram 
     A\rto\dto &  B\dto \\ 
     C\rto & D 
  \enddiagram\] 
and a homotopy fibration 
\(\nameddright{D'}{}{D}{h}{Z}\). 
Composing all the maps in the homotopy pushout with~$h$ and taking homotopy fibres over 
the common base $Z$ gives a homotopy commutative cube
\begin{gather*}
\begin{aligned}
\xymatrixcolsep{1.5pc}
\xymatrixrowsep{1.5pc}
\xymatrix{
 & A^\prime \ar[dl]  \ar[rr]  \ar[dd]|!{[d];[d]}\hole && 
 B^\prime \ar[dl]  \ar[dd]  \\
 C^\prime \ar[dd]  \ar[rr]  &&
 D^\prime \ar[dd] \\
  & A \ar[dl]  \ar[rr]|!{[r];[r]}\hole  && 
 B\ar[dl]    \\
 C \ar[rr]  &&
 D 
 }
\end{aligned}
\end{gather*}
that defines $A'$, $B'$ and $C'$. Then the four sides are all homotopy pullbacks and the top face of the cube is a homotopy pushout.~$\qqed$ 
\end{theorem} 

\begin{remark}\label{cubethmrmk}
As remarked in \cite[Appendix]{HT24a} and \cite[Section 3]{The24b}, the original statement of Mather's Cube Lemma is more general. It assumes that there is a homotopy commutative cube, in which the bottom face is a homotopy pushout, the four sides are homotopy pullbacks, and there are compatibilities among the homotopies, and concludes that the top face is a homotopy pushout. In our case, the stronger hypothesis that the homotopy commutativity of the cube is due to it being induced by taking fibres over a map 
\(\namedright{D}{h}{Z}\)  
lets one apply \cite[Lemma 3.1]{PT19} to bypass the compatibilities, and then establish Theorem \ref{cubethm}. For a rigorous proof of Theorem \ref{cubethm}, one may refer to \cite[Section 5.1]{HT24c}. 
\end{remark}

\newpage
%----------------------------------------------------------------------------------------------------------------------------------------------------------------------------------------------------------%
\section{Half-smash products}
\label{sec: halfsmash}
In this section, we investigate the existence of right homotopy inverses around half-smash products. 
The material will be used to study the existence of right homotopy inverses around homotopy cofibrations and homotopy pushouts in Section \ref{sec: pushout}. 

Recall the {\it left half-smash} of two pointed spaces $X$ and~$A$ is defined as the quotient space $(X\times A)/(X\times \ast)$. It determines a natural cofibration
\[
X\stackrel{i_1}{\longrightarrow} X\times A \stackrel{q_1}{\longrightarrow} X\ltimes A,
\] 
where $i_1$ and $q_1$ are the inclusion and projection maps, respectively. Similarly we can define the right half-smash $X\rtimes A$ . However, as the constructions are symmetric we only need to study one of them. 
\begin{lemma}\label{ltimeslemma0}
There is a natural inclusion $j_2: A \stackrel{}{\longrightarrow} X\ltimes A$ and a natural projection $p_2: X\ltimes A \stackrel{}{\longrightarrow} A$ such that $p_2\circ j_2=1_A$.
\end{lemma}
\begin{proof}
Consider the diagram
\[
\xymatrix{
&                        A \ar@{=}[r] \ar[d]^{i_2}  & A \ar[d]^{j_2}\\
X\ar[r]^<<<<{i_1}   & X\times   A \ar[r]^{q_1} \ar[d]^{\pi_2} & X\ltimes A \ar[d]^{p_2}\\
                  &             A \ar@{=}[r] & A,
}
\]
where $i_2$ and $\pi_2$ are the inclusion and projection maps, respectively, $j_2=q_1\circ i_2$, and $\pi_2$ extends across $q_1$ to a map $p_2$ as $\pi_2\circ i_1$ is trivial. From the diagram it is clear that $p_2\circ j_2=1_A$, and all the involved maps are natural.
\end{proof}

\begin{lemma}\label{ltimeslemma}
A map $c: A\stackrel{}{\longrightarrow} B$ has a right homotopy inverse if and only if $1_X \ltimes c: X\ltimes A \longrightarrow X\ltimes B$ has a right homotopy inverse.
\end{lemma}
\begin{proof}
If $c$ has a right homotopy inverse $r$, it is clear that $1_X \ltimes r$ is a right homotopy inverse of $1_X \ltimes c$. Conversely, suppose that $1_X \ltimes c$ has a right homotopy inverse $t$. Then by Lemma \ref{ltimeslemma0} there is a homotopy commutative diagram
\[
\xymatrix{
B \ar[r]^<<<<{j_2} & X\ltimes B \ar[r]^{t} \ar@{=}[dr]  &X\ltimes A \ar[r]^{p_2} \ar[d]^{1_X \ltimes c}  & A \ar[d]^{c}   \\
&& X\ltimes B \ar[r]^{p_2} & B.
}
\]
Define $t'=p_2\circ t\circ j_2$. Then $c\circ t'=c\circ p_2\circ t\circ j_2 =p_2\circ (1_X \ltimes c)\circ t\circ j_2\simeq p_2 \circ j_2=1_B$, that is, $t'$ is a right homotopy inverse of $c$.
\end{proof}

Lemma \ref{ltimeslemma} addresses the existence of right homotopy inverses around left half-smash products without looping. The proof is straightforward and standard. However, addressing the corresponding problem after looping requires additional effort. 

By Ganea's theorem (Theorem \ref{Ganeathm}), there is a diagram of homotopy fibrations
\begin{equation}\label{ltimesfibdiag}
\begin{aligned}
\xymatrix{
X\ar[r]^<<<<<{i_1}  \ar[d]^{} & X\times   A \ar[d]^{q_1} \ar[r]^<<<{\pi_2} &  A \ar@{=}[d] \\
X\ast \Omega A        \ar[r]^{}          &     X\ltimes A \ar[r]^<<<{p_2}       & A,
}
\end{aligned}
\end{equation}
and the bottom homotopy fibration splits after looping. In order to investigate right homotopy inverses around left half-smash products after looping, a key step is to show that specific model of the bottom homotopy fibration in (\ref{ltimesfibdiag}) is natural in $X$ and $A$.

Let $\Omega A\stackrel{i}{\longrightarrow} PA\stackrel{{\rm ev}}{\longrightarrow} A$ be the standard path fibration of $A$.   
The three fibrations 
\[
\begin{split}
X\times \Omega A\stackrel{1_X\times i}{\longrightarrow} X\times PA\stackrel{{\rm ev}\circ \pi_2}{\longrightarrow} A &,  \\
CX\times \Omega A\stackrel{1_{CX}\times i}{\longrightarrow} CX\times PA\stackrel{{\rm ev}\circ \pi_2}{\longrightarrow} A &,  \\
X\stackrel{i_1}{\longrightarrow} X\times A\stackrel{\pi_2}{\longrightarrow} A &
\end{split}
\]
are all natural in $X$ and $A$. Define the spaces $H$ and $Q$ to be the strict pushouts of 
\[
X\stackrel{\pi_1}{\longleftarrow} X\times \Omega A \stackrel{j\times 1_{\Omega A}}{\longrightarrow} CX\times \Omega A \ \  {\rm and} \ \ 
X\times A\stackrel{1_X\times {\rm ev}}{\longleftarrow} X\times P A \stackrel{j\times 1_{P A}}{\longrightarrow} CX\times P A,
\]
respectively, where $X\stackrel{j}{\hookrightarrow} CX$ is the natural inclusion and $\pi$ is the natural projection. 
The two pushouts, together with the three fibrations, determine a strictly commutative cube
 \[
  \label{} 
  %\spreaddiagramcolumns{-0.3pc}\spreaddiagramrows{-0.3pc} 
   \diagram
      X\times \Omega A \rrto^-{ j\times 1_{\Omega A}}\drto^-{\pi_1}\ddto^-(0.33){1_X\times i} & & CX\times \Omega A\dline^(0.65){1_{CX}\times i}\drto^{} & \\
      & X\rrto^(0.35){}\ddto^(0.25){i_1} & \dto & H \ddto^{h} \\
      X\times PA \rline^(0.6){ j\times 1_{P A}}\drto^(0.6){1_X\times {\rm ev}} & \rto & CX\times PA\drto^{} & \\
      & X\times A\rrto^{} & &  Q.
  \enddiagram 
\] 
It is clear that the rear and left faces are pullbacks. The cube diagram illustrates how the three fibrations are glued together to form a fibration
\[
H\stackrel{h}{\longrightarrow} Q\stackrel{p}{\longrightarrow} A. 
\]
As each involved map is natural, the universal property of strict pushout implies that the structural maps $h$ and $p$ are unique and natural in $X$ and $A$. As we will see later, this strict fibration is a model of the bottom homotopy fibration in Diagram (\ref{ltimesfibdiag}). 

Similarly, the strictly commutative diagram
\[
\xymatrix{
X\times A           \ar@{=}[d] &  X  \ar[d]^{i_1}   \ar[l]_<<<<<{i_1}  \ar[r]^{j}  &  CX \ar[d]^{i_1} \\
X\times A         &  X\times PA  \ar[l]_{1_{X}\times {\rm ev}} \ar[r]^{j\times 1_{PA}}    &  CX\times PA  
}
\]
induces a natural map of pushouts
\[
\varphi_Q:  (X\times A)\cup_X CX \stackrel{\simeq}{\longrightarrow} Q.
\]
Since $j$ and $j\times 1_{PA}$ are cofibrations and $PA$ is contractible, the map $\varphi_Q$ is a homotopy equivalence by the following classical {\it gluing lemma}. 
\begin{lemma}{\cite[Lemma 2.1.3]{MP12}, \cite[Theorem 1.13]{FHT01}}\label{gluinglemma}
Suppose that there is a commutative diagram
\[
\diagram
C \dto^{\lambda_C} & B \lto_{f}  \rto^{g} \dto^{\lambda_B} & D \dto^{\lambda_D}\\
C'                            & B' \lto_{f'}  \rto^{g'}                           & D'
\enddiagram
\]
in which $g$ and $g'$ are cofibrations. If $\lambda_B$, $\lambda_C$ and $\lambda_D$ are homotopy equivalences, then the induced map of pushouts is a homotopy equivalence. $\qqed$
\end{lemma}

\begin{lemma}\label{ltimesplemma}
There is a strictly commutative diagram
\[
\xymatrix{
 (X\times A)\cup_X CX \ar[r]^<<<<{p_{2}'} \ar[d]^{\varphi_Q}_{\simeq} & A\ar@{=}[d]\\
Q\ar[r]^{p}  & A,
}
\]
where the natural map $p_2'$ restricts to $\pi_2$ on $X\times A$ and is the constant map on $CX$.
\end{lemma}
\begin{proof}
Consider the strictly commutative diagram 
\[
\xymatrix{
X\times A           \ar@{=}[d] &  X  \ar[d]^{i_1}   \ar[l]_<<<<<{i_1}  \ar[r]^{j}  &  CX \ar[d]^{i_1} \\
X\times A        \ar[d]^{\pi_2}   &  X\times PA \ar[d]^{{\rm ev}\circ \pi_2}   \ar[l]_{1_{X}\times {\rm ev}} \ar[r]^{j\times 1_{PA}}    &  CX\times PA  \ar[d]^{{\rm ev}\circ \pi_2}\\
A \ar@{=}[r]  & A \ar@{=}[r]  &A.
}
\]
By construction, the pushouts of the three rows are $(X\times A)\cup_X CX$, $Q$ and $A$, respectively, and the diagram induces a natural composite of pushouts
\[
(X\times A)\cup_X CX \stackrel{\varphi_Q}{\longrightarrow} Q\stackrel{p}{\longrightarrow} A.
\]
Note that the composite is $\pi_2$ on $X\times A$, and is the constant map on $CX$ as $\pi_2\circ i_1$ is trivial. 
\end{proof}

\begin{lemma}\label{ltimeshlemma}
There is a homotopy commutative diagram
\[
\xymatrix{
X\ast \Omega A \ar[r]^<<<<{\pi'} \ar[d]^{\varphi_H}_{\simeq} & (X\times A)\cup_X CX \ar[d]^{\varphi_Q}_{\simeq}\\
H\ar[r]^{h}  & Q,
}
\]
where $\varphi_H$ and $\varphi_Q$ are homotopy equivalences, and all the maps are natural in $X$ and $A$. 
\end{lemma}

\begin{proof}
Recall that the unit interval $I=[0,1]$ has $0$ as basepoint. Hence, the inclusion $j: X\stackrel{}{\hookrightarrow} CX$ is defined by $j(x)=[1, x]$, and the evaluation map $j: PA\stackrel{}{\longrightarrow} A$ is defined by ${\rm ev}(\lambda)=\lambda(1)$. Additionally, define a map $c: C\Omega A \stackrel{}{\longrightarrow} A$ by $c([t', \omega])=\omega(1-t')$. Then the composite
\[
\Omega A\stackrel{j}{\longrightarrow} C\Omega A \stackrel{c}{\longrightarrow} A
\]
is the constant map, as $c\circ j (\omega)=c([1, \omega])=\omega (0)=\ast$. 

Consider the strictly commutative diagram
\[
\xymatrix{
X\times C\Omega A  \ar[d]^{\pi_1}  &  X\times \Omega A \ar@{=}[d]  \ar[l]_{1_X\times j}  \ar[r]^{j\times 1_{\Omega A}}  &  CX\times \Omega A \ar@{=}[d]  \\
X             \ar[d]^{i_1} &  X\times \Omega A  \ar[d]^{1_X\times i}   \ar[l]_{\pi_1}  \ar[r]^{j\times 1_{\Omega A}}  &  CX\times \Omega A \ar[d]^{1_{CX}\times i} \\
X\times A           &  X\times PA   \ar[l]_{1_{X}\times {\rm ev}} \ar[r]^{j\times 1_{PA}}    &  CX\times PA.  
}
\]
The pushout of the first row is the join $X\ast\Omega A$ by definition, and the pushouts of the second and third rows are $H$ and $Q$, respectively. Therefore, the diagram induces a composition of pushouts 
\[
X\ast \Omega A\stackrel{\varphi_H}{\longrightarrow} H\stackrel{h}{\longrightarrow} Q.
\] 
Since the maps $j\times 1_{\Omega A}$ are cofibrations and $C\Omega A$ is contractible, the map $\varphi_H$ is a homotopy equivalence by Lemma \ref{gluinglemma}. Since the diagram is natural in $X$ and $A$, so are $\varphi_H$ and $h$. 
Similarly, the strictly commutative diagram
\[
\xymatrix{
X\times C\Omega A  \ar[d]^{1_X\times c}  &  X\times \Omega A \ar[d]^{\pi_1} \ar[l]_{1_X\times j}  \ar[r]^{j\times 1_{\Omega A}}  &  CX\times \Omega A \ar[d]^{\pi_1}  \\
X\times A           \ar@{=}[d] &  X  \ar[d]^{i_1}   \ar[l]_<<<<<{i_1}  \ar[r]^{j}  &  CX \ar[d]^{i_1} \\
X\times A         &  X\times PA   \ar[l]_{1_{X}\times {\rm ev}} \ar[r]^{j\times 1_{PA}}    &  CX\times PA 
}
\]
induces a composition of pushouts 
\[
X\ast \Omega A\stackrel{\pi'}{\longrightarrow} (X\times A)\cup_X CX \stackrel{\varphi_Q}{\longrightarrow} Q
,\] 
in which $\varphi_Q$ is a homotopy equivalence. Since the diagram is natural in $X$ and $A$, so are $\pi'$ and $\varphi_H$.

The two composites $h\circ \varphi_H$ and $\varphi_Q\circ \pi'$ are not equal, but an explicit homotopy between them can be constructed. Indeed, defines three homotopies 
\[
\begin{split}
H^l:   X\times C\Omega A\times I \stackrel{}{\longrightarrow} X\times A, \ \ 
& H^l(x, [t',\omega], s)=(x, \omega(1-(1-s)t')),\\
H^m: X\times \Omega A\times I \stackrel{}{\longrightarrow} X\times PA, \ \ 
&H^m(x, \omega, s)=(x, \omega_s), \ \  \omega_s(t)=\omega(st), \\
H^r: CX\times \Omega A\times I \stackrel{}{\longrightarrow} CX\times PA, \ \ 
&H^r([t', x], \omega, s)=([t', x], \omega_s), \ \  \omega_s(t)=\omega(st).
\end{split}
\]
It is straightforward to check that 
\[
\begin{split}
H^l_0=1_X\times c, \ \ H^m_0=i_1\circ \pi_1, \ \ H^r_0=i_1\circ \pi_1,\\
H^l_1=i_1\circ \pi_1, \ \ H^m_1=1_X\times i, \ \  H^r_1=1_{CX}\times i, 
\end{split}
\]
and 
the diagram 
\[
\xymatrix{
X\times C\Omega A  \times I \ar[d]^{H^l}  &  X\times \Omega A \times I \ar[d]^{H^m} \ar[l]_{1_X\times j\times 1_I}  \ar[r]^{j\times 1_{\Omega A}\times 1_I}  &  CX\times \Omega A \times I\ar[d]^{H^r}  \\
X\times A         &  X\times PA   \ar[l]_{1_{X}\times {\rm ev}} \ar[r]^{j\times 1_{PA}}    &  CX\times PA 
}
\]
strictly commutes. 
Therefore, it induces a morphism of pushouts
\[
H: (X\ast \Omega A) \times I\stackrel{}{\longrightarrow} Q
\]
such that $H_0=\varphi_Q\circ \pi'$ and $H_1=h\circ \varphi_H$. In particular, $\varphi_Q\circ \pi'$ and $h\circ \varphi_H$ are homotopic.
\end{proof}

\begin{proposition}\label{ltimesfibrelemma}
There is a natural homotopy fibration
\[
\Sigma X\wedge \Omega A \stackrel{ }{\longrightarrow} X\ltimes A \stackrel{p_2}{\longrightarrow} A,
\]
which splits after looping to give a natural homotopy equivalence
\[
\Omega A \times \Omega (\Sigma X\wedge \Omega A )\stackrel{\simeq}{\longrightarrow}\Omega (X\ltimes A). 
\]
\end{proposition}
\begin{proof}
Consider the diagram 
\[
\xymatrix{
X\ast \Omega A \ar[r]^<<<<{\pi'} \ar[d]^{\varphi_H}_{\simeq} & (X\times A)\cup_X CX \ar[r]^<<<<{p_{2}'} \ar[d]^{\varphi_Q}_{\simeq} & A\ar@{=}[d]\\
H\ar[r]^{h}  & Q\ar[r]^{p}  & A,
}
\]
where the bottom row is a fibration by construction. 
The left square homotopy commutes by Lemma \ref{ltimeshlemma}, and the right square commutes by Lemma \ref{ltimesplemma}. Accordingly, the top row 
\[
X\ast \Omega A \stackrel{\pi'}{\longrightarrow}  (X\times A)\cup_X CX \stackrel{p_2'}{\longrightarrow} A
\] 
is a homotopy fibration.

By the definition of $p_2'$ in Lemma \ref{ltimesplemma}, it factors as a composite
\[
p_2': (X\times A)\cup_X CX  \stackrel{\psi}{\longrightarrow}   X\ltimes A \stackrel{p_2}{\longrightarrow} A,
\]
where the map $\psi$ pinches the cone $CX$ to a point. It is clear that $\psi$ is a natural homotopy equivalence. Further, it is known that $X\ast \Omega A$ is naturally homotopy equivalent to $\Sigma X\wedge \Omega A$. Therefore, we can replace $(X\times A)\cup_X CX \stackrel{p_2'}{\longrightarrow} A$ and $X\ast \Omega A$ by $X\ltimes A \stackrel{p_2}{\longrightarrow} A$ and $\Sigma X\wedge \Omega A$, respectively, to obtain a natural homotopy fibration
\[
\Sigma X\wedge \Omega A \stackrel{ }{\longrightarrow} X\ltimes A \stackrel{p_2}{\longrightarrow} A.
\]
By Lemma \ref{ltimeslemma0}, the map $p_2$ has a natural section $j_2: A\stackrel{}{\longrightarrow} X\ltimes A$. Therefore, the homotopy fibration splits after looping and 
Corollary \ref{RHIfibcor} implies that the splitting is natural. 
\end{proof}

\begin{remark}
Recall Lemma \ref{eqnaturallemma} implies that if a map $ X\stackrel{\simeq}{\longrightarrow} Y$ is a natural homotopy equivalence, then its homotopy inverse is also natural. Therefore, Lemma \ref{ltimeshlemma} is not entirely necessary for the proof of Proposition \ref{ltimesfibrelemma}, except for the part stating that $\varphi_H$ is a natural homotopy equivalence. However, Lemma \ref{ltimeshlemma} is an interesting fact about related maps and provides an explicitly constructed homotopy fibration $X\ast \Omega A \stackrel{\pi'}{\longrightarrow}  (X\times A)\cup_X CX \stackrel{p_2'}{\longrightarrow} A$. Therefore, it is worth including. 
\end{remark}

We can now prove a loop version of Lemma \ref{ltimeslemma}.

\begin{proposition}\label{loopltimeslemma}
Let $c: A\stackrel{}{\longrightarrow} B$ be a map and $X$ be a space. Then $\Omega c$ has a right homotopy inverse if and only if $\Omega (1_X \ltimes c): \Omega (X\ltimes A) \longrightarrow \Omega (X\ltimes B)$ has a right homotopy inverse.
\end{proposition}
\begin{proof}
Proposition \ref{ltimesfibrelemma} implies that there is a diagram of homotopy fibrations
\[
\xymatrix{
\Sigma X\wedge \Omega A\ar[r] \ar[d]^{\Sigma 1_{X}\wedge \Omega c}  & X\ltimes A \ar[d]^{1_X\ltimes c} \ar[r]^<<<{p_2} & A \ar[d]^{c} \\
\Sigma X\wedge \Omega B \ar[r]  & X\ltimes B \ar[r]^<<<{p_2} & B. 
}
\] 
Since the right homotopy inverse $j_2$ of $p_2$ is natural, Corollary \ref{RHIfibcor} implies that $\Omega (1_X \ltimes c)$ has a right homotopy inverse if and only if both $\Omega (\Sigma 1_{X}\wedge \Omega c)$ and $\Omega c$ have right homotopy inverses, and then if and only if $\Omega c$ has a right homotopy inverse.
\end{proof}

\newpage
%----------------------------------------------------------------------------------------------------------------------------------------------------------------------------------------------------------%
\section{Homotopy cofibrations and homotopy pushouts}
\label{sec: pushout}

In this section, we study the existence of right homotopy inverses around homotopy cofibrations and homotopy pushouts. It is well-known that the loop functor does not behave well with respect to homotopy cofibrations. To overcome the difficulty, we apply the powerful decomposition result of Beben-Theriault in Section \ref{sec: BT} to transform the problem into one around homotopy fibrations or homotopy pullbacks, and then employ the results of Section \ref{sec: pullback} and \ref{sec: halfsmash}. 

In the following, Proposition \ref{RHIcofprop} studies the existence of right homotopy inverses under compatible conditions between homotopy cofibrations. 

\begin{proposition}\label{RHIcofprop}
Let
\[\diagram 
      A\rto^-{f}\dto^{\lambda_A} & Y\rto^-{h}\dto^{\lambda_Y} & Z\dto^{\lambda_Z}  \\ 
      A'\rto^-{f'} & Y'\rto^-{h'} & Z' 
  \enddiagram\] 
 be a diagram of homotopy cofibrations. Suppose that $\Omega h$ and $\Omega h'$ have right homotopy inverses $s$ and $s'$, respectively, such that there is a homotopy commutative diagram 
\[
\diagram 
      \Omega Z\rto^-{s}\dto^{\Omega\lambda_Z} & \Omega Y\dto^{\Omega \lambda_Y} \\ 
      \Omega Z'\rto^-{s'} & \Omega Y'.
  \enddiagram
  \]
  Then $\Omega \lambda_Y$ has a right homotopy inverse if and only if both $\Omega \lambda_A$ and $\Omega\lambda_Z$ have right homotopy inverses. 
 \end{proposition}
 \begin{proof}
 By Theorem \ref{GTcofib} and Proposition \ref{GTcofibnat}, there is a homotopy fibration diagram
   \[
  \begin{aligned}
  \xymatrix{
  \Omega Z\ltimes A \ar[r]^<<<<{\widetilde{\Gamma}} \ar[d]^{\Omega \lambda_Z\ltimes \lambda_A}  & Y \ar[r]^{h} \ar[d]^{\lambda_Y} & Z \ar[d]^{\lambda_Z}\\
   \Omega Z'\ltimes A' \ar[r]^<<<<{\widetilde{\Gamma}'}                           & Y' \ar[r]^{h'} & Z'.
  }
  \end{aligned}
  \]
 Applying Corollary \ref{RHIfibcor} to this diagram, we see that $\Omega \lambda_Y$ has a right homotopy inverse if and only if $\Omega \lambda_Z$ and $\Omega (\Omega \lambda_Z\ltimes \lambda_A)$ have right homotopy inverses.
It remains to show that the latter holds if and only if both $\Omega\lambda_Z$ and $\Omega \lambda_A$ have right homotopy inverses. 

Suppose that $\Omega\lambda_Z$ and $\Omega \lambda_A$ have right homotopy inverses. Then $\Omega \lambda_Z\ltimes 1_{A}$ has a right homotopy inverse, and $\Omega (1_{\Omega Z'}\ltimes \lambda_A)$ has a right homotopy inverse by Proposition \ref{loopltimeslemma}. 
As $\Omega \lambda_Z\ltimes \lambda_A=(1_{\Omega Z'}\ltimes \lambda_A)\circ (\Omega \lambda_Z\ltimes 1_{A})$, it follows that $\Omega (\Omega \lambda_Z\ltimes \lambda_A)$ has a right homotopy inverse. 

Conversely, suppose that $\Omega \lambda_Z$ and $\Omega (\Omega \lambda_Z\ltimes \lambda_A)$ have right homotopy inverses. As $\Omega \lambda_Z\ltimes \lambda_A=(1_{\Omega Z'}\ltimes \lambda_A)\circ (\Omega \lambda_Z\ltimes 1_{A})$, it follows that $\Omega (1_{\Omega Z'}\ltimes \lambda_A)$ has a right homotopy inverse, and then so does $\Omega \lambda_A$ by Proposition \ref{loopltimeslemma}. 

Hence, $\Omega\lambda_Z$ and $\Omega \lambda_A$ have right homotopy inverses if and only if both $\Omega \lambda_Z$ and $\Omega (\Omega \lambda_Z\ltimes \lambda_A)$ have right homotopy inverses, and then if and only if $\Omega \lambda_Y$ has a right homotopy inverse. 
 \end{proof}

 \begin{corollary}\label{RHIwedgecoro}
 Let $A\stackrel{\lambda_A}{\longrightarrow} A'$ and $B\stackrel{\lambda_B}{\longrightarrow} B'$ be two maps. Then $\Omega (\lambda_A\vee \lambda_B)$ has a right homotopy inverse if and only if both $\Omega \lambda_A$ and $\Omega \lambda_B$ have right homotopy inverses.
\end{corollary}
\begin{proof}
The condition gives a diagram of trivial homotopy cofibrations
\[\diagram 
      A\rto^-{i_1}\dto^{\lambda_A} & A\vee B\rto^-{q_2}\dto^{\lambda_A \vee \lambda_B} & B\dto^{\lambda_B}  \\ 
      A'\rto^-{i_1} & A'\vee B'\rto^-{q_2} & B', 
  \enddiagram\] 
  where $i_1$ and $q_2$ are the inclusions and projections. Note that $A\vee B\stackrel{q_2}{\longrightarrow} B$ and $A'\vee B'\stackrel{q_2}{\longrightarrow} B'$ have the inclusions $B\stackrel{i_2}{\longrightarrow} A\vee B$ and $B'\stackrel{i_2}{\longrightarrow} A'\vee B'$ as right homotopy inverses, respectively, such that $(\lambda_A \vee \lambda_B)\circ i_2 \simeq i_2\circ \lambda_B$. Hence, the homotopy cofibration diagram satisfies the condition of Proposition \ref{RHIcofprop}, and then the corollary follows. 
\end{proof}

We turn to study the existence of right homotopy inverses around homotopy pushouts. The related results are summarized in Theorem \ref{pushoutthm}, Theorem \ref{pushoutthm2} and Remark \ref{pushoutthm3}.

\begin{theorem}\label{pushoutthm}
Let
\begin{equation}\label{pushoutdiag}
\diagram 
      X'\rto^-{\varphi'}\dto^-{f} & Y'\dto^-{g} \\ 
      X\rto^-{\varphi} & Y
  \enddiagram
  \end{equation}
be a homotopy pushout. Suppose that the map $\varphi'$ is inert. Then $\Omega f$ has a right homotopy inverse if and only if $\Omega g$ has a right homotopy inverse.
\end{theorem}
\begin{proof}
Let $C$ be the homotopy cofibre of $\varphi$. As (\ref{pushoutdiag}) is a homotopy pushout, $C$ is also the homotopy cofibre of $\varphi$, and there is a diagram of homotopy cofibrations
\[
\diagram 
      X'\rto^-{\varphi'}\dto^-{f} & Y'\dto^-{g} \rto^-{p'} & C\ddouble  \\ 
      X\rto^-{\varphi} & Y\rto^-{p} & C,
  \enddiagram
  \]
  which defines the maps $p'$ and $p$. Since $\varphi'$ is inert, the map $\Omega p'$ admits a right homotopy inverse $s: \Omega C\longrightarrow \Omega Y'$. As $p\circ g\simeq p'$, it implies that $\Omega g\circ s$ is a right homotopy inverse of $\Omega p$. In particular, the maps $p'$ and $p$ have compatible right homotopy inverses after looping. Therefore, we can apply Proposition \ref{RHIcofprop} to the above homotopy cofibration diagram, and show that $\Omega g$ has a right homotopy inverse if and only if $\Omega f$ has a right homotopy inverse.
\end{proof}

In Theorem \ref{pushoutthm} we suppose that the map $\varphi'$ is inert. 
In the following we may consider another situation that $\Omega \varphi'$ has a right homotopy inverse. To study this case, we may need some standard results of homotopy actions.  

Recall a homotopy fibration $F\stackrel{}{\longrightarrow} E\stackrel{}{\longrightarrow} B$ can be extended to the left to give a connecting map $\delta: \Omega B\stackrel{}{\longrightarrow} F$. Further, there is a natural {\it homotopy principal action}
\[
\theta: \Omega B\times F\stackrel{}{\longrightarrow} F
\]
that extends the wedge sum map $\Omega B \vee F\stackrel{\delta\vee 1_F}{\longrightarrow} F \vee F \stackrel{\nabla}{\longrightarrow}F$, where $\nabla$ is the folding map. In particular, if $\delta$ is null homotopic, then the action reduces to a map 
\[
\overline{\theta}: \Omega B\ltimes F\stackrel{}{\longrightarrow} F
\]
through the natural projection $\Omega B\times F \stackrel{q_1}{\longrightarrow} \Omega B\ltimes F$. This fact will be used freely in the proof of Theorem \ref{pushoutthm2}. 

\begin{theorem}\label{pushoutthm2}
Let
\begin{equation}\label{Apushoutdiag}
\diagram 
      A \ddouble  \rto^-{a'} & X'\rto^-{\varphi'}\dto^-{f} & Y'\dto^-{g} \\ 
     A \rto^-{a}          & X\rto^-{\varphi} & Y
  \enddiagram
  \end{equation}
be a diagram of homotopy cofibrations. Suppose that the map $a'$ is inert. If $\Omega g$ has a right homotopy inverse, then $\Omega f$ also has a right homotopy inverse.
\end{theorem}
\begin{proof}
By assumption, let $s:\Omega Y'\stackrel{}{\longrightarrow} \Omega X'$ and $r: \Omega Y\stackrel{}{\longrightarrow} \Omega Y'$ be right homotopy inverses of $\Omega \varphi'$ and $\Omega g$, respectively. Let $t=\Omega f\circ s\circ r$. Then $\Omega \varphi\circ t=\Omega \varphi\circ \Omega f\circ s\circ r\simeq \Omega g \circ \Omega \varphi'\circ s\circ r\simeq \Omega g \circ 1_{\Omega Y'}\circ r\simeq 1_{\Omega Y}$, that is, $t$ is a right homotopy inverse of $\Omega \varphi$. 
In particular, Theorem \ref{GTcofib} applies to give two homotopy fibrations
\[
\Omega Y'\ltimes A \stackrel{}{\longrightarrow} X'\stackrel{\varphi'}{\longrightarrow} Y', \ \ 
\Omega Y\ltimes A \stackrel{\widetilde{\Gamma}}{\longrightarrow} X\stackrel{\varphi}{\longrightarrow} Y,
\]
both of which split after looping. Unfortunately, neither Propositions \ref{GTcofibnat}, \ref{RHIcofprop} nor Corollary \ref{RHIfibcor} can apply since the homotopy right inverses $s$ and $t$ are not necessarily compatible. 

In order to overcome the naturality issue, consider the diagram of homotopy fibrations
\begin{equation}\label{Apushoutpfdiag1}
\diagram 
      F \dto^-{\tau}  \rto^-{i} & X'\rto^-{g\circ\varphi'}\dto^-{f} & Y\ddouble \\ 
     \Omega Y\ltimes A \rto^-{\widetilde{\Gamma}}          & X\rto^-{\varphi} & Y,
  \enddiagram
\end{equation}
where $F$ is the homotopy fibre of the composite $g\circ\varphi'$, and $i$ and $\tau$ are the induced maps. Note that $s\circ r$ and $t=\Omega f \circ s\circ r$ are right homotopy inverses of $\Omega(g\circ\varphi')$ and $\Omega\varphi$, respectively, and they are compatible. Therefore, Corollary \ref{RHIfibcor} can be applied to show that there are compatible splittings 
\begin{equation}\label{Apushoutpfdiag1.5}
\begin{aligned}
\xymatrix{
\Omega Y \times \Omega F \ar[r]^>>>>>>{\chi'}_>>>>>>{\simeq} \ar[d]^{1_{\Omega Y}\times \Omega \tau} & \Omega X'\ar[d]^{\Omega f} \\
\Omega Y \times \Omega (\Omega Y\ltimes A)\ar[r]^<<<{\chi}_<<<{\simeq}  & \Omega X,
}
\end{aligned}
\end{equation}
and moreover $\Omega f$ has a right homotopy inverse if and only if $\Omega\tau$ has a right homotopy inverse. Hence, to prove the theorem it remains to show that $\Omega\tau$ has a right homotopy inverse.

From (\ref{Apushoutpfdiag1.5}) it is clear that the connecting maps of the two homotopy fibrations in Diagram (\ref{Apushoutpfdiag1}) are null homotopic, and therefore there is a diagram of reduced homotopy actions
\begin{equation}\label{Apushoutpfdiag2}
\begin{aligned}
\xymatrix{
\Omega Y\ltimes F \ar[r]^<<<<<<<{\overline{\theta}} \ar[d]^{\Omega 1_{Y} \ltimes \tau} &  F \ar[d]^{\tau} \\
\Omega Y\ltimes (\Omega Y\ltimes A) \ar[r]^<<<{\overline{\theta}}  &  \Omega Y\ltimes A.
}
\end{aligned}
\end{equation}
Further, By Theorem \ref{GTcofib} and Diagram (\ref{Apushoutdiag}), the map $\Omega Y\ltimes A\stackrel{\widetilde{\Gamma}}{\longrightarrow} X$ in Diagram (\ref{Apushoutpfdiag1}) satisfies that $\widetilde{\Gamma}\circ j_2\simeq a \simeq f\circ a'$, where $j_2: A\stackrel{}{\longrightarrow}\Omega Y\ltimes A$ is the natural inclusion in Lemma \ref{ltimeslemma0}. Then as the left square of Diagram (\ref{Apushoutpfdiag1}) is a homotopy pullback, there is a map $j': A\stackrel{}{\longrightarrow } F$ such that the two triangular regions of the following diagram homotopy commute
\[\xymatrix{ 
A         \ar[dr]^(0.6){j'}       \ar@/^0.8pc/[drr]^{a'}    \ar@/_0.8pc/[ddr]_{j_2}  &&\\
&F\rto^-{i}\dto^-{\tau} & X'\dto^-{f} \\ 
&      \Omega Y\ltimes A\rto^-{\widetilde{\Gamma}} & X.
}\]
Combining Diagram (\ref{Apushoutpfdiag2}) and the left triangle $\tau\circ j'\simeq j_2$ gives a homotopy commutative diagram
\[
\xymatrix{
\Omega Y\ltimes A \ar@{=}[d] \ar[r]^{1_{\Omega Y}\ltimes j'} &\Omega Y\ltimes F \ar[r]^<<<<<<<{\overline{\theta}} \ar[d]^{\Omega 1_{Y} \ltimes \tau} &  F \ar[d]^{\tau} \\
\Omega Y\ltimes A  \ar[r]^<<<<{1_{\Omega Y}\ltimes j_2}  &\Omega Y\ltimes (\Omega Y\ltimes A) \ar[r]^<<<{\overline{\theta}}  &  \Omega Y\ltimes A.
}
\]
The proof of Theorem \ref{GTcofib} in \cite{BT22} shows that the composite $\overline{\theta}\circ (1_{\Omega Y}\ltimes j_2)$ in the bottom row is a homotopy equivalence. Hence, the diagram implies that $\tau$ has a right homotopy inverse. In particular, $\Omega \tau$ has a right homotopy inverse and so does $\Omega f$.
\end{proof}

\begin{remark}\label{pushoutthm2rmk}
The converse of Theorem \ref{pushoutthm2} does not hold in general. 
That is, for the homotopy cofibration diagram (\ref{Apushoutdiag}) such that $a'$ is inert, it is possible that $\Omega f$ has a right homotopy inverse, but $\Omega g$ does not have a right homotopy inverse.  

For instance, consider the homotopy cofibration diagram
\[
\diagram 
      S^{m+n-1}\rto^-{[i_1, i_2]}\ddouble & S^m\vee S^n\rto^-{}\dto^{q_1} & S^{m}\times S^{n}\dto^{}  \\ 
      S^{m+n-1}\rto^-{\ast} & S^m \rto^-{} &   S^m\vee S^{m+n},
  \enddiagram
\]
where $[i_1, i_2]$ is the Whitehead product of the two wedge summand inclusions $S^m\stackrel{i_1}{\hookrightarrow} S^m\vee S^n$ and $S^n\stackrel{i_2}{\hookrightarrow} S^m\vee S^n$, and $q_1$ is the projection map. Note that the map $[i_1, i_2]$ is inert by the Hilton-Milnor theorem, and the map $q_1$ clearly has a right homotopy inverse. However, the map $S^{m}\times S^{n}\stackrel{}{\longrightarrow} S^m\vee S^{m+n}$ does not have a right homotopy inverse after looping. This can be seen from the fact that the homomorphism of homotopy groups
\[
\pi_{\ast}(S^m\times S^n)\stackrel{}{\longrightarrow} \pi_{\ast}(S^m \vee S^{m+n})
\]
is neither injective nor surjective.
\end{remark}

\begin{remark}\label{pushoutthm3}
A dual result of Theorem \ref{pushoutthm2} is as follows. Consider the homotopy pushout (\ref{pushoutdiag}). Suppose that the map $\Omega \varphi$ has a right homotopy inverse. If $\Omega f$ has a right homotopy inverse, then $\Omega g$ has a right homotopy inverse. 
Note that this result follows from the homotopy commutativity of the square (\ref{pushoutdiag}) immediately.

The converse of this result does not hold in general. That is, for the homotopy pushout (\ref{pushoutdiag}) such that $\Omega \varphi$ has a right homotopy inverse, it is possible that $\Omega g$ has a right homotopy inverse, but $\Omega f$ does not have a right homotopy inverse.  

For instance, as $A\ltimes  \Sigma B\simeq \Sigma B\vee (A\wedge \Sigma B)$ there is a homotopy cofibration diagram 
\[
\diagram 
 & \Sigma Y\rdouble \dto^{} & \Sigma Y \dto^{} \\
      \Sigma X\rto^-{}\ddouble & \Sigma X\times \Sigma Y\rto^-{}\dto^{} & \Sigma Y\vee (\Sigma X\wedge \Sigma Y) \dto^{}  \\ 
      \Sigma X\rto^-{} & \Sigma X\vee (\Sigma X\wedge \Sigma Y) \rto^-{} &   \Sigma X\wedge \Sigma Y,
  \enddiagram
\]
in which the maps are the obvious inclusions and projections. It is clear that both $\Sigma X\vee (\Sigma X\wedge \Sigma Y)\stackrel{}{\longrightarrow} \Sigma X\wedge \Sigma Y$ and $\Sigma Y\vee (\Sigma X\wedge \Sigma Y)\stackrel{}{\longrightarrow} \Sigma X\wedge \Sigma Y$ have right homotopy inverses. However, the maps $\Sigma X\times \Sigma Y\stackrel{}{\longrightarrow}\Sigma X\vee (\Sigma X\wedge \Sigma Y)$ and $\Sigma X\times \Sigma Y\stackrel{}{\longrightarrow}\Sigma Y\vee (\Sigma X\wedge \Sigma Y)$ do not have right homotopy inverses after looping in general. For example, for $X=S^{m-1}$ and $Y=S^{n-1}$, the map $\Sigma X\times \Sigma Y\stackrel{}{\longrightarrow}\Sigma X\vee (\Sigma X\wedge \Sigma Y)$ is the map $S^m\times S^n \stackrel{}{\longrightarrow}S^m \vee S^{m+n}$ in Remark \ref{pushoutthm2rmk}, which does not have a right homotopy inverse.
\end{remark}

\newpage
%----------------------------------------------------------------------------------------------------------------------------------------------------------------------------------------------------------%
\section{Nonzero degree maps}
\label{sec: deg1}
In this section, we apply the results for homotopy pushouts in Section \ref{sec: pushout} to analyze nonzero degree maps between Poincar\'{e} duality complexes and prove Theorem \ref{inertdeg1thmintro}. It provides fundamental criteria for inertness by comparing two Poincar\'{e} duality complexes of the same dimension. Additionally, since mapping degree can be defined for maps between suitable finite $CW$-complexes that are not necessarily Poincar\'{e} duality complexes, we can work in a slightly broader context.

Let $X$ be a connected $CW$-complex. Its {\it formal dimension} $fd(X)$ is defined to be the largest integer $n$ such that $H_n(X;\mathbb{Z})$ is nontrivial.
Suppose that $f: X\stackrel{}{\longrightarrow} Y$ is a map such that $fd(Y)=fd(X)=n$ and $H_n(X;\mathbb{Z})\cong H_n(Y;\mathbb{Z})\cong \mathbb{Z}$. Then the induced homomorphism
\[
f_\ast: H_n(X;\mathbb{Z})\stackrel{}{\longrightarrow} H_n(Y;\mathbb{Z})
\]
determines a unique integer $k$ up to sign by $f_\ast([X])=k\cdot [Y]$, where {\it the fundamental classes} $[X]$ and $[Y]$ are generators of $H_n(X;\mathbb{Z})$ and $H_n(Y;\mathbb{Z})$, respectively. We call $k$ the {\it degree} of the map $f$, and denote it by ${\rm deg}(f)$. Since the sign of the fundamental classes can be freely changed, the sign of the degree is irrelevant.

Suppose further that both $X$ and $Y$ each have a single $n$-cell. There are homotopy cofibre sequences
\begin{equation}\label{XYcofieq}
S^{n-1}\stackrel{h_X}{\longrightarrow} X_0\stackrel{i_X}{\longrightarrow}X\stackrel{q_X}{\longrightarrow}S^n, \ \  \ \ 
S^{n-1}\stackrel{h_Y}{\longrightarrow} Y_0\stackrel{i_Y}{\longrightarrow}Y\stackrel{q_Y}{\longrightarrow}S^n,
\end{equation}
where $X_0$ and $Y_0$ are the $(n-1)$-skeletons of $X$ and $Y$ with the inclusions $i_X$ and $i_Y$, respectively, the map $h_X$ and $h_Y$ are the attaching map of the top cells, and the maps $q_X$ and $q_Y$ are the pinch maps to the top cells.   

\begin{lemma}\label{degk=pushoutlemma}
There exists a degree $k$ map $f: X\stackrel{}{\longrightarrow} Y$ if and only if there is a diagram of homotopy cofibrations
\begin{equation}\label{degk=pushouteq}
\diagram 
      X_0\rto^-{i_X}\dto^{f_0} & X\dto^{f}   \rto^-{q_X}&  S^{n}\dto^{k} \\ 
     Y_0\rto^-{i_Y}  &   Y \rto^-{q_Y} & S^{n} ,
  \enddiagram
\end{equation}
where $f_0$ is the restriction of $f$ up to homotopy, and $k$ is a degree $k$ self-map of the sphere $S^n$.
\end{lemma}
\begin{proof}
Suppose that there is a diagram of homotopy cofibrations of the form (\ref{degk=pushouteq}). 
As $q_X$ and $q_Y$ induce isomorphisms on the top homology, the homotopy commutativity of the right square of (\ref{degk=pushouteq}) implies that ${\rm deg}(f)$ is equal to the degree of the map $k$, and hence is equal to $k$.

Conversely, suppose that a map $f: X\stackrel{}{\longrightarrow} Y$ is of degree $k$. Then up to homotopy its restriction on the $(n-1)$-skeleton $X_0\stackrel{i_X}{\longrightarrow} X\stackrel{f}{\longrightarrow} Y$ factors through $Y_0$ by the $CW$-approximation, that is, the left square of Diagram (\ref{degk=pushouteq}) exists. Therefore, there is a diagram of homotopy cofibration
\[
\diagram 
      X_0\rto^-{i_X}\dto^{f_0} & X\dto^{f}\rto^-{q_X} & S^{n}\dto^{f_s}  \\ 
      Y_0\rto^-{i_Y} &   Y \rto^-{q_Y}  & S^{n},
  \enddiagram
\]
where $f_s$ is the induced map. Since $f$ is of degree $k$, the map $f_s$ is of degree $k$ and therefore homotopic to the map $k$. Accordingly, the above diagram is the required Diagram (\ref{degk=pushouteq}).
\end{proof}

Degree one maps are of special interest. In this case Lemma \ref{degk=pushoutlemma} can be slightly strengthened. 
\begin{lemma}\label{deg1=pushoutlemma}
Suppose that $Y$ is simply connected.  
There exists a degree one map $f: X\stackrel{}{\longrightarrow} Y$ if and only if there is a diagram of homotopy cofibrations
\begin{equation}\label{deg1=pushouteq}
\diagram 
      S^{n-1}\rto^-{h_X}\ddouble & X_0\rto^-{i_X}\dto^{f_0} & X\dto^{f}  \\ 
      S^{n-1}\rto^-{h_Y} & Y_0\rto^-{i_Y} &   Y.
  \enddiagram
\end{equation}
\end{lemma}
\begin{proof}
Suppose that there is a diagram of the form (\ref{deg1=pushouteq}). By extending the row homotopy cofibrations in Diagram (\ref{deg1=pushouteq}) one step to the right, we obtain a diagram of the form (\ref{degk=pushouteq}) with $k=1$. Then Lemma \ref{degk=pushoutlemma} implies that $f$ is of degree one. 

Conversely, suppose that $f$ is of degree one. By Lemma \ref{degk=pushoutlemma} there is a diagram of the form (\ref{degk=pushouteq}) with $k=1$. Let $Y'$ be the homotopy cofibre of the composite $f_0\circ h_X$. There is a diagram of homotopy cofibrations
\begin{equation}\label{deg1=pushouteq0}
\diagram 
      S^{n-1}\rto^-{h_X}\ddouble & X_0\rto^-{i_X}\dto^{f_0} & X\dto^{f'}  \\ 
      S^{n-1}\rto^-{f_0\circ h_X} & Y_0\rto^-{i_Y'} &   Y'. 
  \enddiagram
\end{equation}
In particular, the right square is a homotopy pushout and $f'$ is of degree one by the previous argument. 
Consider the homotopy commutative diagram
\[
\xymatrix{ 
X_0 \ar[d]^{f_0} \ar[r]^{i_X} & X \ar[d]^{f'}  \ar@/^0.7pc/[ddr]^{f} \\
Y_0\ar[r]^{i_Y'}  \ar@/_0.7pc/[drr]_{i_Y}  &  Y' \ar@{.>}[dr]^(0.3){y} \\
&&Y,
}\]
where the inner square is the right square of Diagram \eqref{deg1=pushouteq0} and the outer square is the left square of Diagram \eqref{degk=pushouteq}. The universal property of homotopy pushout implies that there is a map $Y'\stackrel{y}{\longrightarrow} Y$ such that the two triangular regions homotopy commute. 
In particular, $y\circ i_Y'\simeq i_Y$ implies that $y$ restricts to the identity map on the lower skeletons, and $y\circ f'\simeq f$ implies that $y$ is of degree one as both $f$ and $f'$ are of degree one. It follows that $y$ induces an isomorphism on homology and hence is a homotopy equivalence by the Whitehead theorem. 
It is known that the minimal $CW$-model of a simply connected $CW$-complex is unique up to homotopy equivalence. Therefore, $h_Y\simeq f_0\circ h_X$ and Diagram (\ref{deg1=pushouteq0}) gives the required Diagram (\ref{deg1=pushouteq}).   
\end{proof}

A technical lemma is needed to prove the main result, Theorem \ref{inertdeg1thm}, of this section. 
\begin{lemma}\label{cone+pushoutlemma}
Suppose that there is a homotopy commutative diagram
\[
\diagram
A \rto^{a'} \ddouble & X' \rto^{f'} \dto^{\lambda_X} & Y'\dto^{\lambda_Y}\\
A \rto^{a}               & X \rto^{f}                              & Y.
\enddiagram
\]
If the right square a homotopy pushout and the top row is a homotopy cofibration, then the bottom row is a homotopy cofibration.
\end{lemma}
\begin{proof}
To prove the lemma, we may replace homotopy cofibration and homotopy pushout by strict cofibration and pushout, respectively. Consider the commutative diagram 
\[
\diagram
A \rto^{a'} \ddouble & X' \rto^{\widetilde{f}'} \dto^{\lambda_X} & \widetilde{Y}'\dto^{\lambda_{\widetilde{Y}}}\\
A \rto^{\widetilde{a}}               & X \rto^{\widetilde{f}}                              & \widetilde{Y},
\enddiagram
\]
where $\widetilde{Y}':=X'\cup_{a'} CA$ is the mapping cone of $a'$ with the structural map $\widetilde{f}'$, $\widetilde{a}:=\lambda_X\circ a'$, and $\widetilde{Y}$ is the pushout of $\widetilde{f}'$ and $\lambda_X$ with the structural maps $\widetilde{f}$ and $\lambda_{\widetilde{Y}}$. Since $\widetilde{f}'$ is a cofibration, the pushout $\widetilde{Y}$ is also a homotopy pushout. 
The assumption of the lemma implies that $\widetilde{Y}'\simeq Y'$, $\widetilde{a}\simeq a$, and $\widetilde{Y}\simeq Y$, and the commutative diagram is a rigorous model of the original homotopy commutative diagram. 
Further, the definition of pushout implies that $\widetilde{Y}=X\cup_{X'} (X'\cup_{a'} CA)=
X\cup_{\widetilde{a}} CA$ is the mapping cone of $\widetilde{a}$. Therefore, the sequence $A\stackrel{\widetilde{a}}{\longrightarrow} X\stackrel{\widetilde{f}}{\longrightarrow} Y$ is a cofibration, and then the sequence $A\stackrel{a}{\longrightarrow} X\stackrel{f}{\longrightarrow} Y$ is a homotopy cofibration.
\end{proof}

The following theorem illustrates the basic relations between degree one maps and the inertness of top cell attachments.

\begin{theorem}\label{inertdeg1thm}
Let $X$ and $Y$ be two $CW$-complexes such that $fd(X)=fd(Y)=n$ and both $X$ and $Y$ have a single $n$-cell. Let $h_X$ and $h_Y$ be the attaching maps of the top cells of $X$ and $Y$, respectively. 
Suppose that $f: X\stackrel{}{\longrightarrow} Y$ is a degree one map. Then the following hold:
\begin{itemize}
\item[(1).] suppose that the restriction $f_0: X_0\stackrel{}{\longrightarrow} Y_0$ is inert. Then $h_X$ is inert if and only if $h_Y$ is inert; 
\item[(2).] suppose that $\Omega f$ has a right homotopy inverse. Then if $h_X$ is inert so is $h_Y$;
\item[(3).] suppose that there is a homotopy cofibration $A\stackrel{a_0}{\longrightarrow} X_0 \stackrel{f_0}{\longrightarrow} Y_0$ for some space $A$ and some inert map $a_0$. Then if $h_Y$ is inert so is $h_X$.
\end{itemize}
\end{theorem}
\begin{proof}
By Lemma \ref{degk=pushoutlemma}, the left square of Diagram (\ref{degk=pushouteq}) is a homotopy pushout 
\[
\diagram 
     X_0\rto^-{f_0}\dto^-{i_X} & Y_0\dto^-{i_Y} \\ 
     X\rto^-{f} & Y. 
  \enddiagram
\]
By definition, the inertness of the attaching maps $h_X$ and $h_Y$ is equivalent to the existence of right homotopy inverses of $\Omega i_X$ and $\Omega i_Y$, respectively. Accordingly, Theorem \ref{pushoutthm} can be applied to the homotopy pushout to prove statement (1), while the homotopy commutativity of the square implies statement (2) immediately. 
For statement (3), the assumption implies that there is a homotopy commutative diagram
\[
\diagram 
      A \ddouble  \rto^-{a_0} & X_0\rto^-{f_0}\dto^-{i_X} & Y_0\dto^-{i_Y} \\ 
     A \rto^-{a}          & X\rto^-{f} & Y,
  \enddiagram
\]
with $a=:i_X\circ a_0$. Since the right square is a homotopy pushout and the top row is a homotopy cofibration, Lemma \ref{cone+pushoutlemma} implies that the bottom row is also a homotopy cofibration. Then Theorem \ref{pushoutthm2} can be applied to the diagram to show statement (3).
\end{proof}

For the general case when $f$ is not necessarily of degree one, we can show a local version of Theorem \ref{inertdeg1thm}.

\begin{theorem}\label{inertdegkthm}
Let $X$ and $Y$ be two $CW$-complexes in Theorem \ref{inertdeg1thm} with further assumption that they are nilpotent and at least one of them is simply connected. 
Suppose that $f: X\stackrel{}{\longrightarrow} Y$ is a degree $k$ map with $k\neq 0$. Then the three conclusions in Theorem \ref{inertdeg1thm} hold after localization away from all primes $p$ that divide $k$.
\end{theorem}
\begin{proof}
We work in the homotopy category away from the primes $p$ that divide $k$. 

Suppose that $Y$ is simply connected. 
Consider the diagram of homotopy cofibre sequences
\begin{equation}\label{1/kYdiag}
\diagram 
      S^{n-1}\rto^-{h_Y}\dto^{\frac{1}{k}} & Y_0\rto^-{i_Y}\ddouble & Y\dto^{\tau_k} \rto^{q_Y} & S^n \dto^{\frac{1}{k}} \\ 
      S^{n-1}\rto^-{h_Y'} & Y_0\rto^-{i_Y'} &   Y'  \rto^{q_Y'}  & S^n,
  \enddiagram
\end{equation}
where $h_Y'=k\cdot h_Y$ with the homotopy cofibre $Y'$, the maps $\frac{1}{k}$ are self-maps of degree $\frac{1}{k}$, and $\tau_k$ and $q_Y'$ are the induced maps. In particular, the map $\tau_k$ is of degree $1/k$. 
Further, the five lemma implies that the map $Y\stackrel{\tau_k}{\longrightarrow} Y'$ induces an isomorphism on local homology, and hence a local homotopy equivalence by the Whitehead theorem. In particular, $h_Y$ is inert if and only if $h_Y'$ is inert. 

Combining Diagram (\ref{degk=pushouteq}) and (\ref{1/kYdiag}) we have the homotopy cofibration diagram 
\[
\diagram 
      X_0\rto^-{i_X}\dto^{f_0} & X\dto^{\tau_k\circ f}  \rto^-{q_X}& S^{n}\ddouble &  \\ 
      Y_0\rto^-{i_Y'} &   Y' \rto^-{q_Y'}  &S^{n} ,
  \enddiagram
\]
or equivalently, the map $\tau_k\circ f: X\stackrel{}{\longrightarrow} Y'$ is of degree one and $(\tau_k\circ f)_0\simeq f_0$. Hence, the three conclusions of Theorem \ref{inertdeg1thm} hold for the map $\tau_k\circ f$. 

To summarize, we have showed that the three conclusions of Theorem \ref{inertdeg1thm} hold for the map $\tau_k\circ f$, $(\tau_k\circ f)_0\simeq f_0$, the map $\tau_k$ is a homotopy equivalence, and $h_Y$ is inert if and only if $h_Y'$ is inert. Then it is easy to check that the hypotheses in the three conclusions of Theorem \ref{inertdeg1thm} are equivalent for the maps $f$ and $\tau_k\circ f$. Therefore, the three conclusions of Theorem \ref{inertdeg1thm} also hold for the map $f$. This shows the theorem when $Y$ is simply connected.

Suppose that $X$ is simply connected. We can apply the above trick to $X$ and produce a degree one map $X'\stackrel{\simeq}{\longrightarrow} X\stackrel{f}{\longrightarrow} Y$. A similar argument then will prove the theorem in this case. 
\end{proof}

We can now prove Theorem \ref{inertdeg1thmintro}. 

\begin{proof}[Proof of Theorem \ref{inertdeg1thmintro}]
By assumption, the two Poincar\'{e} duality complexes $M$ and $N$ satisfy that $fd(M)=fd(N)$, both $M$ and $N$ have a single top cell, and $f: M\stackrel{}{\longrightarrow} N$ is a degree one or $k$ map depending on the two cases in the statement. Then the theorem follows from Theorems \ref{inertdeg1thm} and \ref{inertdegkthm} immediately. 
\end{proof}

Theorem \ref{inertdeg1thmintro} can be strengthened when the map $f$ is a covering map. In this case, $f$ has to be a finite cover, for an infinite cover over a closed manifold is an open manifold.  

We begin with some basics on loop spaces. 
For a space $X$ and a loop $\omega: S^1\stackrel{}{\longrightarrow} X$, denote by $\Omega_\omega X$ the path component of $\Omega X$ containing the loop $\omega$. When $\omega$ is null homotopic, denote $\Omega_0 X:=\Omega_\omega X$, which is the path component of $\Omega X$ containing the constant loop.
It is clear that any map $g: X\stackrel{}{\longrightarrow} Y$ induces a map 
\[
\Omega_\omega g: \Omega_\omega X\stackrel{}{\longrightarrow} \Omega_{g\circ \omega} Y,
\]
for any $\omega \in \Omega X$. Further, the loop product operation defines a map
\[
\omega: \Omega X\stackrel{}{\longrightarrow}\Omega X
\]
by $\omega (\omega')=\omega\ast \omega'$. Its restriction
\[
\omega: \Omega_{\omega'} X\stackrel{}{\longrightarrow}\Omega_{\omega\ast \omega'} X 
\]
on each path component $\Omega_{\omega'} X$ is a homotopy equivalence. 

\begin{lemma}\label{loop0lemma}
For a map $g: X\stackrel{}{\longrightarrow} Y$, 
the map $\Omega g: \Omega X\stackrel{}{\longrightarrow} \Omega Y$ has a right homotopy inverse if and only if $g_\ast: \pi_1(X)\stackrel{}{\longrightarrow} \pi_1(Y)$ is an epimorphism and 
$\Omega_0 g: \Omega_0 X\stackrel{}{\longrightarrow} \Omega_0 Y$ has a right homotopy inverse.  
\end{lemma}
\begin{proof}
The necessity is clear. For sufficiency, suppose that $\Omega_0 X\stackrel{\Omega_0 g}{\longrightarrow} \Omega_0 Y$ has a right homotopy inverse $\Omega_0 Y\stackrel{s_0}{\longrightarrow} \Omega_0 X$. 
Consider the homotopy commutative diagram
\[
\diagram
\Omega_0 Y \rto^{s_0} \dto^{\lambda} &  \Omega_0 X \rto^{\Omega_0 g} \dto^{\omega} & \Omega_0 Y \dto^{\lambda} \\
\Omega_\lambda Y \rto^{s_\lambda}  &  \Omega_\omega X \rto^{\Omega_\omega g} & \Omega_\lambda Y,
\enddiagram
\]
where $\lambda= g\circ \omega$ and $s_\lambda:=\omega \circ s_0\circ \lambda^{-1}$. The right square commutes automatically, while the left square homotopy commutes as $\lambda^{-1}$ is a homotopy inverse of $\lambda$. Since $\lambda$ is a homotopy equivalence and the top composition $\Omega_0 g\circ s_0$ is homotopic to the identity map, the diagram implies that the bottom composition $\Omega_\omega g\circ s_\lambda$ is homotopic to the identity map, that is, $\Omega_\omega g$ has a right homotopy inverse. Since by assumption $g_\ast: \pi_1(X)\stackrel{}{\longrightarrow} \pi_1(Y)$ is an epimorphism, for each path component $\Omega_\lambda Y$ there is a loop $\omega \in \Omega X$ such that $g_\ast([\omega])=[\lambda]$, and then $\Omega_{g\circ \omega} Y=\Omega_\lambda Y$. Combining the previous argument, we see that for each path component $\Omega_\lambda Y$ of $\Omega Y$, there exists a path component $\Omega_\omega X$ of $\Omega X$ such that the map $\Omega g$ restricts to $\Omega_\omega X\stackrel{\Omega_\omega g}{\longrightarrow}\Omega_\lambda Y$, and this restriction has a right homotopy inverse. It follows that the map $\Omega X\stackrel{\Omega g}{\longrightarrow}\Omega Y$ has a right homotopy inverse. 
\end{proof}

Let $M$ and $N$ be two connected $n$-dimensional Poincar\'{e} duality complexes with a single top cell. There are homotopy cofibrations
\[
S^{n-1}\stackrel{h_M}{\longrightarrow} M_0\stackrel{i_M}{\longrightarrow} M, \ \ 
S^{n-1}\stackrel{h_N}{\longrightarrow} N_0\stackrel{i_N}{\longrightarrow} N,
\]
where $M_0$ and $N_0$ are the $(n-1)$-skeletons of $M$ and $N$ with the indicated structural maps, respectively. In the following we study the inertness property around a finite covering map $f: M\stackrel{}{\longrightarrow}N$. Note that $f$ is clearly not cellular, while the existence of the finite cover $f$ implies that the cellular finite cover of the $CW$-complex $N$ is homotopy equivalent to a $CW$-complex $M$ with a single top cell. 

\begin{lemma}\label{coverinertlemma1}
Let $f: M\stackrel{}{\longrightarrow} N$ be a finite cover of connected $n$-dimensional Poincar\'{e} duality complexes with a single top cell and $n\geq 2$. 
If the attachment map for the top cell of $M$ is inert, then the attachment map for the top cell of $N$ is inert.
\end{lemma}
\begin{proof}
By the $CW$-approximation there is the homotopy commutative square 
\[
\diagram
M_0 \rto^{f_0} \dto^{i_M}  & N_0 \dto^{i_N}\\
M\rto^{f}                       &  N,
\enddiagram
\]
where $f_0$ is the restriction of $f$ on the lower skeletons. Applying the loop functor to the square and then   restricting to the path components containing the constant loop, we obtain the homotopy commutative square  
\[
\diagram
\Omega_0 M_0 \rto^{\Omega_0 f_0} \dto^{\Omega_0 i_M}  & \Omega_0 N_0 \dto^{\Omega_0 i_N}\\
\Omega_0 M\rto^{\Omega_0 f}                       &  \Omega_0 N.
\enddiagram
\]
By Lemma \ref{loop0lemma}, the condition that the attachment map for the top cell of $M$ is inert implies that $\Omega_0 i_M$ has a right homotopy inverse. 
Since $f$ is a cover, it is clear that $\Omega_0 f$ is a homotopy equivalence, and then the diagram implies that $\Omega_0 i_N$ has a right homotopy inverse. Also, since the dimension $n\geq 2$, the lower skeleton inclusion $i_N: N_0\stackrel{}{\longrightarrow} N$ induces an epimorphism on the fundamental groups. Therefore, Lemma \ref{loop0lemma} implies that $\Omega i_N$ has a right homotopy inverse, that is, the attachment map for the top cell of $N$ is inert.
\end{proof}

\begin{lemma}\label{coverinertlemma2}
Let $f: M\stackrel{}{\longrightarrow} N$ be a finite cover of connected $n$-dimensional Poincar\'{e} duality complexes with a single top cell. 
If the attachment map for the top cell of $N$ is inert, then the attachment map for the top cell of $M$ is inert.
\end{lemma}
\begin{proof}
Consider the pullback diagram 
\[\xymatrix{ 
M|_0\rto^-{f|_0}\dto^-{i|_0} & N_0\dto^-{i_N} \\ 
      M\rto^-{f} & N
}\]
defining the space $M|_0$ with the structural maps $f|_0$ and $i|_0$. Since $f$ is a finite cover, so is $f|_0$ and the pullback square is also a homotopy pullback. By assumption, the attachment map for the top cell of $N$ is inert, that is, the map $\Omega i_N$ has a right homotopy inverse. Then as the loop functor preserves homotopy pullback, Theorem \ref{pullbackthm} (1) implies that $\Omega i|_0$ has a right homotopy inverse. 
Furthermore, since the lower skeleton $N_0$ is homotopy equivalent to a $CW$-complex of dimension strictly less than $n$, so is its finite cover $M|_0$, and then by the $CW$-approximation the map $M|_0\stackrel{i|_0}{\longrightarrow} M$ factors as 
\[
M|_0\stackrel{\mathfrak{i}|_0}{\longrightarrow} M_0\stackrel{i_M}{\longrightarrow} M
\]
for some map $\mathfrak{i}|_0$. As we have shown that $\Omega i|_0$ has a right homotopy inverse, so does the map $\Omega i_M$. This means that the attachment map for the top cell of $M$ is inert.
\end{proof}

Combining Lemma \ref{coverinertlemma1} and \ref{coverinertlemma2}, we obtain the following proposition immediately.
\begin{proposition}\label{coverprop}
Let $f: M\stackrel{}{\longrightarrow} N$ be a finite cover of connected $n$-dimensional Poincar\'{e} duality complexes with a single top cell and $n\geq 2$. 
Then the attachment map for the top cell of $M$ is inert if and only if the attachment map for the top cell of $N$ is inert.       $\qqed$
\end{proposition}

\newpage
%%%%%%%%%%%%%%%%%%%%%%%%%%%%%%%%%%%%%%%%%%%%%%%%%%%%%%%%%%%%%%%%%%%%%%

%\part{Understand inertness from domain}
%\label{part2}

%----------------------------------------------------------------------------------------------------------------------------------------------------------------------------------------------------------%
\section{Algebraic and geometric intersections}
\label{sec: int}
In this section, we study the inertness of the top cell attachment of a Poincar\'{e} duality complex by intersection theory. 

The original idea is geometric. Let $M$ be an oriented closed smooth manifold. Let $A$ and $B$ be two closed submanifolds of $M$ with complementary dimensions. In differential topology, there is a classical notion of the intersection number of $A$ and $B$ in $M$. If the intersection number is nonzero, it is possible to construct a nonzero degree map $f: A\times B\stackrel{}{\longrightarrow} M$. Then Theorem \ref{inertdeg1thmintro} (1) can be applied to study the inertness of the top cell attachment of $M$. 

It turns out that this idea can be generalized by an algebraic version of the intersection number for general Poincar\'{e} duality complexes, based on the classical fact that intersection number can be characterized by cup product and Poincar\'{e} duality. Accordingly, we first introduce a notion of algebraic intersection number for Poincar\'{e} duality complexes, and then show the equality of intersection number and mapping degree under mild conditions. Next, we work in the context of algebraic intersection and prove Theorem \ref{pdtintthm}, which addresses the inertness of the top cell attachment of a Poincaré duality complex. We then specify the result to various cases, including Poincar\'{e} duality complexes with homotopy actions, topological homogeneous spaces and smooth intersections of submanifolds. Finally, we discuss Steenrod's realization problem of cohomology classes by submanifolds, as it is closely related. The last two subsections, which are in the smooth context, explain the geometric ideas for this section.  
%----------------------------------------------------------------------------------
\subsection{Algebraic intersection}
\label{subsec: alg-int}

$\, $

Let $M$ be an $n$-dimensional connected Poincar\'{e} duality complex. Denote by $[-]$ the {\it fundamental class} of a Poincar\'{e} duality complex. Recall that the classical Poincar\'{e} duality is an isomorphism
\[
[M]\cap (-): H^{i}(M;\mathbb{Z})\stackrel{}{\longrightarrow} H_{n-i}(M;\mathbb{Z})
\]
for any $i\in \mathbb{Z}$, where $\cap$ is the cap product. For any $w\in  H_{n-i}(M;\mathbb{Z})$, denote its unique preimage by $w^\ast$, called the {\it Poincar\'{e} dual} of $w$. In particular, $[M]\cap w^\ast=w$.

Let $A$ and $B$ be two connected Poincar\'{e} duality complexes of dimension $m$ and $n-m$, respectively, and $n\geq m$. Denote by 
\[
\langle- ,- \rangle: H^\ast(M;\mathbb{Z})\times H_\ast(M;\mathbb{Z})\longrightarrow \mathbb{Z}
\]
the {\it Kronecker pairing} between cohomology and homology. The following definition is an algebraic version of the corresponding geometric concept which will be explained in Subsection \ref{subsec: geo-int}. 
\begin{definition}\label{alg-intdef}
Let $A\stackrel{f_A}{\longrightarrow} M$ and $B\stackrel{f_B}{\longrightarrow} M$ be two maps. 
The integer
\begin{equation}\label{int=1eq}
A \testcap B:= \langle (f_{A\ast}([A]))^\ast\cup (f_{B\ast}([B]))^\ast, [M]\rangle \in \mathbb{Z}
\end{equation}
is called the {\it algebraic intersection number} of $f_A$ and $f_B$, or simply the {\it intersection number} of $A$ and $B$ in $M$. 

Further, 
if $A \testcap B\neq 0$, we say that $A$ and $B$ {\it essentially intersect in $M$}; 
if $A \testcap B=1$ we say that $A$ and $B$ are {\it dual to each other in $M$}.
\end{definition}

For convenience, the Poincar\'{e} duals $(f_{A\ast}([A]))^\ast$ and $(f_{B\ast}([B]))^\ast$ are usually referred to as the {\it Poincar\'{e} duals} of $A$ and $B$ in $M$. 
\begin{remark}\label{ordxrmk}
Note that if $A$ and $B$ essentially intersect in $M$, the Poincar\'{e} duals of $A$ and $B$ are of infinite order. 
\end{remark}

%----------------------------------------------------------------------------------
\subsection{Intersection number vs degree}
\label{subsec: int-deg} 

$\, $

Let $M$ be an $n$-dimensional connected Poincar\'{e} duality complex. Let $A$ and $B$ be two connected Poincar\'{e} duality complexes of dimension $m$ and $n-m$, respectively, and $n\geq m$. 
Suppose that there is a map 
\begin{equation}\label{fABMeq}
f: A\times B\stackrel{}{\longrightarrow} M.
\end{equation}
Denote by 
\[
f_A: A\stackrel{i_1}{\longrightarrow} A\times B\stackrel{f}{\longrightarrow} M ,\ \ \ 
f_B: B\stackrel{i_2}{\longrightarrow} A\times B\stackrel{f}{\longrightarrow} M 
\]
the restrictions of $f$ to the two factors of $A\times B$, where $i_1$ and $i_2$ are the inclusions. On the one hand, from Definition \ref{alg-intdef} there is the intersection number  $A\testcap B$ of $A$ and $B$ in $M$ defined through the two maps $f_A$ and $f_B$. On the other hand, as $A\times B$ and $M$ are Poincar\'{e} duality complexes of the same dimension there is the degree ${\rm deg}(f)$ of the map $f$. We want to show that they are equal under a primitive condition.  

\begin{definition}\label{primdef}
A cohomology class $z\in H^\ast(M;\mathbb{Z})$ is {\it $f$-primitive} if its pullback $f^\ast(z)$ satisfies 
\[
f^\ast(z)=f_A^\ast(z)\otimes 1 + 1\otimes f_B^\ast(z)\in H^\ast(A;\mathbb{Z})\otimes H^\ast(B;\mathbb{Z})
\]  
where $H^\ast(A;\mathbb{Z})\otimes H^\ast(B;\mathbb{Z})$ has been identified as a direct summand of $H^\ast(A\times B;\mathbb{Z})$ by the K\"{u}nneth formula. 
\end{definition}

\begin{remark}\label{prim-rmk}
The notion of primitivity in Definition (\ref{primdef}) can de defined for any map $f:X\times Y\stackrel{}{\longrightarrow} Z$ without further restriction. Further, when $X=Y=Z$ is an $H$-space and $f$ is the multiplication map of $Z$, the notion of $f$-primitivity is the usual one for the primitive elements of the Hopf algebra $H^\ast(Z;\mathbb{Z})$.
\end{remark}

For the map $f$ (\ref{fABMeq}) with its restrictions $f_A$ and $f_B$, denote by 
\[
x:=(f_\ast([A]))^\ast=(f_{A\ast}([A]))^\ast\in H^{n-m}(M;\mathbb{Z}), \ \ \ \ y:=(f_\ast([B]))^\ast=(f_{B\ast}([B]))^\ast\in H^{m}(M;\mathbb{Z})
\]
the Poincare duals of $A$ and $B$ in $M$. 
 
\begin{lemma}\label{int=deglemma}
Suppose that $A$ is $k$-connected with its dimension $m\geq k+1$. Suppose that $A$ and $B$ essentially intersect in $M$, and the self-intersection number of $A$ or $B$ is zero, that is, $A \testcap A=0$ or $B \testcap B=0$. 

If $x$ is $f$-primitive and $n>2m-k-1$, then up to sign the degree of the map $f$ is equal to the intersection number of $A$ and $B$ in $M$
\[
{\rm deg}(f)=\pm A\testcap B.
\]
\end{lemma}
\begin{proof}
Write $x=x_{n-m}$, $y=y_{m}$, $A=A^{m}$ and $B=B^{n-m}$ to indicate their dimensions. 
The arithmetic conditions imply that
\[
m\geq k+1\geq 1, \ \ \  n>2m-k-1\geq m.
\] 
By the assumptions that $x_{n-m}$ is $f$-primitive and $A$ is $k$-connected, we obtain 
\[\begin{split}
f^\ast(y_{m})&=f_A^\ast(y_{m})\otimes 1+ 1\otimes f_B^\ast(y_{m})+\sum\limits_{l}a_l\otimes b_l,\\
f^\ast(x_{n-m})&=f_A^\ast(x_{n-m})\otimes 1+1\otimes f_B^\ast(x_{n-m}),
\end{split}\]
for some $a_l\in H^{\geq k+1}(A^{m};\mathbb{Z})$ and $b_l\in H^{+}(B^{n-m};\mathbb{Z})$, where possible torsion part of $f^\ast(y_{m})$ can be omitted since we will take Kronecker pairing with fundamental class eventually. 
Then we have 
\[
\begin{split}
f^\ast(y_{m})\cup f^\ast(x_{n-m}) 
&= (f_A^\ast(y_{m}\cup x_{n-m})\otimes 1 ) + (f_A^\ast(y_{m})\otimes f_B^\ast(x_{n-m}))\\
&\ \ \ \pm (f_A^\ast(x_{n-m})\otimes f_B^\ast(y_{m}))+ (1\otimes f_B^\ast(y_{m}\cup x_{n-m}))\\
&\ \ \ \pm \sum\limits_{l}(f_A^\ast(x_{n-m})\cup a_l)\otimes b_l +\sum\limits_{l}a_l\otimes (b_l\cup f_B^\ast(x_{n-m}))\\ 
&=  (f_A^\ast(y_{m})\otimes f_B^\ast(x_{n-m})) \pm (f_A^\ast(x_{n-m})\otimes f_B^\ast(y_{m})),\\
\end{split}
\]
where the last equality holds by degree reason. 
Denote $A\testcap B=k$. It follows that 
\begin{equation}\label{kdegf1eq}
\begin{split}
k{\rm deg}(f)&= {\rm deg}(f) \langle x_{n-m}\cup y_{m}, [M]\rangle\\
 &=\pm\langle y_{m}\cup x_{n-m}, f_\ast ([A^{m}\times B^{n-m}])\\
&=\pm\langle f^\ast(y_{m})\cup f^\ast(x_{n-m}), [A^{m}\times B^{n-m}]\rangle \\
&=\pm \langle f_A^\ast(y_{m})\otimes f_B^\ast(x_{n-m}), [A^{m}\times B^{n-m}]\rangle \\
&\ \ \ \pm \langle f_A^\ast(x_{n-m})\otimes f_B^\ast(y_{m}), [A^{m}\times B^{n-m}]\rangle.
\end{split}
\end{equation}
In general, for any $z$, $z'\in H^\ast(M;\mathbb{Z})$ 
\[
\begin{split}
\langle f_A^\ast(z)\otimes f_B^\ast(z'), [A^m\times B^{n-m}]\rangle 
&=\pm\langle f_A^\ast(z), [A^{m}]\rangle \cdot \langle f_B^\ast(z'), [B^{n-m}]\rangle\\
&=\pm\langle z, f_{A\ast}([A^{m}])\rangle \cdot \langle z', f_{B\ast}([B^{n-m}])\rangle\\
&=\pm\langle z, [M]\cap x_{n-m}\rangle \cdot  \langle z', [M]\cap y_{m}\rangle \\
&=\pm\langle x_{n-m}\cup z, [M]\rangle \cdot \langle y_{m}\cup z', [M]\rangle.
\end{split}
\]
In particular, 
\[
\begin{split}
\langle f_A^\ast(y_{m})\otimes f_B^\ast(x_{n-m}), [A^{m}\times B^{n-m}]\rangle 
&=\pm\langle x_{n-m}\cup y_m, [M]\rangle \cdot \langle y_{m}\cup x_{n-m}, [M]\rangle=\pm k^2, \\
\langle f_A^\ast(x_{n-m})\otimes f_B^\ast(y_{m}), [A^{m}\times B^{n-m}]\rangle 
&= \pm \langle x_{n-m}^2, [M]\rangle \cdot \langle y_{m}^2, [M]\rangle =0,
\end{split}
\]
where the last equality follows from the assumption $A\testcap A=0$ or $B\testcap B=0$.
Combining these equalities with (\ref{kdegf1eq}) we have $k{\rm deg}(f)=\pm k^2$. 
Since $A$ and $B$ essentially intersect, that is, $k=A\testcap B\neq 0$, it follows that ${\rm deg}(f)=\pm k$.
\end{proof}

\begin{remark}\label{int=degrmk}
In Lemma \ref{int=deglemma}, if $n\neq 2m$, the condition that $A\testcap A=0$ or $B\testcap B=0$ is automatically satisfied by degree reason. 
\end{remark}
%----------------------------------------------------------------------------------
\subsection{Inertness via intersection number}
\label{subsec: iner-int} 

$\, $

Let $f: A\times B\stackrel{}{\longrightarrow} M$ be a map between simply connected Poincar\'{e} duality complexes of the same dimension. By the $CW$-approximation its restriction on the lower skeletons gives a map $f_0: (A\times B)_0\stackrel{}{\longrightarrow} M_0$. 
Recall that $A$ and $B$ essentially intersect in $M$ if $A\testcap B\neq 0$ through the restriction maps $A\stackrel{f_A}{\longrightarrow} M$ and $B\stackrel{f_B}{\longrightarrow} M$ of $f$; similarly, the self-intersection number of $A$ or $B$ is zero if $A\testcap A=0$ or  $B\testcap B=0$, and it holds when ${\rm dim}(A)\neq {\rm dim}(B)$. 

The following theorem indicates a situation in which the inertness of the top cell attachment of $M$ can be deduced from the map $f$ with additional assumptions. The context of the theorem is motivated from the context of transversal intersections of smooth submanifolds, which will be explained in Subsection \ref{subsec: geo-int}. Denote by ${\rm conn}(X)$ the {\it connectivity} of a $CW$-complex $X$. 
  
\begin{theorem}\label{pdtintthm}
Let $f: A\times B\stackrel{}{\longrightarrow} M$ be a map between simply connected Poincar\'{e} duality complexes. Suppose that $A$ and $B$ essentially intersect in $M$, and the self-intersection number of $A$ or $B$ is zero. Suppose that the Poincar\'{e} dual of $A$ in $M$ is $f$-primitive and ${\rm dim}(M)>2{\rm dim}(A)-{\rm conn}(A)-1>0$. 

If the restriction map 
$
f_0: (A\times B)_0\stackrel{}{\longrightarrow} M_0
$ 
is inert, then the attaching map for the top cell of $M$ is inert after localization away from all primes $p$ that divide $A\testcap B$.

In particular, if $A\testcap B=\pm 1$ then the attaching map for the top cell of $M$ is inert. 
\end{theorem}
\begin{proof} 
As $A$ is simply connected, $2{\rm dim}(A)>{\rm conn}(A)+1\geq 2$. Then the Poincar\'{e} duality complex $A$ is non-contractible and ${\rm dim}(A)\geq {\rm conn}(A)+1$. Also, ${\rm dim}(B)={\rm dim}(M)-{\rm dim}(A)>{\rm dim}(A)-{\rm conn}(A)-1\geq 0$, and then the Poincar\'{e} duality complex $B$ is non-contractible.

Since both $A$ and $B$ are non-contractible, Lemma \ref{pdt-inert-lemma} below implies that  the attaching map for the top cell of $A\times B$ is inert. Together with the assumption that $f_0$ is inert, Theorem \ref{inertdeg1thmintro} (1) with its local counterpart implies that the attaching map for the top cell of $M$ is inert after localization away from all primes $p$ that divide ${\rm deg}(f)$ if ${\rm deg}(f)\neq 0$. However, Lemma \ref{int=deglemma} implies that ${\rm deg}(f)=\pm A\testcap B\neq 0$, and then the theorem follows .
\end{proof}

\begin{remark}
Theorem \ref{pdtintthm} excludes the case when $A$ or $B$ is contractible. 
Without loss of generality let us assume that $B$ is contractible. In this case, suppose that $A$ and $B$ essentially intersect in $M$, and the attaching map for the top cell of $A$ is inert. If the restriction map $f_0: (A\times B)_0\stackrel{}{\longrightarrow} M_0$ is inert, then the attaching map for the top cell of $M$ is inert after localization away from all primes $p$ that divide $A\testcap B$.

Indeed, in this case the Poincar\'{e} duality complex $B$ is a point, and the map $f$ becomes $A\stackrel{f}{\longrightarrow} M$. Since a Poincar\'{e} dual of a point is the top generator of $H^\ast(M;\mathbb{Z})$, the condition $A$ and $B=\ast$ essentially intersect implies that the Poincar\'{e} dual of $f_\ast([A])$ is the nonzero integer $k=A\testcap B\in \mathbb{Z}\cong H^0(M;\mathbb{Z})$. By Poincar\'{e} duality it follows that $f_\ast([A])=k[M]$. Hence, ${\rm deg}(f)= k=A\testcap B\neq 0$. Combining this with the assumptions that $f_0$ is inert and the attaching map for the top cell of $A$ is inert, Theorem \ref{inertdeg1thmintro} (1) with its local counterpart implies that the attaching map for the top cell of $M$ is inert after localization away from all primes $p$ that divide $A\testcap B$. 

From this viewpoint, we see that the necessity part of Theorem \ref{inertdeg1thmintro} (1) is a degenerate case of Theorem \ref{pdtintthm}.
\end{remark}

\begin{remark}
We call a map $h: X\stackrel{}{\longrightarrow} Z$ a {\it wedge summand inclusion} if there exists a homotopy equivalence $Z\simeq X\vee Y$ for some complex $Y$ such that $h$ is homotopic to the composite $X\stackrel{}{\hookrightarrow} X\vee Y\stackrel{\simeq}{\longrightarrow}Z$. In this case, there is a homotopy cofibration
\[
X\stackrel{h}{\longrightarrow}Z\stackrel{q}{\longrightarrow} Y,
\]
where the projection $q$ has a right homotopy inverse. Accordingly, the map $h$ is inert. 

In particular, if the restriction map $
f_0: (A\times B)_0\stackrel{}{\longrightarrow} M_0
$ in Theorem \ref{pdtintthm} 
is a wedge summand inclusion, then it is inert. 
\end{remark}

\begin{lemma}\label{pdt-inert-lemma}
Let $A$ and $B$ be two non-contractible connected $CW$-complexes with a single top cell. Then the  attaching map for the top cell of $A\times B$ is inert.
\end{lemma}
\begin{proof}
For the product space $A\times B$, there is the homotopy commutative diagram
\[
\xymatrix{ 
 A_0\times B_0   \ar[r]^{1_{A_0}\times i_B} \ar[d]^{i_A\times 1_{B_0}} & A_0\times B \ar[d]^{}  \ar@/^1.20pc/[ddr]^{i_A\times 1_B} \\
 A\times B_0 \ar[r]^{}  \ar@/_1.20pc/[drr]_{1_A\times i_B}  &    (A\times B)_0 \ar[dr]^(0.45){i_{(A\times B)}} \\
&&A\times B,
}\]
where the inner square is a homotopy pushout and the map $X_0\stackrel{i_X}{\longrightarrow} X$ denotes the lower skeleton inclusion for $X=A$, $B$ or $A\times B$. 
Note that the composites
\[
\begin{split}
&A\stackrel{i_1}{\longrightarrow} A\times B_0\stackrel{}{\longrightarrow} (A\times B)_0\stackrel{i_{(A\times B)}}{\longrightarrow} A\times B  \\
&B\stackrel{i_2}{\longrightarrow} A_0\times B\stackrel{}{\longrightarrow} (A\times B)_0\stackrel{i_{(A\times B)}}{\longrightarrow} A\times B  
\end{split}
\]
are homotopic to the inclusions of the two factors $A\stackrel{i_1}{\longrightarrow} A\times B$ and $B\stackrel{i_2}{\longrightarrow} A\times B$, respectively. Consider the homotopy commutative diagram
\[
\diagram
\Omega A \times \Omega B  \rto^<<<<{\Omega i_1\times \Omega i_2} &
 \Omega (A\times B_0)\times \Omega  (A_0\times B ) \rto  \drto_<<<<<{\Omega (1_A\times i_B)\times \Omega (i_A\times 1_B)} & 
 \Omega (A\times B)_0\times \Omega (A\times B)_0\rto^<<<{\mu} \dto^{\Omega i_{(A\times B)}\times \Omega i_{(A\times B)}}  & 
 \Omega (A\times B)_0 \dto^{\Omega i_{(A\times B)}}  \\
  && \Omega (A\times B)\times \Omega (A\times B)  \rto^<<<<{\mu} & \Omega( A\times B  ),
\enddiagram
\]
where the left triangle homotopy commutes as it is the product of the two triangles in the previous diagram, the maps $\mu$ are the loop multiplications and the right square homotopy commutes by the naturality of loop map. Since the lower direction composition around the diagram restricts to $\Omega A\stackrel{i_1}{\longrightarrow}\Omega A\times \Omega  B$ and $\Omega B\stackrel{i_2}{\longrightarrow} \Omega A\times \Omega B$ on the two factors, respectively, it is a weak equivalence and then is a homotopy equivalence by the Whitehead theorem. Accordingly, the upper direction composition around the diagram implies that $\Omega i_{(A\times B)}$ has a right homotopy inverse, that is, the  attaching map for the top cell of $A\times B$ is inert.
\end{proof}

Theorem \ref{pdtintthmintro} is a special case of Theorem \ref{pdtintthm}. 

\begin{proof}[Proof of Theorem \ref{pdtintthmintro}]
From the condition ${\rm dim}(B)>{\rm dim}(A)>0$, we see that the Poincar\'{e} duality complex $A$ is non-contractible with ${\rm dim}(A)\geq {\rm conn}(A)+1$, and 
\[
{\rm dim}(M)={\rm dim}(A)+{\rm dim}(B)>2{\rm dim}(A)>2{\rm dim}(A)-{\rm conn}(A)-1>0.
\] 
Also, since ${\rm dim}(M)\neq 2{\rm dim}(A)$ or $2{\rm dim}(B)$, the self-intersection number $A\testcap A=B\testcap B=0$ by degree reason. 

Additionally, by assumption the restriction map $f_0: (A\times B)_0\longrightarrow M_0$ is inert, $A\testcap B\neq 0$ and the Poincar\'{e} dual $[A]^\ast$ is $f$-primitive. Therefore, all the conditions of Theorem \ref{pdtintthm} are satisfied, and the theorem follows immediately. 
\end{proof}
%----------------------------------------------------------------------------------
\subsection{Poincar\'{e} duality complexes with homotopy actions}
\label{subsec: GPoin}

$\, $

If $M$ has a homotopy action from an $H$-complex, it is possible to define a nonzero degree map from $A\times B$ to $M$ when $A$ and $B$ essentially intersect. Recall an {\it $H$-complex} $G$ is a $CW$-complex with a multiplication and a unit up to homotopy, and a {\it $G$-complex} $Z$ is a $CW$-complex equipped with a {\it homotopy action} 
\[G\times Z\stackrel{\mu}{\longrightarrow} Z.\] In particular, there is a restriction map 
\begin{equation}\label{muleq}
\mu_l:=\mu(-,\ast): G\stackrel{}{\longrightarrow} Z
\end{equation}
by acting on the basepoint of $Z$. For the $G$-complex $Z$, a cohomology class $z\in H^\ast(Z;\mathbb{Z})$ is called {\it primitive} if 
\[
\mu^\ast(z)=  \mu_l^\ast(z)\otimes 1+1\otimes z\in H^\ast(G;\mathbb{Z})\otimes H^\ast(Z;\mathbb{Z}),
\]
where $H^\ast(G;\mathbb{Z})\otimes H^\ast(Z;\mathbb{Z})$ has been identified as a direct summand of $H^\ast(G\times Z;\mathbb{Z})$ by the K\"{u}nneth formula. Note that in the sense of Definition \ref{primdef}, this is equivalent to that the class $z$ is $\mu$-primitive, and by Remark \ref{prim-rmk} when $Z$ itself is an $H$-complex, the notion of primitivity is the usual one for the primitive elements of the Hopf algebra $H^\ast(Z;\mathbb{Z})$. 

Suppose $\mu: G\times M\stackrel{}{\longrightarrow} M$ is a homotopy action of $G$ on $M$. 
Let $f_A: A\stackrel{\mathfrak{j}_A}{\longrightarrow}G\stackrel{\mu_l}{\longrightarrow} M$ and $f_B: B\longrightarrow M$ be two maps. They determine a composite
\begin{equation}\label{intABfeq}
f: A\times B \stackrel{\mathfrak{j}_A\times f_B}{\longrightarrow} G\times M\stackrel{\mu}{\longrightarrow} M.
\end{equation}
It is clear that if a cohomology class $z\in H^\ast(M;\mathbb{Z})$ is primitive, then it is $f$-primitive. 
Then Theorem \ref{pdtintthm} can be applied to the map $f$ to prove the following theorem.  Its smooth version will be stated in Theorem \ref{smintthm}. 
  
\begin{theorem}\label{intthm}
Let $G$ be an $H$-complex. 
Let $M$ be a simply connected Poincar\'{e} duality $G$-complex. 
Let $A$ and $B$ be two simply connected Poincar\'{e} duality complexes with complementary  positive dimensions which essentially intersect in $M$ through two maps $f_A: A\stackrel{\mathfrak{j}_A}{\longrightarrow} G\stackrel{\mu_l}{\longrightarrow}M$ and $f_B: B\stackrel{}{\longrightarrow} M$. 
Suppose either the Poincar\'{e} dual of $A$ in $M$ is primitive and ${\rm dim}(B)>{\rm dim}(A)$ or the Poincar\'{e} dual of $B$ in $M$ is primitive and ${\rm dim}(A)>{\rm dim}(B)$. 

If the restriction of the product map $f$ (\ref{intABfeq}) of $f_A$ and $f_B$ on the lower skeletons 
$
f_0: (A\times B)_0\stackrel{}{\longrightarrow} M_0
$ 
is inert, 
then the attaching map for the top cell of $M$ is inert after localization away from all primes $p$ that divide $A\testcap B$.

In particular, if $A\testcap B=\pm 1$ then the attaching map for the top cell of $M$ is inert. 
\end{theorem}
\begin{proof}
The proof is similar to that of Theorem \ref{pdtintthmintro}.
\end{proof}

When $M$ itself is an $H$-complex, any two maps $f_A: A\stackrel{}{\longrightarrow} M$ and $f_B: B\longrightarrow M$ can be multiplied to produce a map
\begin{equation}\label{MHintABfeq}
f: A\times B \stackrel{f_A\times f_B}{\longrightarrow} M\times M\stackrel{\mu}{\longrightarrow} M,
\end{equation}
where $\mu$ is the multiplication of $M$. 
In this case, Theorem \ref{intthm} reduces to the following theorem. 
\begin{theorem}\label{MHintthm}
Let $M$ be a simply connected Poincar\'{e} duality $H$-complex. 
Let $A$ and $B$ be two simply connected Poincar\'{e} duality complexes with complementary dimensions which essentially intersect in $M$ through two maps $A\stackrel{f_A}{\longrightarrow}M$ and $B\stackrel{f_B}{\longrightarrow} M$. Suppose that the Poincar\'{e} dual of $A$ in $M$ is primitive and ${\rm dim}(B)>{\rm dim}(A)>0$.

If the restriction of the product map $f$ (\ref{MHintABfeq}) of $f_A$ and $f_B$ on the lower skeletons 
$
f_0: (A\times B)_0\stackrel{}{\longrightarrow} M_0
$ 
is inert, 
then the attaching map for the top cell of $M$ is inert after localization away from all primes $p$ that divide $A\testcap B$.

In particular, if $A\testcap B=\pm 1$ then the attaching map for the top cell of $M$ is inert.  ~$\qqed$
\end{theorem}

%----------------------------------------------------------------------------------
\subsection{Topological homogeneous spaces}
\label{subsec: top-hom}

$\, $

Let $G$ be a connected topological group. Let $K$ be a connected topological subgroup of $G$. Denote by $K\lvertneqq G$ if $K$ is proper. The orbit space $G/K$ is called a {\it topological homogeneous space}, or simply {\it homogeneous space}. There is the canonical action of $G$ on $G/K$
\[
\mu: G\times G/K\stackrel{}{\longrightarrow} G/K
\]
defined by the multiplication of $G$. 
Note that $G$ has the identity element $e$ as the basepoint and then $G/K$ has $eK$ as the basepoint. It follows that the restriction map $\mu_l$ (\ref{muleq}) of $\mu$ is the standard quotient 
\[
\pi_K: G\stackrel{}{\longrightarrow} G/K.
\]
Let $K\lvertneqq H\lvertneqq G$ be a sequence of topological groups. There is the canonical fibration
\[\label{KHGeq}
H/K\stackrel{j_{KH}}{\hookrightarrow} G/K\stackrel{\pi_{KH}}{\longrightarrow} G/H.
\]

\begin{theorem}\label{tophomthm}
Let $K\lvertneqq H\lvertneqq G$ be a sequence of nontrivial connected topological groups such that $G$ and $H$ are simply connected. 
Suppose that there is a simply connected Poincar\'{e} duality complex $N$ with a map $N\stackrel{f_N}{\longrightarrow} G\stackrel{\pi_K}{\longrightarrow}G/K$ such that $N$ and $H/K$ essentially intersect in $G/K$. 
Suppose either the Poincar\'{e} dual of $N$ in $G/K$ is primitive and ${\rm dim}(H/K)>{\rm dim}(N)$ or the Poincar\'{e} dual of $H/K$ in $G/K$ is primitive and ${\rm dim}(N)>{\rm dim}(H/K)$.

If the restriction of the product map $f: N\times H/K\stackrel{f_N\times j_{KH}}{\longrightarrow} G\times G/K\stackrel{\mu}{\longrightarrow} G/K$ on the lower skeletons 
$
f_0: (N\times H/K)_0\stackrel{}{\longrightarrow} (G/K)_0
$ 
is inert, 
then the attaching map for the top cell of $G/K$ is inert after localization away from all primes $p$ that divide $N\testcap (H/K)$.

In particular, if $N\testcap (H/K)=\pm 1$ then the attaching map for the top cell of $G/K$ is inert. 
\end{theorem}
\begin{proof}
Since $N$ and $H/K$ essentially intersect in $G/K$, ${\rm dim}(N)={\rm dim}(G/K)-{\rm dim}(H/K)={\rm dim}(G/H)$. Since $K\lvertneqq H\lvertneqq G$, the dimension ${\rm dim}(H/K)>0$ and ${\rm dim}(N)={\rm dim}(G/H)>0$. 
Further, consider the canonical fibrations 
\[
K\stackrel{}{\longrightarrow} G\stackrel{}{\longrightarrow} G/K, \ \ \ 
K\stackrel{}{\longrightarrow} H\stackrel{}{\longrightarrow} H/K.
\]
Since $G$ and $H$ are simply connected and $K$ is connected, the homogeneous spaces $G/K$ and $H/K$ are simply connected by the long exact sequences of the homotopy groups of a fibration. 
Therefore, the conditions of Theorem \ref{intthm} are satisfied for $M=G/K$, $A=N$ and $B=H/K$. Then the theorem follows immediately. 
\end{proof}

\begin{remark}\label{tophom-rmk}
In favorable cases, it is possible to choose $N$ to be $G/H$. For instance, see Theorem \ref{VCHthm} and its proof for Stiefel manifolds. 
\end{remark}
%----------------------------------------------------------------------------------
\subsection{Smooth intersection}
\label{subsec: geo-int}

$\, $

In this subsection, we apply the results of Subsections \ref{subsec: alg-int} and \ref{subsec: GPoin} to the smooth case. The material in this and the next subsection partly explains the ideas in this entire section. 

Let $M$ be an $n$-dimensional connected oriented closed smooth manifold. 
Let $M_0$ be the manifold $M$ with a small open disk removed. By Morse theory and Poincar\'{e} duality, a connected oriented closed smooth manifold is homotopy equivalent to a Poincar\'{e} duality complex with a single top cell, and the deleted manifold $M_0$ is homotopy equivalent to the $(n-1)$-skeleton of $M$. Hence, the choice of notations are agreed upon.  

Let $A\stackrel{j_A}{\hookrightarrow} M$ and $B\stackrel{j_B}{\hookrightarrow} M$ be two embeddings of oriented closed smooth submanifolds with complementary positive dimensions. Denote ${\rm dim}(A)=m$. Then ${\rm dim}(B)=n-m$. Suppose that $A$ and $B$ are transversal to each other in $M$. In differential topology, there is a well-defined {\it geometric intersection number} of $A$ and $B$ in $M$, given by 
\[
\sum\limits_{p\in A\cap B} {\rm sgn}(p),
\]
where the {\it local intersection number} ${\rm sgn}(p)=\pm 1$ is determined by the orientations of $A$, $B$ and $M$. It is a classical result that the geometric intersection number is equal to the algebraic intersection number in the sense of Definition (\ref{alg-intdef})
\[
A\testcap B=\sum\limits_{p\in A\cap B} {\rm sgn}(p).
\]
As before we say that $A$ and $B$ {\it essentially intersect} if their intersection number is nonzero, while they are {\it dual} if their intersection number is one. 

Now suppose that there is a compact Lie group $G$ acting on $M$ by 
\[
G\times M\stackrel{\mu}{\longrightarrow} M
\]
such that the image of the embedding $A\stackrel{j_A}{\hookrightarrow} M$ lies in an orbit of $G$, say $Gx$ for some $x\in M$, and the isotropy group of $x$ is trivial. Note that as the isotropy groups of points in the same orbit are conjugate, the latter assumption does not depend on the choice of $x$ and is equivalent to the orbit $Gx$ being {\it free}: $Gx\cong G$.  
Then by choosing a basepoint $\ast\in Gx$, the embedding $j_A$ factors through $G$ as
\[
j_A: A\stackrel{\mathfrak{j}_A}{\longrightarrow} G\stackrel{\mu_l}{\longrightarrow} M,
\]  
where $\mu_l(g)=\mu(g, \ast)$ for any $g\in G$, and $\mathfrak{j}_A(a)\in G$ is the unique element such that $\mu(\mathfrak{j}_A(a), \ast)=j_A(a)$ for any $a\in A$. Note that the definition of $\mathfrak{j}_A$ depends on the choice of the basepoint $\ast\in Gx$, but it is irrelevant to our application here. 

As before we multiply the two maps $A\stackrel{\mathfrak{j}_A}{\longrightarrow} G$ and $B\stackrel{j_B}{\hookrightarrow} M$ together to define a composite 
\begin{equation}\label{smintABfeq}
f: A\times B \stackrel{\mathfrak{j}_A\times j_B}{\longrightarrow} G\times M\stackrel{\mu}{\longrightarrow} M.
\end{equation}
We call $f$ the {\it product of $A$ and $B$ in $M$}. It is clear that if a cohomology class $z\in H^\ast(M;\mathbb{Z})$ is primitive, then it is $f$-primitive. 
By Lemma \ref{int=deglemma} if $A$ and $B$ essentially intersect and either the Poincar\'{e} dual of $A$ in $M$ is primitive with $n>2m$ or the Poincar\'{e} dual of $B$ in $M$ is primitive with $n<2m$, we have that 
\[
{\rm deg}(f)=\pm A\testcap B=\pm \sum\limits_{p\in A\cap B} {\rm sgn}(p).
\]

With the above discussions, the geometric version of Theorem \ref{intthm} can be formulated as follows. 
\begin{theorem}\label{smintthm}
Let $G$ be a compact Lie group. 
Let $M$ be an $n$-dimensional simply connected closed smooth $G$-manifold. 
Let $A$ and $B$ be two non-contractible simply connected oriented closed smooth submanifolds of $M$ with dimension $m$ and $n-m$, respectively. 
Suppose that $A$ and $B$ essentially intersect in $M$, and $A$ lies in a free orbit of $G$. 
Suppose that either the Poincar\'{e} dual of $A$ in $M$ is primitive with $n>2m$ or the Poincar\'{e} dual of $B$ in $M$ is primitive with $n<2m$. 

If the restriction of the product $f$ (\ref{smintABfeq}) on the $(n-1)$-skeletons 
$
f_0: (A\times B)_0\stackrel{}{\longrightarrow} M_0
$ 
is inert, 
then the attaching map for the top cell of $M$ is inert after localization away from all primes $p$ that divide 
\[
A\testcap B=\sum\limits_{p\in A\cap B} {\rm sgn}(p).
\]
In particular, if $A$ and $B$ are dual up to sign in $M$, then the attaching map for the top cell of $M$ is inert.  ~$\qqed$
\end{theorem}
 
 \begin{remark}\label{smintthm-cond-rmk}
In Subsection \ref{subsec: steenq} below we will see that the geometric conditions of Theorem \ref{smintthm} is reasonable and not strong in certain sense. Accordingly, only the homotopy condition that $f_0$ is inert is tricky and may need to be checked on a case-by-case basis. 
 \end{remark}
 
\begin{remark}\label{smintthm-result-rmk}
It is clear that the result of Theorem \ref{smintthm} holds in a more general context; that is, it can be stated for general maps $A\stackrel{}{\longrightarrow} M$ and $B\stackrel{}{\longrightarrow} M$ rather than just embeddings. However, based on Remark \ref{smintthm-cond-rmk} and Subsection \ref{subsec: steenq} in the sequel, embedding is a reasonable  geometric context. 
\end{remark}
 
When $M=G$ itself is a Lie group, the $G$-action on itself is free, and 
any two maps $j_A: A\stackrel{}{\longrightarrow} G$ and $j_B: B\longrightarrow G$ can be multiplied to produce a map
\begin{equation}\label{MGintABfeq}
f: A\times B \stackrel{j_A\times j_B}{\longrightarrow} G\times G\stackrel{\mu}{\longrightarrow} G,
\end{equation}
where $\mu$ is the multiplication of $G$. 
In this case, Theorem \ref{smintthm} reduces to the following theorem. 
\begin{theorem}\label{smMHintthm}
Let $G$ be a simply connected closed Lie group. 
Let $A\stackrel{}{\hookrightarrow} G$ and $B\stackrel{}{\hookrightarrow} G$ be two embeddings of simply connected closed smooth submanifolds with complementary dimensions which essentially intersect. 
Suppose that the Poincar\'{e} dual of $A$ in $M$ is primitive and ${\rm dim}(B)>{\rm dim}(A)>0$.

If the restriction of the product map $f$ (\ref{MGintABfeq}) on the lower skeletons 
$
f_0: (A\times B)_0\stackrel{}{\longrightarrow} G_0
$ 
is inert, 
then the attaching map for the top cell of $G$ is inert after localization away from all primes $p$ that divide 
\[
A\testcap B=\sum\limits_{p\in A\cap B} {\rm sgn}(p).
\]
In particular, if $A$ and $B$ are dual up to sign in $G$ then the attaching map for the top cell of $G$ is inert.  ~$\qqed$
\end{theorem}

%----------------------------------------------------------------------------------
\subsection{Steenrod's realization problem}
\label{subsec: steenq}

$\, $

In 1949, Steenrod \cite{Eil49} raised the following problem: for a given homology class $w\in H_i(K)$ of a finite polyhedron $K$, does there exist an orientable closed manifold $A$ and a map $A\stackrel{f_w}{\longrightarrow} K$ such that $f_{w\ast}[A]=w$? A cohomological version of this problem is closely related to our results in this section. 
\begin{problem}\label{steenq}
Let $M$ be an $n$-dimensional oriented closed smooth manifold. For a cohomology class $x\in H^{n-m}(M;\mathbb{Z})$, does there exist an embedding $A\stackrel{j_A}{\hookrightarrow} M$ such that the Poincar\'{e} dual of $j_{A\ast}([A])$ is $x$? 
\end{problem}
Thom \cite{Tho54} studied this problem in depth. Among others, he provided a positive answer to this problem for rational cohomology classes.

\begin{theorem}\label{thomqthm}
Let $M$ be an $n$-dimensional oriented closed smooth manifold. 
For any cohomology class $x\in H^{n-m}(M;\mathbb{Z})$ with $n\geq m\geq 0$, there exists a nonzero integer $d$ and an embedding $A\stackrel{j_A}{\hookrightarrow} M$ such that the Poincar\'{e} dual of $j_{A\ast}([A])$ is $dx$.  ~$\qqed$
\end{theorem}
In this case, we say $dx$ can be {\it realized} by an embedding. Additionally, Thom proved the following concrete result.
\begin{theorem}\label{thom8thm}
For any oriented closed smooth manifold of dimension less than $9$, all integral cohomology classes can be realized by submanifold embeddings. ~$\qqed$
\end{theorem}

In Theorems \ref{smintthm} and \ref{smMHintthm} two embeddings with complementary dimensions are given as a condition for a $G$-manifold $M$. With Thom's results we see that this condition is not strong. For instance, we may choose two generators $x$, $y\in H^{\ast}(M;\mathbb{Z})$ such that $\langle x\cup y, [M]\rangle=1$, and either $x$ is primitive with ${\rm deg}(x)>{\rm deg}(y)$ or $y$ is primitive with ${\rm deg}(y)>{\rm deg}(x)$. Then there exist two nonzero integers $d(x)$ and $d(y)$ such that $d(x)x$ and $d(y)y$ can be realized by some embeddings $A\stackrel{j_A}{\hookrightarrow} M$ and $B\stackrel{j_B}{\hookrightarrow} M$, respectively. Note that their intersection number is $d(x)d(y)\neq 0$. This satisfies all the geometric conditions of Theorems \ref{smintthm} and \ref{smMHintthm} except that the embedding of $A$ is required to lie in a free orbit in Theorem \ref{smintthm}. For the latter condition, it is satisfied particularly when the action is free and $A$ is a subspace of $G$. Indeed, if the action of $G$ on $M$ is free, the restriction map $\mu_l: G\longrightarrow M$ (\ref{muleq}) is injective, and any subspace $A$ of $G$ gives an embedding $A\stackrel{}{\hookrightarrow} G\stackrel{\mu_l}{\hookrightarrow} M$ which realizes a certain cohomology class of $M$.

\newpage
%----------------------------------------------------------------------------------------------------------------------------------------------------------------------------------------------------------%

\section{Stiefel manifolds}
\label{sec: degen}
In this section, we prove results regarding the inertness of the top cell attachments of Stiefel manifolds. These results can be viewed as special examples of the results in Section \ref{sec: int}. To this end, we first prove a more general result, which is Part (2) of Theorem \ref{FEBthmintro} for strict fibrations.

\begin{proposition}\label{FEBprop}
Let $F\stackrel{j}{\longrightarrow} E\stackrel{\pi}{\longrightarrow} B$ be a homotopy fibration of connected Poincar\'{e} duality complexes of positive dimensions. 
If the homotopy fibration has a homotopy section, then the attaching map for the top cell of $E$ is inert.

Additionally, the assertion holds after localization at any set of primes.
\end{proposition}
\begin{proof}
Let $B\stackrel{s}{\longrightarrow} E$ be a homotopy section of the homotopy fibration, that is, $\pi\circ s\simeq 1_B$. Since $F$ and $B$ are of positive dimensions, by the $CW$-approximation the maps $F\stackrel{j}{\longrightarrow} E$ and $B\stackrel{s}{\longrightarrow} E$ factor through the lower skeleton $E_0$ of $E$
\[
F\stackrel{j_0}{\longrightarrow} E_0 \stackrel{i_E}{\longrightarrow} E \ \ {\rm and} \ \ 
B\stackrel{s_0}{\longrightarrow} E_0 \stackrel{i_E}{\longrightarrow} E,
\]
respectively, where $i_E$ is the lower skeleton inclusion. Consider the homotopy commutative diagram
\[
\diagram
\Omega F\times \Omega B \rto^{\Omega j_0\times  \Omega s_0} \drto_{\Omega j\times  \Omega s}  &  \Omega E_0\times \Omega E_0   \rto^<<<{\mu} \dto^{\Omega i_E\times  \Omega i_E}  & \Omega E_0   \dto^{\Omega i_E}\\
& \Omega E\times \Omega E   \rto^<<<<{\mu}  & \Omega E,
\enddiagram
\]
where the maps $\mu$ are the loop multiplications and the right square homotopy commutes by the naturality of loop maps. By the previous discussion, the lower direction composition around the diagram is a weak homotopy equivalence, and then is a homotopy equivalence by the Whitehead theorem. Therefore, the upper direction composition around the diagram implies that $\Omega i_E$ has a right homotopy inverse, that is, the attaching map for the top cell of $E$ is inert. This proves the proposition in the integral case, while the proposition in the local case follows by the same argument. 
\end{proof}

For application we are interested in a special case of Proposition \ref{FEBprop}. 
Let 
\begin{equation}\label{FMSeq}
F\stackrel{j}{\longrightarrow} M\stackrel{\pi}{\longrightarrow} S^m
\end{equation}
be a homotopy fibration of connected nilpotent Poincar\'{e} duality complexes. Suppose that $M$ has a homotopy action from an $H$-complex $G$
\[
\mu: G\times M\stackrel{}{\longrightarrow} M.
\]
As in (\ref{muleq}) the action on the basepoint of $M$ gives a map $\mu_l: G\stackrel{}{\longrightarrow} M$.  
 
\begin{corollary}\label{FMScor}
For the $G$-complex $M$ in the homotopy fibration (\ref{FMSeq}) with $m\geq 1$, suppose that there exists a map 
\[
s: S^m\stackrel{\mathfrak{s}}{\longrightarrow} G\stackrel{\mu_l}{\longrightarrow} M
\]
such that $\pi\circ s$ is homotopic to a degree $k$ self-map of $S^m$. If $F$ is non-contractible, then the attaching map for the top cell of $M$ is inert after localization away from all primes $p$ that divide $k$.  
\end{corollary}
\begin{proof}
We work in the homotopy category away from all primes $p$ that divide $k$. The composite  
\[
s': S^m\stackrel{1/k}{\longrightarrow} S^m\stackrel{s}{\longrightarrow} M
\]
is a homotopy section of the homotopy fibration (\ref{FMSeq}). Then the corollary follows from Proposition \ref{FEBprop} immediately. 
\end{proof}
\begin{remark}
Alternatively, Corollary \ref{FMScor} can be proved by showing that the composite
\[
S^m\times F\stackrel{\mathfrak{s}\times j}{\longrightarrow} G\times M\stackrel{\mu}{\longrightarrow} M
\]
is a local homotopy equivalence. 
\end{remark}
When $m=1$ Corollary \ref{FMScor} can be strengthened.  
\begin{corollary}\label{FMS1cor}
Let 
\[
F\stackrel{}{\longrightarrow} M\stackrel{\pi}{\longrightarrow} S^1
\]
be a homotopy fibration of connected Poincar\'{e} duality complexes over the circle. If $F$ is non-contractible, then the attaching map for the top cell of $M$ is inert.
\end{corollary}
\begin{proof}
By the long exact sequence of the homotopy groups of a fibration, the sequence 
\[
\pi_1(M)\stackrel{\pi_\ast}{\longrightarrow} \pi_1(S^1)\cong\mathbb{Z}\stackrel{}{\longrightarrow} \pi_0(F)=0
\]
is exact. Therefore, $\pi_\ast$ is surjective and the homotopy fibration has a homotopy section. Then Proposition \ref{FEBprop} implies that the attaching map for the top cell of $M$ is inert.
\end{proof}
Corollary \ref{FMS1cor} could be useful to study inertness of the top cell attachments for certain manifolds by various so-called {\it fibering theorem} in geometric topology. In this context, a manifold $M$ is {\it fibered over a circle} if there is a fibre bundle
\[
N\stackrel{}{\longrightarrow} M\stackrel{\pi}{\longrightarrow} S^1
\] 
such that the fibre $N$ is a manifold. Furthermore, a manifold $M$ is {\it virtually fibered over a circle} if some finite cover of $M$ is fibered over a circle. 
The following corollary follows from Corollary \ref{FMS1cor} and Lemma \ref{coverinertlemma1} immediately.     
\begin{corollary}\label{fibering-inert-cor}
If a connected orientable closed manifold $M$, with a single top cell, is virtually fibered over a circle, then the attaching map for the top cell of $M$ is inert. $\qqed$
\end{corollary}
For instance, a classical result of Stallings \cite{Sta61} characterizes closed $3$-manifolds that  are surface bundles over $S^1$. Roughly, a $3$-manifold $M$ is fibered over a circle if and only if its fundamental group is an extension of $\mathbb{Z}$ by a surface group: 
\[
1\stackrel{}{\longrightarrow}G\stackrel{}{\longrightarrow}\pi_1(M)\stackrel{}{\longrightarrow} \mathbb{Z}\stackrel{}{\longrightarrow} 1,
\]
where $G\cong \pi_1(S)$ is the fundamental group of a surface $S$. 
There are generalizations to higher dimensions; see \cite{BL66, Far71, Sie70, FS14, QSW23} for instance. 
Furthermore, the affirmative solution to Thurston's virtually fibered conjecture \cite{Ago13, PW18} implies that any closed hyperbolic $3$-manifold is virtually fibered over a circle. 
Corollary \ref{fibering-inert-cor} then implies that the attaching maps for the top cells of these manifolds are inert. 
The inertness of $3$-manifolds will be further discussed in Subsection \ref{subsec: lowfolds}, where a definitive result, Theorem \ref{3inertthm}, will be proved.

%----------------------------------------------
\subsection{Complex and quaternionic Stiefel manifolds}
\label{subsec: stiefel}

$\, $

As Stiefel manifolds are quotients of classical Lie groups, we start with some necessary information of corresponding Lie groups. 

Let $U(n)$ and $Sp(n)$ be the $n$-th unitary and symplectic groups, respectively. 
Recall that the cohomology ring of $U(n)$ and $Sp(n)$ are the exterior algebras
\[
H^\ast(U(n);\mathbb{Z})\cong \Lambda (e_1, e_3,\ldots, e_{2n-1}), \ \ \ 
H^\ast(Sp(n);\mathbb{Z})\cong \Lambda (e_3, e_7,\ldots, e_{4n-1}),
\]
where ${\rm deg}(e_l)=l$. By Bott periodicity the homotopy groups of $U(n)$ and $Sp(n)$ satisfy 
\[
\pi_{2i-1}(U(n))\cong \pi_{2i}(BU(n))\cong \mathbb{Z}, \ \ \ \pi_{4i-1}(Sp(n))\cong \pi_{4i}(BSp(n))\cong\mathbb{Z}, 
\]
for each $1\leq i\leq n$, where $B$ is the classifying functor. 
For a class $w\in H_l(X;\mathbb{Z})$, denote by $w^\#\in H^l(X;\mathbb{Z})$ its dual under the Kronecker pairing. 
There is the remarkable Borel-Hirzebruch divisibility theorem \cite[Section 26.10]{BH59} in manifold topology; an equivalent homotopical formulation can be found in \cite[Theorem 6.13]{MT91}.

\begin{lemma}\label{BHlemma}
Let $1\leq i\leq n$. 
\begin{itemize}
\item For the complex bundle over $S^{2i}$ classified by a generator of $\pi_{2i}(BU(n))$, its top Chern class is $\pm (i-1)! [S^{2i}]^{\#}$. Equivalently, for the map 
\[
s_i^{\mathbb{C}}: S^{2i-1}\stackrel{}{\longrightarrow}  U(n)
\] 
representing a generator of $\pi_{2i-1}(U(n))$, $s_i^{\mathbb{C}\ast}(e_{2i-1})=\pm (i-1)! [S^{2i-1}]^{\#}$; 
\item For the symplectic bundle over $S^{4i}$ classified by a generator of $\pi_{4i}(BSp(n))$, its top symplectic Pontryagin class is $\pm (2i-1)! [S^{4i}]^{\#}$ if $i$ is odd, and is $\pm (2i-1)! \cdot 2[S^{4i}]^{\#}$ if $i$ is even. Equivalently, for the map 
\[
s_i^{\mathbb{H}}: S^{4i-1}\stackrel{}{\longrightarrow}  Sp(n)
\] 
representing a generator of $\pi_{4i-1}(Sp(n))$, 
\[
s_i^{\mathbb{H}\ast}(e_{4i-1})=
\left\{\begin{array}{cc}
\pm (2i-1)!\cdot  [S^{4i-1}]^{\#}  & ~i~{\rm is}~{\rm odd} \\
\pm (2i-1)! \cdot 2[S^{4i-1}]^{\#} & ~i~{\rm is}~{\rm even}.
\end{array}\right.
\]
~$\qqed$
\end{itemize}
\end{lemma}

Let $V_{n,k}(\mathbb{C})=U(n)/U(n-k)$ and $V_{n,k}(\mathbb{H})=Sp(n)/Sp(n-k)$ be the complex and quaternionic Stiefel manifolds, respectively. There are the canonical fibre bundles
\begin{equation}\label{V=Uqeq}
\begin{split}
U(n-k)\stackrel{}{\longrightarrow} U(n)\stackrel{p}{\longrightarrow}  V_{n,k}(\mathbb{C})&, \\
Sp(n-k)\stackrel{}{\longrightarrow} Sp(n)\stackrel{p}{\longrightarrow}  V_{n,k}(\mathbb{H})&,
\end{split}
\end{equation}
such that their Serre spectral sequences collapse at the $E_2$-terms and then there are isomorphisms of cohomology groups
\[
\begin{split}
H^\ast(U(n);\mathbb{Z})&\cong H^\ast(U(n-k);\mathbb{Z})\otimes H^\ast(V_{n,k}(\mathbb{C});\mathbb{Z}), \\
H^\ast(Sp(n);\mathbb{Z})&\cong H^\ast(Sp(n-k);\mathbb{Z})\otimes H^\ast(V_{n,k}(\mathbb{H});\mathbb{Z}).
\end{split}
\] 
Moreover, the cohomology rings of $V_{n,k}(\mathbb{C})$ and $V_{n,k}(\mathbb{H})$ are the  exterior algebras 
\[
\begin{split}
H^\ast(V_{n,k}(\mathbb{C});\mathbb{Z})&\cong \Lambda (e_{2(n-k+1)-1}, e_{2(n-k+2)-1},\ldots, e_{2n-1}),\\
H^\ast(V_{n,k}(\mathbb{H});\mathbb{Z})&\cong \Lambda (e_{4(n-k+1)-1}, e_{4(n-k+2)-1},\ldots, e_{4n-1}).
\end{split}
\]

\begin{theorem}\label{VCHthm}
Let $k\geq 2$. For the complex and quaternionic Stiefel manifolds, 
\begin{itemize}
\item
the attaching map for the top cell of $V_{n,k}(\mathbb{C})$ is inert after localization at any prime $p$ with $p>n-1$;
\item
 the attaching map for the top cell of $V_{n,k}(\mathbb{H})$ is inert after localization at any prime $p$ with $p>2n-1$.
 \end{itemize}
\end{theorem}
\begin{proof}
Consider the complex Stiefel manifold $V_{n,k}(\mathbb{C})$. 
There is the canonical fibre bundle 
\begin{equation}\label{VCn-1k-1eq}
V_{n-1,k-1}(\mathbb{C})\stackrel{i}{\longrightarrow}V_{n,k}(\mathbb{C}) \stackrel{\pi}{\longrightarrow} S^{2n-1}
\end{equation}
such that its Serre Spectral sequence collapses at the $E_2$-term. In particular, $\pi^\ast ([S^{2n-1}]^{\#})=e_{2n-1}$. By Lemma \ref{BHlemma}, there is a map $s_n^{\mathbb{C}}: S^{2n-1}\stackrel{}{\longrightarrow}  U(n)$ such that $s_n^{\mathbb{C}\ast}(e_{2n-1})=\pm (n-1)! [S^{2n-1}]^{\#}$. Let $s$ be the composite 
\[
s: S^{2n-1}\stackrel{s_n^{\mathbb{C}}}{\longrightarrow}U(n)\stackrel{p}{\longrightarrow}  V_{n,k}(\mathbb{C}),
\]
where $p$ is the canonical projection in (\ref{V=Uqeq}).  
It satisfies that
\[
(\pi \circ s)^\ast([S^{2n-1}]^{\#})=s^\ast(e_{2n-1})=s_n^{\mathbb{C}\ast}(e_{2n-1})=\pm (n-1)! [S^{2n-1}]^{\#},
\]
or equivalently, the composite $\pi \circ s$ is of degree $\pm (n-1)!$. This means that the fibre bundle (\ref{VCn-1k-1eq}) has a homotopy section up to degree $\pm (n-1)!$. 

We prove the theorem by induction on $k$, keeping the difference $n-k$ fixed.  
When $k=2$, we have $V_{n-k+1, 1}\cong S^{2n-2k+1}$ and the fibre bundle (\ref{VCn-1k-1eq}) becomes
\[
S^{2n-2k+1}\stackrel{i}{\longrightarrow}V_{n-k+2,2}(\mathbb{C}) \stackrel{\pi}{\longrightarrow} S^{2n-2k+3}.
\]
Since $\pi\circ s$ is homotopic to a degree $\pm (n-k+1)!$ self-map of $S^{2n-2k+3}$, Corollary \ref{FMScor} implies that the attaching map for the top cell of $V_{n-k+2,2}(\mathbb{C})$ is inert after localization at any prime $p$ with $p>n-k+1$.

Suppose by induction that the attaching map for the top cell of $V_{n-1,k-1}(\mathbb{C})$ is inert after localization at any prime $p$ with $p>n-2$. We have showed that the fibre bundle (\ref{VCn-1k-1eq}) satisfies that $\pi\circ s$ is homotopic to a degree $\pm (n-1)!$ self-map of $S^{2n-1}$. Then Corollary \ref{FMScor} implies that the attaching map for the top cell of $V_{n,k}(\mathbb{C})$ is inert after localization at any prime $p$ with $p>n-1$. 
This completes the induction and the theorem is proved for $V_{n,k}(\mathbb{C})$.

The proof of the theorem for $V_{n,k}(\mathbb{H})$ is similar by considering the canonical fibre bundle 
\[
V_{n-1,k-1}(\mathbb{H})\stackrel{i}{\longrightarrow}V_{n,k}(\mathbb{H}) \stackrel{\pi}{\longrightarrow} S^{4n-1}.
\]
\end{proof}

%----------------------------------------------
\subsection{Real Stiefel manifolds}
\label{subsec: stiefel2}

$\, $

Let $SO(n)$ and $Spin(n)$ be the $n$-th special orthogonal group and spin group, respectively. Let $V_{n,k}(\mathbb{R})=SO(n)/SO(n-k)$ be the real Stiefel manifold. At any odd prime $p$, there are homotopy equivalences of Harris \cite{Har61}
\[
\begin{split}
SO(2n+1)&\simeq_{p} Spin(2n+1)\simeq_{p} Sp(n), \\
SO(2n)&\simeq_{p} Spin(2n)\simeq_{p} Spin(2n-1)\times S^{2n-1}.
\end{split}
\]

\begin{lemma}\label{VHarlemma}
At any odd prime $p$, there are homotopy equivalences 
\[
\begin{split}
V_{2n+1,2k}(\mathbb{R})&\simeq_{p}V_{n,k}(\mathbb{H}),\\
V_{2n+2, 2k+1}(\mathbb{R})  &\simeq_{p}S^{2n+1} \times V_{n,k}(\mathbb{H}),\\
V_{2n+2, 2k}(\mathbb{R}) & \simeq_{p} S^{2n+1}\times V_{2n+1,2k-1}(\mathbb{R}).
\end{split}
\]
\end{lemma}
\begin{proof}
By the homotopy equivalences of Harris, there are local homotopy equivalences
\[
\begin{split}
V_{2n+1,2k}(\mathbb{R})&=SO(2n+1)/SO(2(n-k)+1) \\ &\simeq_{p} Sp(n)/Sp(n-k)\\
            &=V_{n,k}(\mathbb{H}),\\
V_{2n+2, 2k+1}(\mathbb{R})&=SO(2n+2)/SO(2(n-k)+1)\\ 
&\simeq_{p} S^{2n+1}\times Spin(2n+1)/Spin(2(n-k)+1) \\
&\simeq_{p} S^{2n+1}\times Sp(n)/Sp(n-k) \\
&= S^{2n+1} \times V_{n,k}(\mathbb{H}).
\end{split}
\] 
This proves the first two homotopy equivalences in the lemma. Further, by the homotopy equivalence
\[
\begin{split}
V_{2n+1, 2k-1}(\mathbb{R})&=SO(2n+1)/SO(2(n-k)+2)\\
&\simeq_{p} Spin(2n+1)/Spin(2(n-k)+2),
\end{split}
\]
it follows that 
\[
\begin{split}
V_{2n+2, 2k}(\mathbb{R})&=SO(2n+2)/SO(2(n-k)+2)\\ 
&\simeq_{p} S^{2n+1}\times Spin(2n+1)/Spin(2(n-k)+2) \\
&\simeq_{p}  S^{2n+1}\times V_{2n+1, 2k-1}(\mathbb{R}).
\end{split}
\] 
This proves the last homotopy equivalence in the lemma. 
\end{proof}

Combining Theorem \ref{VCHthm}, Lemmas \ref{VHarlemma} and \ref{pdt-inert-lemma} we obtain the following theorem immediately. 
\begin{theorem}\label{VRthm}
Let $k\geq 2$. The following hold: 
\begin{itemize}
\item the attaching map for the top cell of $V_{2n+1,2k}(\mathbb{R})$ is inert after localization at any prime $p$ with $p>2n-1$;
\item the attaching map for the top cell of $V_{2n+2, 2k-1}(\mathbb{R})$ is inert after localization at any odd prime;
\item the attaching map for the top cell of $V_{2n+2, 2k-2}(\mathbb{R})$ is inert after localization at any odd prime.
\end{itemize}
 ~$\qqed$
\end{theorem}

\newpage
%%%%%%%%%%%%%%%%%%%%%%%%%%%%%%%%%%%%%%%%%%%%%%%%%%%%%%%%%%%%%%%%%%%%%%

%\part{Understand inertness from codomain}
%\label{part3}

%----------------------------------------------------------------------------------------------------------------------------------------------------------------------------------------------------------%
\section{A comparison theorem}
\label{sec: det}
In this section, we study the inertness of the top cell attachment of a Poincar\'{e} duality complex by comparing it with a twisted product of spheres. The main result is Theorem \ref{detthm}, which implies Theorem \ref{detthmintro} as a special case. 

Let $M$ be an $n$-dimensional simply connected Poincar\'{e} duality complex.  Let $M_0$ be the $(n-1)$-skeleton of $M$. There is a homotopy cofibre sequence
\begin{equation}\label{Mcofibeq}
S^{n-1}\stackrel{h_M}{\longrightarrow} M_0\stackrel{i_M}{\longrightarrow}M \stackrel{q_M}{\longrightarrow}S^n
\end{equation}
where $i_M$ is the inclusion map, $h_M$ is the attaching map for the top cell, and $q_M$ is the pinch map to the top cell. 

Let $D$ be an $n$-dimensional Poincar\'{e} duality complex with a homotopy fibration
\begin{equation}\label{Dfibeq}
S^{n-m}\stackrel{\iota_D}{\longrightarrow} D\stackrel{\psi}{\longrightarrow} S^m
\end{equation}
such that it admits a homotopy section and $n>m+1> 2$. 
\begin{lemma}\label{Dlemma}
The Poincar\'{e} duality complex $D$ satisfies the following:
\begin{itemize}
\item[(1).] 
the $(n-1)$-skeleton $D_0\simeq S^{m}\vee S^{n-m}$ and there is a homotopy cofibre sequence 
\begin{equation}\label{Dcofibeq}
S^{n-1}\stackrel{h_D}{\longrightarrow} S^{m}\vee S^{n-m}\stackrel{i_D}{\longrightarrow} D \stackrel{q_D}{\longrightarrow} S^n
\end{equation}
such that $\psi\circ i_D\simeq q_1$, where $q_D$ is the pinch map to the top cell and $q_1$ is the projection onto the first wedge summand;  
\item[(2).] 
$H^\ast(D;\mathbb{Z})\cong H^\ast(S^m\times S^{n-m};\mathbb{Z})$ as graded rings; and  
\item[(3).] 
$\Omega D\simeq \Omega (S^m\times S^{n-m})$;
\item[(4).] 
the attaching map for the top cell of $D$ is inert.
\end{itemize}
\end{lemma}
\begin{proof}
(1). Let $S^m\stackrel{s_D}{\longrightarrow} D$ be a homotopy section, that is, $\psi\circ s_D\simeq 1_{S^m}$. By the $CW$-approximation, both $s_D$ and $\iota_D$ factor through $D_0$ up to homotopy as
\[
S^m\stackrel{s_0}{\longrightarrow} D_0 \stackrel{i_D'}{\longrightarrow} D \ \ {\rm and} \ \  
S^{n-m} \stackrel{\iota_0}{\longrightarrow} D_0 \stackrel{i_D'}{\longrightarrow} D 
\]
for some maps $s_0$ and $\iota_0$, respectively, where $i_D'$ is the inclusion map.  
Then the composite 
\[
e: S^m\vee S^{n-m} \stackrel{s_0\vee \iota_0}{\longrightarrow} D_0\vee D_0 \stackrel{\nabla}{\longrightarrow} D_0
\] 
induces an isomorphism on homology and hence is a homotopy equivalence by the Whitehead theorem, where $\nabla$ is the folding map. 
Let $i_D$ be the composite
\[
i_D: S^m\vee S^{n-m} \stackrel{s_D\vee \iota_D}{\longrightarrow} D\vee D \stackrel{\nabla}{\longrightarrow} D.
\] 
By construction, $i_D$ is homotopic to the composite $S^m\vee S^{n-m} \stackrel{e}{\longrightarrow}D_0\stackrel{i_D'}{\hookrightarrow} D$. 
The cell structure of $D$ implies that there is a homotopy cofibre sequence of the form (\ref{Dcofibeq}). 
Also, since $\psi\circ s_D\simeq 1_{S^m}$ and $\psi\circ\iota_D$ is null homotopic, it follows that $\psi \circ i_D\simeq q_1$. This shows statement (1). 

(2). Since $D$ is a Poincar\'{e} duality complex with exactly one cell in each dimension $m$, $n-m$ and $n$, respectively, statement (2) follows immediately. 

(3). With the first two statements, statement (3) follows from \cite[Lemma 2.3]{BT14}.

(4). Proposition \ref{FEBprop} implies statement (4) immediately. 
\end{proof}

Suppose that there is a degree one map 
\[
f: M\stackrel{}{\longrightarrow} D.
\]
By the $CW$-approximation, the restriction map of $f$ on the $(n-1)$-skeleton $M_0$ factors through $D_0$ up to homotopy, and we denote the map by $f_0: M_0\stackrel{}{\longrightarrow} D_0$. 
By Lemma \ref{degk=pushoutlemma} and the homotopy cofibre sequences (\ref{Mcofibeq}) and (\ref{Dcofibeq}), there is a homotopy cofibration diagram  
\begin{equation}\label{MDdiag}
\diagram 
       M_0\rto^-{i_M}\dto^{f_0} & M\dto^{f}  \rto^-{q_M}& S^{n}\ddouble   \\ 
       S^m\vee S^{n-m}\rto^-{i_D} &  \rto^-{q_D}  D  &  S^{n} .
  \enddiagram
\end{equation}
By Lemma \ref{Dlemma} (4) the attaching map $h_D$ is inert, that is, $\Omega i_D$ has a right homotopy inverse. Our aim is to use Diagram (\ref{MDdiag}) to derive the inertness of $h_M$ from that of $h_D$ under certain reasonable conditions. For this, we adopt a cubic method of Theriault in \cite{The24b}, rooted in the remarkable work \cite{BT14} of Beben-Theriault.

Consider the composites
\begin{equation}\label{phieq}
\varphi: M\stackrel{f}{\longrightarrow} D\stackrel{\psi}{\longrightarrow} S^m, \ \ \ 
\varphi_0:M_0\stackrel{i_M}{\longrightarrow} M\stackrel{\varphi}{\longrightarrow} S^m.
\end{equation}
Note that $\varphi_0$ is the restriction of $\varphi$ on $M_0$. 
Let $E$ and $E_0$ be the homotopy fibres of $\varphi$ and $\varphi_0$, respectively. We have the homotopy fibrations
\begin{equation}\label{4fibeq}
\begin{split}
E_0\stackrel{}{\longrightarrow}M_0\stackrel{\varphi_0}{\longrightarrow} S^m,&\\
E\stackrel{}{\longrightarrow}M\stackrel{\varphi}{\longrightarrow} S^m,&\\
S^{n-m}\stackrel{\iota_D}{\longrightarrow} D\stackrel{\psi}{\longrightarrow} S^m, \\
\Omega S^m\ltimes S^{n-m}\stackrel{}{\longrightarrow}S^{m}\vee S^{n-m}\stackrel{q_1}{\longrightarrow} S^m, 
\end{split}
\end{equation}
where the third homotopy fibration is (\ref{Dfibeq}), and the last homotopy fibration is by Lemma \ref{wedgeqlemma}. 
Since 
\[
\psi\circ f=\varphi \ \  ~{\rm and}~ ~\ \ \psi\circ i_D\simeq q_1
\]
by Lemma \ref{Dlemma} (1), the above four homotopy fibrations form a homotopy cube 
\begin{equation} 
  \label{E0Ecube} 
  %\spreaddiagramcolumns{-0.3pc}\spreaddiagramrows{-0.3pc} 
   \diagram
      E_0 \rrto^-{g_0}\drto^-{i_E}\ddto^-(0.33){} & & \Omega S^m\ltimes S^{n-m}\dline^-{}\drto^{i_S} & \\
      & E\rrto^(0.35){g}\ddto^(0.25){} & \dto & S^{n-m}\ddto^{\iota_D} \\
      M_0\rline^(0.6){f_0}\drto^(0.6){i_M} & \rto & S^{m}\vee S^{n-m}\drto^{i_D} & \\
      & M\rrto^{f} & &  D
  \enddiagram 
\end{equation}  
over the same base $S^m$, where the bottom face is the left square of (\ref{MDdiag}) and hence a homotopy pushout, the four vertical faces are homotopy pullbacks, and the maps $i_E$, $i_S$, $g_0$ and $g$ are the induced maps. Mather's Cube Lemma (Theorem \ref{cubethm}) implies that the top face is a homotopy pushout. From the construction, it is clear that the map $i_S$ is homotopic to the canonical projection $p_2$ in Lemma \ref{ltimeslemma0}.

Recall our aim is to study the inertness of the attaching map $h_M$ by comparing it to that of $h_D$, or equivalently, to study the existence of right homotopy inverse of $\Omega i_M$ from the existence of right homotopy inverse of $\Omega i_D$. Diagram (\ref{E0Ecube}), along with Lemma \ref{loopphilemma} below, allows us to transfer this problem relying on the bottom face to the corresponding problem relying on the top face; that is, to study the existence of right homotopy inverse of $\Omega i_E$ from the existence of right homotopy inverse of $\Omega i_S$.

\begin{lemma}\label{loopphilemma}
Suppose that $\Omega M_0\stackrel{\Omega \varphi_0}{\longrightarrow} \Omega S^m$ has a right homotopy inverse. Then 
\begin{itemize}
\item[(1).] the map $\Omega i_M$ has a right homotopy inverse if and only if $\Omega i_E$  has a right homotopy inverse; and
\item[(2).] the map $\Omega f_0$ has a right homotopy inverse if and only if $\Omega g_0$  has a right homotopy inverse.
\end{itemize}
\end{lemma}
\begin{proof}
The condition that $\Omega \varphi_0$ has a right homotopy inverse is equivalent to that the homotopy fibration $E_0\stackrel{}{\longrightarrow}M_0\stackrel{\varphi_0}{\longrightarrow} S^m$ splits after looping, and hence is equivalent to that the looped map $\Omega E_0\longrightarrow \Omega M_0$ has a left homotopy inverse. Then Theorem \ref{pullbackthm} (3) can be applied to the left face of Diagram (\ref{E0Ecube}) to show that the map $\Omega i_M$ has a right homotopy inverse if and only if $\Omega i_E$  has a right homotopy inverse, and to the rear face of Diagram (\ref{E0Ecube}) to show that the map $\Omega f_0$ has a right homotopy inverse if and only if $\Omega g_0$  has a right homotopy inverse.
\end{proof}

The following proposition gives a criterion on the inertness of the top cell attachment of a Poincar\'{e} duality complex by comparing it to a twisted product of spheres. 

\begin{proposition}\label{detprop}
Let $M$ be an $n$-dimensional simply connected Poincar\'{e} duality complex. Let $D$ be an $n$-dimensional Poincar\'{e} duality complex determined by a homotopy fibration (\ref{Dfibeq}) with a homotopy section and $n>m+1>2$. Suppose that there is a degree one map 
\[
f: M\stackrel{}{\longrightarrow} D
\]
such that its looped restriction $\Omega f_0: \Omega M_0\stackrel{}{\longrightarrow} \Omega (S^m\vee S^{n-m})$ has a right homotopy inverse. If there is a homotopy pushout 
\begin{equation}\label{YM0diag}
\diagram 
      Y\rto^-{}\dto^{} & M_0\dto^{f_0}  \\ 
    S^{m}\rto^-{i_1} & S^m\vee S^{n-m}
  \enddiagram
\end{equation}
for some complex $Y$ with $i_1$ the inclusion into the first wedge summand, then the following hold
\begin{itemize}
\item[(1).]
the attaching map for the top cell of $M$ is inert;
\item[(2).]
the map $\Omega f$ has a right homotopy inverse.
\end{itemize}

Additionally, if the map $f$ is of degree $k$ with $k\neq 0$ and the conditions hold after localization away from the primes $p$ that divide $k$, then the two conclusions hold after localization away from the primes $p$ that divide $k$.
\end{proposition}
\begin{proof}
If the map $f$ is of degree $k$ with $k\neq 0$ we may work in the local category away from  after localization away from the primes $p$ that divide $k$. 
By the same argument in the proof of Theorem \ref{inertdegkthm}, we can replace the map $f$ by a map $f'$ of degree one without changing involved homotopy types and homotopy classes. Hence, it suffices to prove the proposition for the case when $f$ is of degree one.  

Let $s: \Omega (S^m\vee S^{n-m}) \stackrel{}{\longrightarrow} \Omega M_0$ be a right homotopy inverse of $\Omega f_0$. By Lemma \ref{Dlemma}, Diagrams (\ref{MDdiag}) and (\ref{phieq}) we have 
\[
\varphi_0=\varphi\circ i_M=\psi\circ f\circ i_M\simeq \psi\circ i_D\circ f_0\simeq q_1\circ f_0: M_0\stackrel{}{\longrightarrow} S^m.
\]
Then the composite
\[
\Omega S^m\stackrel{\Omega i_1}{\longrightarrow} \Omega (S^m\vee S^{n-m}) \stackrel{s}{\longrightarrow} \Omega M_0
\]
satisfies that $\Omega \varphi_0\circ s\circ \Omega i_1\simeq  \Omega q_1\circ \Omega f_0\circ s\circ \Omega i_1\simeq \Omega q_1\circ \Omega i_1\simeq 1_{\Omega S^m}$, that is, $\Omega \varphi_0$ has a right homotopy inverse. In particular, Lemma \ref{loopphilemma} is allowed to be applied.

(1). We want to apply Theorem \ref{pushoutthm2} to the top homotopy pushout in Diagram (\ref{E0Ecube}). To this end, we need to show that the map $g_0$ is a homotopy cofibre quotient. 
Let $E_{Y}$ be the homotopy fibre of the map $Y\stackrel{ }{\longrightarrow} S^m$ in \eqref{YM0diag}. We have the homotopy fibrations
\[
\begin{split}
E_{Y} \stackrel{}{\longrightarrow}Y\stackrel{ }{\longrightarrow} S^m&,\\
\ast  \stackrel{}{\longrightarrow} S^m  \stackrel{1_{S^m}}{\longrightarrow} S^m&.
\end{split}
\]
Combining these homotopy fibrations with the first and the last homotopy fibrations in \eqref{4fibeq} gives a homotopy cube
\[
  \label{E1/2E0cube} 
  %\spreaddiagramcolumns{-0.3pc}\spreaddiagramrows{-0.3pc} 
   \diagram
      E_{Y} \rrto^-{ \mathfrak{j}}\drto^-{}\ddto^-(0.33){} & & E_0\dline^-{}\drto^{g_0} & \\
      & \ast\rrto^(0.35){}\ddto^(0.25){} & \dto & \Omega S^m\ltimes S^{n-m}\ddto^{} \\
      Y\rline^(0.6){\mathfrak{i}}\drto^(0.6){ } & \rto & M_0\drto^{f_0} & \\
      & S^m\rrto^{i_1} & &  S^m\vee S^{n-m}
  \enddiagram 
\] 
over the same base $S^m$, where the bottom face is the homotopy pushout (\ref{YM0diag}), the four vertical faces are homotopy pullbacks by construction, and $ \mathfrak{j}$ is the induced map. 
Note that the right face of the diagram is the rear face of Diagram (\ref{E0Ecube}). 
Mather's Cube Lemma (Theorem \ref{cubethm}) implies that the top face of the diagram is a homotopy pushout, that is, there is the homotopy cofibration
\[
E_{Y} \stackrel{ \mathfrak{j}}{\longrightarrow} E_0\stackrel{g_0}{\longrightarrow} \Omega S^m\ltimes S^{n-m}.
\]

Consider the homotopy cofibration diagram
\[
\diagram 
      E_{Y}  \ddouble  \rto^-{\mathfrak{j}} & E_0\rto^-{g_0}\dto^-{i_E} & \Omega S^m\ltimes S^{n-m} \dto^-{i_S} \\ 
     E_{Y} \rto^-{\mathfrak{l}}          & E\rto^-{g} & S^{n-m},
  \enddiagram
 \]
where the right square is the top face of Diagram (\ref{E0Ecube}), and $\mathfrak{l}:=\mathfrak{i}_E\circ \mathfrak{j}$. Since the top row is a homotopy cofibration and the right square is a homotopy pushout, Lemma \ref{cone+pushoutlemma} implies that the bottom row is a homotopy cofibration. 
By  Lemma \ref{loopphilemma} (2), $\Omega f_0$ has a right homotopy inverse implies that $\Omega g_0$ has a right homotopy inverse, that is, the map $\mathfrak{j}$ is inert. Also, the construction of the right face of Diagram (\ref{E0Ecube}) indicates that the map $i_S$ is homotopic to the canonical projection $p_2$ in Lemma \ref{ltimeslemma0}, and then it has the inclusion $S^{n-m}\stackrel{j_2}{\hookrightarrow} \Omega S^m\ltimes S^{n-m}$ as a right homotopy inverse. Therefore, Theorem \ref{pushoutthm2} can be applied to show that $\Omega i_E$ has a right homotopy inverse. From Lemma \ref{loopphilemma} (1), it implies that $\Omega i_M$ has a right homotopy inverse, that is, the attaching map $h_M$ of the top cell of $M$ is inert.
 
(2). Consider the left square of Diagram (\ref{MDdiag}). By Lemma \ref{Dlemma} (4) the map $\Omega (S^m\vee S^{n-m})\stackrel{\Omega i_D}{\longrightarrow} \Omega D$ has a right homotopy inverse, and by assumption the map $\Omega M_0\stackrel{\Omega f_0}{\longrightarrow} \Omega (S^m\vee S^{n-m})$ has a right homotopy inverse. Hence, the homotopy commutativity of the left square of Diagram (\ref{MDdiag}) implies that $\Omega f$ has a right homotopy inverse. 
\end{proof}

In Proposition \ref{detprop}, the condition that there exists a homotopy pushout (\ref{YM0diag}) roughly means that $S^{n-m}$ can be subtracted from $M_0$. This condition can be satisfied when $M_0$ has certain nice cell structures and then several variations of Proposition \ref{detprop} can be proved. We start with a lemma that characterizes the structure of a twisted product of spheres.  

\begin{lemma}\label{Klemma}
Let $K$ be a simply connected Poincar\'{e} duality complex determined by a homotopy cofibration
\[
S^{n-1}\stackrel{h_K}{\longrightarrow} K_0\stackrel{i_K}{\longrightarrow} K
\]
such that $H^\ast(K)\cong H^\ast(S^m\times S^{n-m})$. 
Then $K$ is the total complex of a homotopy fibration 
\[
S^{n-m}\stackrel{}{\longrightarrow} K\stackrel{\psi_K}{\longrightarrow} S^m
\]
with a homotopy section if and only if there are maps $S^m\stackrel{i}{\longrightarrow} K_0$ and $K_0\stackrel{q}{\longrightarrow} S^m$ such that $q\circ i\simeq 1_{S^m}$ and the composite  
\[
S^{n-1}\stackrel{h_K}{\longrightarrow} K_0\stackrel{q}{\longrightarrow} S^m
\] 
is null homotopic. 

Furthermore, if the above statements hold then $q\simeq \psi_K\circ i_K$. 
\end{lemma}
\begin{proof}
Suppose that there is a homotopy fibration 
$
S^{n-m}\stackrel{}{\longrightarrow} K\stackrel{\psi_K}{\longrightarrow} S^m
$
with a homotopy section. Then by Lemma \ref{Dlemma} (1), $K_0\simeq S^m\vee S^{n-m}$, and the composite $S^m\vee S^{n-m}\stackrel{i_K}{\longrightarrow} K\stackrel{\psi_K}{\longrightarrow} S^m$ is homotopic to the canonical projection $S^m\vee S^{n-m}\stackrel{q_1}{\longrightarrow} S^m$. 
We may not distinguish $K_0 $ and $S^m\vee S^{n-m}$ and let $q=q_1$. It is clear that the composite 
\[
S^m\stackrel{i_1}{\longrightarrow} S^m\vee S^{n-m} \stackrel{q_1}{\longrightarrow} S^m
\]
is the identity map.
Further, the composite $q\circ h_K$ is homotopic to the composite
\[
S^{n-1}\stackrel{h_K}{\longrightarrow} S^m\vee S^{n-m}\stackrel{i_K}{\longrightarrow} K\stackrel{\psi_K}{\longrightarrow} S^m
\]
which is null homotopic as $i_K\circ h_K$ is. 

Conversely, suppose that there are maps $S^m\stackrel{i}{\longrightarrow} K_0$ and $K_0\stackrel{q}{\longrightarrow} S^m$ such that $q\circ i\simeq 1_{S^m}$ and the composite  
$
S^{n-1}\stackrel{h_K}{\longrightarrow} K_0\stackrel{q}{\longrightarrow} S^m
$
is null homotopic. Then there is a diagram of homotopy cofibrations
\[
\diagram
S^{n-1}\rto^{h_K} \ddouble  & K_0 \rto^{i_K} \dto^{q}  & K \dto^{q'}\\
S^{n-1}      \rto^{\ast}         & S^m       \rto^<<<<{i_1}   & S^m\vee S^n\\
\enddiagram
\]
where $q'$ is the induced map. Consider the composite  
\[
\psi_K: K\stackrel{q'}{\longrightarrow}  S^{m}\vee S^{n}\stackrel{q_1}{\longrightarrow} S^{m}.
\]
Let $F$ be the homotopy fibre of $\psi_K$. Since by assumption $H^\ast(K;\mathbb{Z})\cong H^\ast(S^m\times S^{n-m};\mathbb{Z})$, an easy argument on the Serre spectral sequence shows that $H^\ast(F)\cong H^\ast(S^{n-m})$. Then the bottom cell inclusion $S^{n-m}\longrightarrow F$ induces an isomorphism on homology and therefore is a homotopy equivalence by the Whitehead theorem. Hence, there a homotopy fibration
\[
S^{n-m}\stackrel{}{\longrightarrow} K\stackrel{\psi_K}{\longrightarrow} S^m.
\]
Moreover, as the composite $S^m\stackrel{i}{\longrightarrow} K_0\stackrel{q}{\longrightarrow} S^m$ is homotopic to the identity map and $\psi_K=q_1\circ q'$ by definition, there is the homotopy commutative diagram
\[
\diagram
S^m \rto^{i}  \drdouble    & K_0 \rto^{i_K} \dto^{q} & K \dto^{q'}  \drto^{\psi_K}\\
 & S^m   \rto^<<<<{i_1}            & S^{m}\vee S^{n}  \rto^{q_1}  & S^m,
\enddiagram
\]
where the middle square is the right square of the previous diagram. It follows that $\psi_K\circ i_K\simeq q_1 \circ i_1\circ q \simeq q$, and $\psi_K \circ i_K \circ i\simeq q\circ i$ is homotopic to the identity map, that is, $\psi_K$ has a homotopy section.
\end{proof}

The following theorem is a variation of Proposition \ref{detprop}. 
Denote by $X_1\vee X_2\stackrel{q_i}{\longrightarrow} X_i$ the canonical projection onto the $i$-th wedge summand with $i=1$ or $2$. 

\begin{theorem}\label{detthm}
Let $M$ be an $n$-dimensional simply connected Poincar\'{e} duality complex. Let $D$ be an $n$-dimensional Poincar\'{e} duality complex determined by a homotopy fibration (\ref{Dfibeq}) with a homotopy section and $n>m+1>2$. Suppose that there is a degree one map 
\[
f: M\stackrel{}{\longrightarrow} D 
\]
satisfying the following:
\begin{itemize}
\item
the composite $M_0\stackrel{f_0}{\longrightarrow} S^{m}\vee S^{n-m}\stackrel{q_2}{\longrightarrow}  S^{n-m}$ has a right homotopy inverse after looping, and it can be extended to a homotopy cofibration 
\[
Y\stackrel{}{\longrightarrow}M_0\stackrel{q_2\circ f_0}{\longrightarrow} S^{n-m}
\]
for some complex $Y$;
\item the composite $Y\stackrel{}{\longrightarrow}M_0\stackrel{f_0}{\longrightarrow} S^{m}\vee S^{n-m}\stackrel{q_1}{\longrightarrow} S^{m}$ has a right homotopy inverse after looping. 
\end{itemize}
Then the attaching map for the top cell of $M$ is inert.

Additionally, if the map $f$ is of degree $k$ with $k\neq 0$ and the conditions hold after localization away from the primes $p$ that divide $k$, then the attaching map for the top cell of $M$ is inert after localization away from the primes $p$ that divide $k$.
\end{theorem}
\begin{proof}
As discussed in the beginning of the proof of Proposition \ref{detprop}, we only need to consider the case when $f: M\stackrel{}{\longrightarrow} D$ is a degree one map. Instead of studying $f$ itself, we would like to construct a new degree one map $g: M\stackrel{}{\longrightarrow} K$ from $f$, and prove that it satisfies all the conditions of Proposition \ref{detprop}.

By assumption there is a homotopy cofibration
\[
Y\stackrel{j_Y}{\longrightarrow} M_0\stackrel{q_2\circ f_0}{\longrightarrow} S^{n-m}. 
\]
Denote by $Y\stackrel{f_Y}{\longrightarrow} S^m$ the composite 
\begin{equation}\label{fYeq}
f_Y: Y\stackrel{j_Y}{\longrightarrow}M_0\stackrel{f_0}{\longrightarrow} S^{m}\vee S^{n-m}\stackrel{q_1}{\longrightarrow} S^{m}.
\end{equation}
Then the map $j_Y$ is inert and the loop map $\Omega f_Y$ has a right homotopy inverse by assumption.

Let $K_0$ be the homotopy pushout of $j_Y$ and $f_Y$. Then there is the diagram of homotopy cofibrations 
\begin{equation}\label{YM0Kdiag}
\diagram 
      Y\rto^-{j_Y}\dto^{f_Y} & M_0\dto^{g_0}  \rto^{q_2\circ f_0} & S^{n-m}   \ddouble \\ 
    S^{m}\rto^-{i_{K_0}} & {K_0} \rto^{q_{K_0}} & S^{n-m}   
  \enddiagram
\end{equation}
where $i_{K_0}$, $g_0$, and $q_{K_0}$ are the induced maps. 
Further, the diagram of homotopy cofibrations
\begin{equation}\label{MKdiag}
\diagram 
      S^{n-1}\rto^{h_M}\ddouble & M_0\dto^{g_0}  \rto^{i_M} & M   \dto^{g}\\ 
    S^{n-1}\rto^{h_{K}} & {K_0} \rto^{i_{K}} & K   
  \enddiagram
\end{equation}
defines the $CW$-complex $K$ as the homotopy cofibre of $h_{K}:=g_0\circ h_M$, where $i_K$ is the inclusion map and $g$ is the induced map. We want to show that the degree one map 
\[
g: M\stackrel{}{\longrightarrow} K
\]
satisfies all the conditions of Proposition \ref{detprop}. This is divided into three steps.

{\it (I). $K$ is a Poincar\'{e} duality complex with $H^\ast(K;\mathbb{Z})\cong H^\ast(S^m\times S^{n-m};\mathbb{Z})$.}

Compare $K$ with the Poincar\'{e} duality complex $D$ through the homotopy commutative diagram
\[\label{DMKdiag}
\diagram
D_0 \dto^{i_D} & M_0 \rto^{g_0} \lto_{f_0}  \dto^{i_M}& K_0 \dto^{i_K} \\
D                   & M \rto^{g} \lto_{f}  & K,  
\enddiagram
\]
where the left and right squares are derived from the restrictions of $f$ and $g$ on the lower skeletons, respectively. 
Consider the bottom homotopy cofibration of Diagram \eqref{YM0Kdiag}. 
From the constructions, it is easy to check that the long exact sequence in cohomology of the homotopy cofibration splits such that  
\[
H^\ast(K_0;\mathbb{Z})\cong H^\ast(D_0;\mathbb{Z})\cong H^\ast(S^m\vee S^{n-m};\mathbb{Z}),
\]
and the injective morphisms $g_0^\ast$ and $f_0^\ast$ have the same image in $H^\ast(M_0;\mathbb{Z})$. It follows that $g^\ast$ and $f^\ast$ have the same image in $H^\ast(M;\mathbb{Z})$ up to dimension $n-1$. 
Since the maps $g$ and $f$ are of degree one, a straightforward argument on cup product structure shows that $g^\ast$ is injective and has the same image as $f^\ast$ in $H^\ast(M;\mathbb{Z})$. In particular, $K$ is a Poincar\'{e} duality complex such that 
\[
H^\ast(K;\mathbb{Z})\cong H^\ast(D;\mathbb{Z})\cong H^\ast(S^m\times S^{n-m};\mathbb{Z}).
\]

{\it (II). There is a homotopy fibration
\[
S^{n-m}\stackrel{}{\longrightarrow} K\stackrel{}{\longrightarrow} S^m
\]
with a homotopy section.
}

Consider the homotopy commutative diagram
\begin{equation}\label{K0qSmdiag}
\begin{aligned}
\xymatrix{ 
Y \ar[r]^{j_Y} \ar[d]^{f_Y} & M_0 \ar[d]^{g_0}  \ar@/^0.7pc/[ddr]^{q_1\circ f_0} \\
S^m\ar[r]^{i_{K_0}}  \ar@/_0.7pc/[drr]_{=}  &  {K_0} \ar@{.>}[dr]^(0.3){q} \\
&&S^m,
}
\end{aligned}
\end{equation}
where the inner square is a homotopy pushout by Diagram (\ref{YM0Kdiag}) and the outer square homotopy commutes by the definition of $f_Y$ in (\ref{fYeq}). The universal property of homotopy pushout implies that there is a map $q:{K_0}\stackrel{}{\longrightarrow}S^m$ such that the two triangular regions homotopy commute. In particular, there is a homotopy commutative diagram 
\[
\diagram 
      S^{n-1}\rto^{h_M}\ddouble & M_0\dto^{g_0}  \rto^<<<{f_0}  &  S^m\vee S^{n-m} \rto^<<<<{q_1}& S^m   \ddouble\\ 
    S^{n-1}\rto^{h_{K}} & {K_0} \rrto^{q} && S^m  
  \enddiagram
\]
where the left square is the left square of (\ref{MKdiag}). 
Since the map $f$ is of degree one, Lemma \ref{deg1=pushoutlemma} implies that the composite $f_0\circ h_M$ is homotopic to the attaching map $S^{n-1}\stackrel{h_D}{\longrightarrow} S^{m}\vee S^{n-m}$ for the top cell of $D$. Then by Lemma \ref{Klemma} and its proof the top row composition $q_1\circ f_0\circ h_M \simeq q_1\circ h_D$ is null homotopic. It follows that $q\circ h_K$ is null homotopic by the homotopy commutativity of the above diagram. Combining this with the homotopy $q\circ i_{K_0}\simeq 1_{S^m}$ in Diagram \eqref{K0qSmdiag}, we can apply Lemma \ref{Klemma} to show that there is a homotopy fibration 
$
S^{n-m}\stackrel{}{\longrightarrow} K\stackrel{}{\longrightarrow} S^m
$
with a homotopy section. In particular, we have $K_0\simeq S^m\vee S^{n-m}$ by Lemma \ref{Dlemma} (1). 

{\it (III). The restriction map $g_0: M_0\stackrel{}{\longrightarrow} S^m\vee S^{n-m}$ has a right homotopy inverse after looping.}

Consider the left square of Diagram (\ref{YM0Kdiag}) which is a homotopy pushout. By assumption the map $j_Y$ is inert and the loop map $\Omega f_Y$ has a right homotopy inverse. Then Theorem \ref{pushoutthm} implies that the loop map $\Omega g_0$ has a right homotopy inverse.

To summarize, we have showed that $g: M\stackrel{}{\longrightarrow} K$ is a degree one map between simply connected Poincar\'{e} duality complexes such that $K$ is a twisted product of spheres and the looped restriction $\Omega g_0: \Omega M_0\stackrel{}{\longrightarrow} \Omega (S^m\vee S^{n-m})$ has a right homotopy inverse. Further, we have the left square of Diagram (\ref{YM0Kdiag}) for free. Therefore, all the conditions of Proposition \ref{detprop} are satisfied, and it follows that the attaching map for the top cell of $M$ is inert.
\end{proof}

\begin{remark}\label{detprop=thm-rmk}
Theorem \ref{detthm} was proved from Proposition \ref{detprop} (1). Conversely, we can also prove Proposition \ref{detprop} (1) from Theorem \ref{detthm}. 

Indeed, suppose that a degree one map $f: M\stackrel{}{\longrightarrow} D$ satisfies the conditions of Proposition \ref{detprop}. Then Diagram (\ref{YM0diag}) implies the diagram of homotopy cofibrations
\[
\diagram 
      Y\rto^-{}\dto^{} & M_0\dto^{f_0}  \rto^{q_2\circ f_0} & S^{n-m}   \ddouble \\ 
    S^{m}\rto^-{i_{1}} & S^m\vee S^{n-m}\rto^<<<<{q_2} & S^{n-m}.   
  \enddiagram
\]
Since $\Omega f_0: \Omega M_0\stackrel{}{\longrightarrow} \Omega (S^m\vee S^{n-m})$ has a right homotopy inverse by assumption, then both the composites
\[
\Omega M_0\stackrel{\Omega f_0}{\longrightarrow} \Omega (S^m\vee S^{n-m})\stackrel{\Omega q_1}{\longrightarrow} \Omega S^m \ \ \ {\rm and} \ \  
\Omega M_0\stackrel{\Omega f_0}{\longrightarrow} \Omega (S^m\vee S^{n-m})\stackrel{\Omega q_2}{\longrightarrow} \Omega S^{n-m}
\]
have right homotopy inverses. In particular, the first condition of Theorem \ref{detthm} is satisfied. Further, applying Theorem \ref{pushoutthm} to the left square in the above diagram, it follows that the map $Y\stackrel{}{\longrightarrow} S^m$ has a right homotopy inverse after looping. It is clear that $Y\stackrel{}{\longrightarrow} S^m$ is homotopic to the composite 
\[
Y\stackrel{}{\longrightarrow}M_0\stackrel{f_0}{\longrightarrow} S^{m}\vee S^{n-m}\stackrel{q_1}{\longrightarrow} S^{m}.
\]
Hence, the latter composite has a right homotopy inverse after looping, and the second condition of Theorem \ref{detthm} is satisfied. 
Then all the conditions of Theorem \ref{detthm} are satisfied and it follows that the attaching map for the top cell of $M$ is inert.  

In conclusion, Theorem \ref{detthm} and Proposition \ref{detprop} (1) are equivalent. However, it is unknown that whether the map $f$ in Theorem \ref{detthm} satisfies the conditions of Proposition \ref{detprop}. 
\end{remark}

The conditions of Theorem \ref{detthm} can be satisfied when $M_0\simeq Y\vee S^{n-m}$ for some complex $Y$, and we are led to the special case stated in the introduction.  

\begin{proof}[Proof of Theorem \ref{detthmintro}]
As discussed in the beginning of the proof of Proposition \ref{detprop}, we only need to consider the case when $f: M\stackrel{}{\longrightarrow} D$ is a degree one map. We want to show that $f$ satisfies the conditions of Theorem \ref{detthm}. 

By assumption $M_0\simeq Y\vee S^{n-m}$ and $f_0\simeq f_Y\vee 1_{S^{n-m}}: Y\vee S^{n-m}\stackrel{}{\longrightarrow}S^m\vee S^{n-m}$. Then there is the homotopy cofibration diagram
\[
\diagram 
      Y\rto^-{i_1}\dto^-{f_Y} & Y\vee S^{n-m}\dto^-{f_0} \rto^-{q_2} & S^{n-m}\ddouble  \\ 
      S^m\rto^-{i_1} & S^m\vee S^{n-m}\rto^-{q_2} & S^{n-m},
  \enddiagram
  \]
  where $i_1$ and $q_2$ are the natural injections and projections. Since the inclusion map $S^{n-m}\stackrel{i_2}{\longrightarrow} Y\vee S^{n-m}$ is a right homotopy inverse of $q_2$, the first condition of Theorem \ref{detthm} is satisfied. Further, it is clear that $Y\stackrel{f_Y}{\longrightarrow} S^m$ is homotopic to the composite 
 \[
Y\stackrel{i_1}{\longrightarrow}Y\vee S^{n-m}\stackrel{f_Y\vee 1_{S^{n-m}}}{\longrightarrow} S^{m}\vee S^{n-m}\stackrel{q_1}{\longrightarrow} S^{m}.
\]
Then it has a right homotopy inverse after looping by the assumption that $\Omega f_Y$ has a right homotopy inverse. This means that the second condition of Theorem \ref{detthm} is satisfied. 
Therefore, the two conditions of Theorem \ref{detthm} for $f$ are satisfied, and we can apply Theorem \ref{detthm} to conclude that the attaching map for the top cell of $M$ is inert.
\end{proof}

A further special case can be formulated in the following corollary. Since the restriction morphism $i_M^\ast: H^{<n}(M;\mathbb{Z})\stackrel{\cong}{\longrightarrow}H^{<n}(M_0;\mathbb{Z})$ is an isomorphism, we may use same notation to denote a class in $H^{<n}(M;\mathbb{Z})$ and its restriction in $H^{<n}(M_0;\mathbb{Z})$. Denote by $s_l\in H^l(S^l;\mathbb{Z})$ the generator such that $\langle s_l, [S^l]\rangle=1$. 
\begin{corollary}\label{detcor}
Let $n>m+1> 2$ such that $\pi_{n-1}(S^m)$ is a torsion group. 
Let $M$ be an $n$-dimensional simply connected Poincar\'{e} duality complex with $M_0\simeq Y\vee S^{n-m}$ for some complex $Y$. Suppose that there is a map $f_Y: Y\stackrel{}{\longrightarrow} S^{m}$ such that $\Omega f_Y$ has a right homotopy inverse and the pairing number $\ell:=\langle f_Y^\ast (s_m) \cup s_{n-m}, [M] \rangle$ is nonzero. 
Then the following hold:
\begin{itemize}
\item[(1).]
if $\pi_{n-1}(S^m)$ is trivial and $\ell=\pm 1$, then the attaching map for the top cell of $M$ is inert;
\item[(2).]
if $\pi_{n-1}(S^m)$ is nontrivial and its order is divisible by $\ell$, then the attaching map for the top cell of $M$ is inert after localization away from all primes $p$ that divide the order of $\pi_{n-1}(S^m)$.
\end{itemize}
\end{corollary}
\begin{proof}
(1). Consider the diagram of homotopy cofibrations
\[
\diagram
S^{n-1} \ddouble \rto^<<<{h_M}  & Y \vee S^{n-m} \rto^<<<{i_M} \dto^{f_Y\vee 1_{S^{n-m}}} & M \dto^{f} \\
S^{n-1}               \rto^<<<{h_D}  &  S^m\vee S^{n-m}  \rto^<<<{i_D}                                       & D
\enddiagram
\]
defining the $CW$-complex $D$ as the homotopy cofibre of $h_D:= (f_Y\vee 1_{S^{n-m}})\circ h_M$ with $i_D$ the inclusion map, where $f$ is the induced degree one map. By the condition $\langle f_Y^\ast (s_m) \cup s_{n-m}, [M] \rangle=\ell=\pm 1$, it is easy to see that $D$ is a Poincar\'{e} duality complex. Since $\pi_{n-1}(S^m)$ is trivial, the composite
\[
S^{n-1}\stackrel{h_D}{\longrightarrow} S^m\vee S^{n-m} \stackrel{q_1}{\longrightarrow} S^m
\]
is null homotopic, and then Lemma \ref{Klemma} implies that there is a homotopy fibration 
\[
S^{n-m}\stackrel{}{\longrightarrow} D\stackrel{}{\longrightarrow} S^m
\]
with a homotopy section. Therefore, the degree one map $M\stackrel{f}{\longrightarrow} D$ satisfies that its restriction $f_0\simeq f_Y\vee 1_{S^{n-m}}$, $\Omega f_Y$ has a right homotopy inverse by assumption, and $D$ is a twisted product of spheres. Then we can applied Theorem \ref{detthmintro} to show that the attaching map for the top cell of $M$ is inert.

(2). In this case we may work in the local category away from all primes $p$ that divide the order of $\pi_{n-1}(S^m)$. Then locally $\pi_{n-1}(S^m)$ is trivial and $\ell$ is a unit element. Hence, by the same argument as in (1) we can apply a local version of Theorem \ref{detthmintro} to prove that the attaching map for the top cell of $M$ is locally inert.
\end{proof}

\begin{remark}\label{det-rmk}
If for a cohomology class $x\in H^m(X;\mathbb{Z})$ there exists a map $g: X\stackrel{}{\longrightarrow} S^m$ such that $g^\ast(s_m)=x$, we say that the map $g$ {\it detects} the class $x$. Then the condition $\langle f_Y^\ast (s_m) \cup s_{n-m}, [M] \rangle=1$ in Corollary \ref{detcor} means that the map $f_Y$ detects the Poincar\'{e} dual of $S^{n-m}$ in $M$. 
\end{remark}

\newpage
%----------------------------------------------------------------------------------------------------------------------------------------------------------------------------------------------------------%
\section{A further comparison}
\label{sec: ex1}
In Section \ref{sec: det}, we studied the inertness property by comparing a Poincar\'{e} duality complex with a twisted product of spheres. Building on this approach, we extend our investigation to compare a Poincar\'{e} duality complex with other candidates. In this section, we provide another such comparison and apply it to revisit a family of Poincar\'{e} duality complexes introduced by Beben-Theriault \cite{BT14}, as well as reprove a result of Theriault \cite{The24a} concerning connected sums. We also discuss certain low dimensional manifolds as concrete examples. 

Let $M$ be an $n$-dimensional Poincar\'{e} duality complex with a single top cell such that its $(n-1)$-skeleton satisfies
\[
M_0\simeq A\vee B
\] 
for some complexes $A$ and $B$. Then there is a homotopy cofibration
\[
S^{n-1}\stackrel{h_M}{\longrightarrow} A \vee B\stackrel{i_M}{\longrightarrow}M,
\] 
where $h_M$ is the attaching map and $i_M$ is the inclusion map. 
Let $N$ be the homotopy cofibre of the composite 
\[
A\stackrel{i_A}{\hookrightarrow}A\vee B\stackrel{i_M}{\longrightarrow}M,
\]
where $i_A$ is the inclusion map. Then there is a homotopy cofibration
\[
S^{n-1}\stackrel{q_B\circ h_M}{\longrightarrow} B\stackrel{i_N}{\longrightarrow} N,
\]
where $A\vee B\stackrel{q_B}{\longrightarrow} B$ is the projection map, and $i_N$ is the inclusion map. Note that $N$ is not necessarily a Poincar\'{e} duality complex. 
\begin{proposition}\label{M0=A+Bprop}
Let $M$ as above. If the map $q_B\circ h_M$ is inert, then the following hold:
\begin{itemize}
\item[(1).] the attaching map $h_M$ is inert;
\item[(2).] there is a homotopy equivalence
\[
\Omega M \simeq \Omega N\times \Omega(\Omega N\ltimes A).
\]
\end{itemize}
\end{proposition}
\begin{proof}
(1). Notice that there is a homotopy cofibration 
\[
A\stackrel{i_M\circ i_A}{\longrightarrow} M\stackrel{f}{\longrightarrow} N
\]
for a degree one map $f$ such that it restricts to $q_B$ on the $(n-1)$-skeletons. It follows that there is the homotopy cofibration diagram 
\begin{equation}\label{AMNdiag}
\diagram 
& S^{n-1} \dto^{h_M} \rdouble & S^{n-1} \dto^{q_B\circ h_M}\\
      A\rto^-{i_A}\ddouble & A\vee B\rto^-{q_B}\dto^{i_M} & B\dto^{i_N}  \\ 
      A\rto^-{i_M\circ i_A} & M\rto^-{f} &   N.
  \enddiagram
\end{equation} 
It is clear that $q_B$ has a right homotopy inverse, and then $i_A$ is inert. Since $\Omega i_N$ has right homotopy inverse by assumption, Theorem \ref{pushoutthm2} implies that $\Omega i_M$ has a right homotopy inverse, that is, the attaching map $h_M$ is inert.

(2). Recall both $\Omega i_N$ and $\Omega q_B$ have right homotopy inverses, and so does their composite. Then the homotopy commutativity of the right lower square of (\ref{AMNdiag}) implies that $\Omega f: \Omega M\stackrel{}{\longrightarrow} \Omega N$ has a right homotopy inverse. Therefore, applying Theorem \ref{GTcofib} to the bottom homotopy cofibration of Diagram (\ref{AMNdiag}), we obtain a homotopy equivalence
\[
\Omega M\simeq \Omega N\times \Omega(\Omega N\ltimes A).
\]
\end{proof}
From the proof it is clear that the local version of Proposition \ref{M0=A+Bprop} holds automatically, that is, if the attaching map for the top cell of $N$ is locally inert then the attaching map for the top cell of  $M$ is locally inert. 
%--------------------------------------------------------------------------------------%
\subsection{Beben-Theriault's complexes}
\label{subsec: BT}

$\, $

Let $M$ be an $n$-dimensional Poincar\'{e} duality complex with a single top cell such that its $(n-1)$-skeleton satisfies
\[
M_0\simeq J\vee S^m\vee S^{n-m}
\]  
with $n>m\geq 1$. We will not distinguish these two complexes. 
Then there is a homotopy cofibration
\[
S^{n-1}\stackrel{h_M}{\longrightarrow} J\vee S^m\vee S^{n-m}\stackrel{i_M}{\longrightarrow}M,
\]
where $i_M$ is the inclusion map and $h_M$ is the attaching map for the top cell. Let $x\in H^m(M;\mathbb{Z})$ and $y\in H^{n-m}(M;\mathbb{Z})$ be the cohomology classes dual to $i_{M\ast}([S^m])$ and $i_{M\ast}([S^{n-m}])$, respectively, where $[-]$ stands for the fundamental class of a Poincar\'{e} duality complex. In particular, $x$ and $y$ are generator elements. Suppose that 
\begin{equation}\label{xy=1eq}
\langle x\cup y, [M]\rangle =1.
\end{equation}

\begin{theorem}\label{exJthm}
Let $M$ as above. Then the following hold:
\begin{itemize}
\item
the attaching map for the top cell of $M$ is inert;
\item 
there is a homotopy equivalence
\[
\Omega M\simeq \Omega (S^m\times S^{n-m})\times \Omega(\Omega (S^m\times S^{n-m})\ltimes J).
\]
\end{itemize}
\end{theorem}
\begin{proof}
Let $D$ be the homotopy cofibre of the composite 
\[J\stackrel{i_J}{\hookrightarrow}J\vee S^m\vee S^{n-m}\stackrel{i_M}{\longrightarrow}M,
\]
where $i_J$ is the inclusion map. 
Then there is a homotopy cofibration 
\[
J\stackrel{i_M\circ i_J}{\longrightarrow} M\stackrel{f}{\longrightarrow} D
\]
for a degree one map $f$. From the assumption (\ref{xy=1eq}) it is clear that $D$ is a Poincar\'{e} duality complex determined by a homotopy cofibration 
\[
S^{n-1}\stackrel{}{\longrightarrow} S^{m}\vee S^{n-m}\stackrel{i_D}{\longrightarrow} D
\]
such that $H^\ast(D;\mathbb{Z})\cong H^\ast(S^m\times S^{n-m};\mathbb{Z})$. By \cite[Lemma 2.3]{BT14} $\Omega D\simeq \Omega (S^m\times S^{n-m})$ and $\Omega i_D$ has a right homotopy inverse.
Hence, $M$ satisfies the condition of Proposition \ref{M0=A+Bprop}, and it implies that $h_M$ is inert, and there is a homotopy equivalence
\[
\begin{split}
\Omega M 
&\simeq  \Omega D\times \Omega(\Omega D\ltimes J)\\
&\simeq\Omega (S^m\times S^{n-m})\times \Omega(\Omega (S^m\times S^{n-m})\ltimes J).
\end{split}
\]
\end{proof}

\begin{remark}\label{BT14rmk}
Suppose that $M$ is $(m-1)$-connected with torsion free homology and $n-m\geq m\geq 2$. The corresponding result in this case was proved by Beben-Theriault in \cite{BT14}. Accordingly, we may call the Poincar\'{e} duality complex $M$ in Theorem \ref{exJthm} a {\it Beben-Theriault complex}. 
\end{remark}

\begin{example}\label{BTex1}
Let $M$ be an $(n-1)$-connected $2n$-dimensional Poincar\'{e} duality complex with $n\geq 2$. It is clear that 
\[
H^n(M)\cong \bigoplus_{d}\mathbb{Z}
\]
for some integer $d\geq 0$, and there is a homotopy cofibration
\[
S^{2n-1}\stackrel{}{\longrightarrow} \bigvee_{d} S^n\simeq M_0\stackrel{}{\longrightarrow} M.
\]
Suppose that $d\geq 2$. By the arguments of \cite[Example 4.2]{BT22}, \cite[Lemma 3.1]{Hua22} and \cite[Lemma 3.2]{Hua24}, it can be shown that $M$ is a Beben-Theriault complex. Then Theorem \ref{exJthm} implies that the attaching map for the top cell of M is inert.
\end{example}

\begin{example}\label{BTex2}
Let $M$ be an $(n-1)$-connected $(2n+1)$-dimensional Poincar\'{e} duality complex with $n\geq 2$. It is clear that 
\[
H^n(M)\cong \bigoplus_{d}\mathbb{Z}
\]
for some integer $d\geq 0$. Suppose that $d\geq 1$. By the argument of \cite[Example 4.4]{BT22}, it can be shown that there is a homotopy cofibration
\[
S^{2n}\stackrel{}{\longrightarrow}  S^n\vee S^{n+1}\vee \Sigma X\simeq M_0\stackrel{}{\longrightarrow} M
\]
for some complex $X$ and $M$ is a Beben-Theriault complex. Then Theorem \ref{exJthm} implies that the attaching map for the top cell of M is inert.
\end{example}
%--------------------------------------------------------------------------------------%
\subsection{Connected sum}
\label{subsec: sum}

$\, $

In this subsection, we reproduce a result of Theriault \cite{The24a} by our method.

\begin{theorem}{\cite[Theorem 1.4]{The24a}}\label{exsumthm}
Let $M$ and $N$ be two $n$-dimensional Poincar\'{e} duality complexes with a single top cell. Suppose that the attaching map for the top cell of $M$ is inert. Then the following hold:
\begin{itemize}
\item[(1).] the attaching map for the top cell of $M\# N$ is inert;
\item[(2).] there is a homotopy equivalence
\[
\Omega (M\# N) \simeq \Omega M\times \Omega(\Omega M\ltimes N_0).
\]
\end{itemize}
\end{theorem}
\begin{proof}
From the definition of connected sum, there is a canonical homotopy cofibration diagram
\[
\diagram 
 & S^{n-1} \rdouble \dto^{h} & S^{n-1} \dto^{h_M} \\
      N_0\rto^-{}\ddouble & M_0\vee N_0\rto^-{q}\dto^{} & M_0\dto^{}  \\ 
      N_0\rto^-{} & M\# N\rto^-{p} &   M,
  \enddiagram
\]
where $h$ and $h_M$ are the attaching maps for the top cells, and $q$ and $p$ are the projection maps. Since $h_M=q\circ h$ is inert by assumption, the theorem follows immediately from Proposition \ref{M0=A+Bprop}.
\end{proof}
The following corollary follows from Theorem \ref{exsumthm} and Lemma \ref{pdt-inert-lemma} immediately.
\begin{corollary}\label{SSsumcor}
Let $n>m>0$. 
Let $N$ be an $n$-dimensional Poincar\'{e} duality complex with a single top cell. 
Then the following hold:
\begin{itemize}
\item[(1).] the attaching map for the top cell of $(S^{m}\times S^{n-m})\# N$ is inert;
\item[(2).] there is a homotopy equivalence
\[\hspace{2.5cm}
\Omega ((S^{m}\times S^{n-m})\# N) \simeq \Omega S^{m}\times  \Omega  S^{n-m}\times \Omega(\Omega (S^{m}\times S^{n-m})\ltimes N_0).   \hspace{2.5cm}\Box
\]  
\end{itemize}
\end{corollary}

%--------------------------------------------------------------------------------------%
\subsection{Low dimensional manifolds}
\label{subsec: lowfolds}

$\, $

The results in this section are useful to study inertness for orientable low dimensional manifolds. In Examples \ref{BTex1} and \ref{BTex2} we have seen that the top cell attachments for most simply connected $4$- and $5$-manifolds are inert. Here, we study the inertness property for $2$-, $3$- and $6$-dimensional manifolds. 

We start with a general observation. Recall a closed manifold $M$ is called {\it aspherical} if it is path connected and all its higher homotopy groups vanish. In other word, the manifold $M$ is a Eilenberg-MacLane space 
\[
M\simeq K(\pi_1(M), 1).
\] 
\begin{lemma}\label{asph-inert-lemma}
Let $M$ be an $n$-dimensional orientable closed manifold with $n\geq 2$. If $M$ is aspherical, then the attaching map for the top cell of $M$ is inert.
\end{lemma}
\begin{proof}
Consider the lower skeleton inclusion $i_M: M_0\stackrel{}{\longrightarrow} M$. Since $n\geq 2$, it is clear that 
\[
i_{M\ast}: \pi_1(M_0)\stackrel{}{\longrightarrow} \pi_1(M)
\]
is surjective. The looped map $\Omega i_M: \Omega M_0\stackrel{}{\longrightarrow} \Omega M\simeq K(\pi_1(M), 0)$ is homotopic to the composite
\[
\Omega M_0\stackrel{}{\longrightarrow}\pi_0(\Omega M_0)\stackrel{(\Omega i_M)_\ast}{\longrightarrow}\pi_0(\Omega M),
\]
where the first map is defined by sending each point to its path component, and $\pi_0(\Omega M_0)$ and $\pi_0(\Omega M)$ are endowed with discrete topology. Since $i_{M\ast}$ is surjective, this composite has a right homotopy inverse, and then so does $\Omega i_M$. This shows that the attaching map for the top cell of $M$ is inert.
\end{proof}

It is well-known that the smooth category and topology category of $n$-manifolds are equivalent for $n\leq 3$. Hence we will not distinguish them in the sequel. 
When $n=2$, since an orientable closed surface of positive genus is aspherical, the attaching map for its top cell is inert. When $n=3$, recall an orientable $3$-manifold $M$ is {\it irreducible} if any embedded $2$-sphere $S^2$ in $M$ bounds a $3$-ball $D^3$, while $M$ is {\it prime} if it can not be decomposed as a nontrivial connected sum of two manifolds, that is, if $M\cong M_1\# M_2$ then $M_1$ or $M_2$ is the $3$-sphere $S^3$. An irreducible manifold is clearly prime and the converse is almost true except that $S^1\times S^2$ is prime but not irreducible. 
There is the fundamental {\it prime decomposition theorem} for $3$-manifolds.
\begin{theorem}\label{pdecthm}
Every orientable closed $3$-manifold is a connected sum of finitely many $3$-manifolds that are either $S^1\times S^2$ or irreducible. Moreover, the connected summands are unique up to the ordering and orientation preserving homeomorphism. $\qqed$
\end{theorem}
From Theorem \ref{pdecthm} it is known that for the $3$-manifold $M$ its fundamental group is a free product
\[
\pi_1(M)\cong F \ast G_1\ast \cdots \ast G_k,
\]
where $F$ is a free group of certain rank and each $G_i$ is the fundamental group of an irreducible $3$-manifold. Conversely, by Stallings' solution \cite{Sta71} on the Kneser conjecture, any free product decomposition of the fundamental group $\pi_1(M)$ can be realized as a connected sum decomposition of $M$.

We also need the following lemma which is a corollary of the famous sphere theorem of Papakyriakopoulos \cite{Pap57}.
\begin{lemma}\label{asphlemma}
Let $M$ be an irreducible $3$-manifold with infinite fundamental group. Then $M$ is aspherical. ~$\qqed$
\end{lemma}
We can now prove the following result on the inertness of the top cell attachments for $3$-manifolds. 
\begin{theorem}\label{3inertthm}
Let $M$ be an orientable closed $3$-manifold. If the fundamental group of $M$ is not isomorphic to a free product of finite groups, then the attaching map for the top cell of $M$ is inert.
\end{theorem}
\begin{proof}
Suppose that $M$ is prime. Then $\pi_1(M)$ can not be expressed as a nontrivial free product by Stallings' realization theorem. Therefore, the assumption implies that $\pi_1(M)$ is an infinite group. If $M\cong S^1\times S^2$, then the attaching map for the top cell of $M$ is inert by Lemma \ref{pdt-inert-lemma}. Otherwise, $M$ is irreducible with infinite fundamental group and hence is aspherical by Lemma \ref{asphlemma}. It follows from Lemma \ref{asph-inert-lemma} that the attaching map for the top cell of $M$ is inert. This proves the theorem when $M$ is prime.

Otherwise $M$ is not prime. Then by Theorem \ref{pdecthm} there is a prime decomposition 
\[
M\cong M_1\#\cdots \# M_k,
\]
where each $M_i$ is prime. If the fundamental group of each $M_i$ is a free product of finite groups, then so is the fundamental group of $M$, contradicting the assumption that $M$ is not a free product of finite groups. Hence, there exists a connected summand $M_i$ such that $\pi_1(M_i)$ is not a free product of finite groups. The discussion in the first paragraph then implies that the attaching map for the top cell of $M_i$ is inert. Accordingly, Theorem \ref{exsumthm} implies that the attaching map for the top cell of $M$ is inert. This completes the proof of theorem. 
\end{proof}

Theorem \ref{3inertthm} indicates that most $3$-manifolds satisfy the inertness property. Indeed, by Lemma \ref{asphlemma} an irreducible $3$-manifold $M$ is not aspherical only when $\pi_1(M)$ is finite. In this case, the universal covering of $M$ is a homotopy $3$-sphere, hence is the $3$-sphere $S^3$ by Perelman's solution to the Poincar\'{e} conjecture \cite{Per02, Per03a, Per03b}. Indeed, such a $3$-manifold $M$ is a quotient $S^3/\Gamma$ with $\Gamma$ a finite subgroup of $SO(4)$, and is called a {\it spherical $3$-manifold}. 
It is clear that the attaching map for the top cell of $S^3$ is not inert. Then Proposition \ref{coverprop} implies that the attaching map for the top cell of a spherical $3$-manifold is not inert as well.

Let us turn to $6$-manifolds. 
The following well-known splitting theorem for $6$-manifolds was proved by Wall \cite{Wal66} in the smooth category, while Jupp \cite{Jup73} pointed out that the theorem holds in the topological category by the same argument.
\begin{theorem}\cite[Theorem 1]{Wal66}\label{wallsplitthm}
Let $M$ be a simply connected closed $6$-manifold with third Betti number $b_3(M)=2m$. Then there exists a $6$-manifold $M_1$ such that
\[
 \hspace{5.8cm}
M\cong M_1 \# \mathop{\#}\limits_{m} (S^3\times S^3).
 \hspace{5.8cm}\Box
\]
\end{theorem}
With Theorem \ref{wallsplitthm}, Corollary \ref{SSsumcor} implies the following proposition immediately.
\begin{proposition}\label{wallinertprop}
Let $M$ be a simply connected closed $6$-manifold. If $H^3(M;\mathbb{Q})\neq 0$, then the attaching map for the top cell of $M$ is inert.   $\qqed$ 
\end{proposition}

\newpage
%----------------------------------------------------------------------------------------------------------------------------------------------------------------------------------------------------------%
\section{Two geometric comparison theorems}
\label{sec: det2}
In the previous sections, we studied inertness by comparing two Poincar\'{e} duality complexes of the same dimension through a nonzero degree map. We now extend this comparison idea further. 

In this section, we consider two geometric contexts of inert cell attachments and prove Part (1) of Theorems \ref{FEBthmintro} and \ref{geoinertthmintro}. The first context concerns strict fibrations of Poincar\'{e} duality complexes, while the second focuses on fibre bundles of manifolds. 
Both contexts demonstrate approaches to studying inertness by comparing Poincar\'{e} duality complexes of different dimensions through a fibration. 

%----------------------------
\subsection{Strict fibrations}
\label{subsec: fibsec}

$\, $

Let 
\begin{equation}\label{FEBeq}
F\stackrel{j}{\longrightarrow} E\stackrel{\pi}{\longrightarrow} B
\end{equation}
 be a strict fibration of connected Poincar\'{e} duality complexes. Suppose that both $B$ and $F$ are of positive dimensions and have a single top cell. In general, denote by $e(X)$ the top closed cell of a $CW$-complex $X$ with a single top cell. 
 Then the product $e(F)\times e(B)$ of the two top cells of $F$ and $B$ is the single top cell $e(E)$ of $E$. Therefore, in this situation we can talk about top cell attachments for the three Poincar\'{e} duality complexes in the fibration (\ref{FEBeq}). 

In general, for an $n$-dimensional Poincar\'{e} duality complex $X$ with a single top cell, denote by $X_0$ its $(n-1)$-skeleton obtained by deleting the top open cell $\accentset{\circ}{e}(X)$ of $X$, and by
\[
S^{n-1}\stackrel{h_X}{\longrightarrow} X_0\stackrel{i_X}{\longrightarrow} X
\] 
the homotopy cofibration associated to the top cell attachment of $X$. 

Restricting the fibration (\ref{FEBeq}) to the lower skeleton $B_0$ of $B$ gives a strictly commutative diagram of fibrations
\begin{equation}\label{EFBB0diag}
\diagram
F\ddouble  \rto^{j|_{0}}   & E|_{0}  \dto^{i|_{0}}  \rto^{ \pi|_{0}}  & B_0 \dto^{i_B}  \\
F              \rto^{j}   & E                           \rto^{\pi}       & B,
\enddiagram
\end{equation}
where $E|_{0}$ is the pullback with the structural maps $\pi|_{0}$ and $i|_{0}$, and $j|_{0}$ is the induced map.

The following theorem is Part (1) of Theorem \ref{FEBthmintro}.

\begin{theorem}\label{FEBthm}
Let 
\[\label{FEBeqintro}
F\stackrel{}{\longrightarrow} E\stackrel{\pi}{\longrightarrow} B
\]
 be a strict fibration of connected Poincar\'{e} duality complexes with a single top cell. 
 If the attaching map for the top cell of $B$ is inert, then the attaching map for the top cell of $E$ is inert.
 
Additionally, if the attaching map for the top cell of $B$ is inert after localization away from a prime $p$, then the attaching map for the top cell of $E$ is inert after localization away from a prime $p$.
 \end{theorem}
\begin{proof}
If $F$ is contractible, the theorem is trivial. Suppose that $F$ is of positive dimension. We follow the notations and constructions in this subsection. Consider Diagram (\ref{EFBB0diag}). Since the fibration $\pi|_{0}$ is the restriction of the fibration $\pi$ on the lower skeletons, it is clear that $E|_{0}=\pi^{-1}(B_0)$ and the map $i|_{0}: E|_{0}\stackrel{}{\longrightarrow} E$ is an inclusion. Also, the restriction of the fibration $\pi$ over the top cell $e(B)$ of $B$ is trivial; that is, it is fibrewisely isomorphic to $e(B)\times F$. Therefore, the total complex $E$ is the union of $E|_{0}$ and $e(B)\times F$ along their boundaries through a homeomorphism $\partial E|_{0}\cong \partial e(B)\times F$. In particular, as the top cell $e(E)=e(B)\times e(F)\subseteq e(B)\times F$, the inclusion $i|_{0}: E|_{0}\stackrel{}{\longrightarrow} E$ factors as a composition
\[
E|_{0}\stackrel{\mathfrak{i}|_0}{\longrightarrow} E_0 \stackrel{i_E}{\longrightarrow} E,
\] 
for an inclusion map $\mathfrak{i}|_0$. 

Consider the integral case. By assumption the attaching map for the top cell of $B$ is inert, that is, $\Omega i_B$ has a right homotopy inverse. 
Since the right square of Diagram (\ref{EFBB0diag}) is a pullback and $\pi$ is a fibration, the right square is also a homotopy pullback, so is the loop of the right square. 
By Theorem \ref{pullbackthm} (1), the assumption that $\Omega i_B$ has a right homotopy inverse implies that $\Omega i|_{0}$ also has a right homotopy inverse, say $s: \Omega E\stackrel{}{\longrightarrow} \Omega E|_{0}$. Then $1_{\Omega E}\simeq \Omega i|_{0}\circ  s\simeq  \Omega i_E\circ \Omega \mathfrak{i}|_0\circ  s$. Therefore $\Omega \mathfrak{i}|_0\circ  s: \Omega E\stackrel{}{\longrightarrow} \Omega E_0$ is a right homotopy inverse of $\Omega E_0 \stackrel{\Omega i_E}{\longrightarrow} \Omega E$, which means that the attaching map for the top cell of $E$ is inert.

The statement for the local case follows by the same argument. 
\end{proof}

\begin{example}\label{Huafibex}
Let $n\geq 2$. 
Let $N$ be an $(n-1)$-connected $2n$-dimensional Poincar\'{e} duality complex with $H^n(N;\mathbb{Z})\cong \mathbb{Z}^{\oplus d}$ and $d\geq 2$, or an $(n-1)$-connected $(2n+1)$-dimensional Poincar\'{e} duality complex with $H^n(N;\mathbb{Z})\cong \mathbb{Z}^{\oplus d}$ and $d\geq 1$. 
By Example \ref{BTex1} and \ref{BTex2}, the attaching map for the top cell of $N$ is inert. 
Let 
\[
S^{k-1}\stackrel{}{\longrightarrow} M\stackrel{\pi}{\longrightarrow} N
\]
be the sphere bundle of a rank $k$ real vector bundle over $N$ with $k\geq 2$. Then Theorem \ref{FEBthm} implies that the attaching map for the top cell of $M$ is inert.
\end{example}

\begin{example}\label{DVex}
Let $M$ be a simply connected closed smooth manifold such that it can be decomposed into a union of two same disk bundles along their boundaries
\[
M\cong D\xi\cup_g D\xi,
\]
where $\xi$ is a rank $k$ real vector bundle over a simply connected closed manifolds $B$, and $S\xi\stackrel{g}{\rightarrow} S\xi$ is a self diffeomorphism of the sphere bundle of $\xi$. The manifold $M$ is so-called a {\it twisted double} and also a {\it manifold of focal genus $2$} following \cite{Dua11}. 

Let $\mathcal{G}(\xi)$ be the gauge group of the principal bundle associated to the vector bundle $\xi$. If $g\in \mathcal{G}(\xi)$, the decomposition of $M$ implies a spherical fibre bundle
\[
S^k\stackrel{}{\longrightarrow} M\stackrel{\pi}{\longrightarrow} B
\]
where $\pi$ is the union of the bundle projections of the two copies of $D\xi$ through $g$. By Theorem \ref{FEBthm}, if the attaching map for the top cell of $B$ is inert, then the attaching map for the top cell of $M$ is inert.

We remark that, in differential geometry, there are interesting examples of focal genus $2$ manifolds that appear in the study of Dupin hypersurfaces \cite{Tho83}, including isoparametric hypersurfaces \cite{Mun80, Heb81} and the principal orbits of strict cohomogeneity one actions \cite{Mos57}. 
In algebraic topology, a focal genus $2$ manifold is a special case of a double mapping cylinder, the rational homotopy of which was thoroughly studied by Grove and Halperin \cite{GH87}.
\end{example}

%----------------------------
\subsection{Spherical fibre bundles}
\label{subsec: Sfib}

$\, $

The converse statement of Theorem \ref{FEBthm} can be proved for special fibre bundles. 
Let $k=2, 4$, or $8$.  
Let 
\begin{equation}\label{SNMeq}
S^{k-1}\stackrel{j}{\longrightarrow} N\stackrel{p}{\longrightarrow} M
\end{equation}
be a fibre bundle of connected oriented closed smooth manifolds. Denote ${\rm dim}(M)=n$. 
Let $M_0$ be the manifold $M$ with a small open disk removed. By Morse theory and Poincar\'{e} duality, a connected oriented closed smooth manifold is homotopy equivalent to a Poincar\'{e} duality complex with a single top cell, and the deleted manifold $M_0$ is homotopy equivalent to the $(n-1)$-skeleton of $M$. This justifies the choice of notation.  

The following lemma borrows the argument of \cite[Lemma 3.1]{HT23}. 
Recall that by standard results in differential topology, the choice of the removed open disk is irrelevant for defining the deleted manifold $M_0$, and the same notation $M_0$ will be used for all specific choices of the delete manifold. Similarly for the deleted manifold $N_0$. 

In particular, since the inclusion $p^{-1}(M_0)\stackrel{i}{\hookrightarrow} N$ of the restriction of the fibre bundle (\ref{SNMeq}) on $M_0$ is not surjective, it factors through the deleted manifold $N_0$:
\[
p^{-1}(M_0)\stackrel{i_0}{\hookrightarrow}N_0\stackrel{i_N}{\longrightarrow} N,
\]
where $i_0$ and $i_N$ are the inclusion maps.
\begin{lemma}\label{emblemma}
If $j$ is null homotopic and $n\geq k+2$. Then the following hold:
\begin{itemize}
\item[(1).]
the fibre inclusion $j$ can be extended to an embedding $D^k\stackrel{}{\hookrightarrow} N$ such that up to isotopy it factors through $p^{-1}(M_0)$:
\[
D^k\stackrel{}{\hookrightarrow}  p^{-1}(M_0)\stackrel{i}{\hookrightarrow} N;
\]
\item[(2).]
there is a smooth map $N_0\stackrel{\kappa}{\longrightarrow} p^{-1}(M_0)$ such that the composite
\[
N_0\stackrel{\kappa}{\longrightarrow} p^{-1}(M_0)\stackrel{i_0}{\longrightarrow} N_0
\]
is isotopic to the identity map.
\end{itemize}
\end{lemma}
\begin{proof}
Statement (1) will be proved in two parts. 

(1a). Since $S^{k-1}\stackrel{j}{\longrightarrow} N$ is null homotopic, it can be extended to a map $D^k\stackrel{J}{\longrightarrow} N$. Then the condition $n\geq k+2$, or equivalently $n+k-1\geq 2k+1$, implies that $J$ is homotopic to a smooth embedding $D^k\stackrel{}{\hookrightarrow} N$ relative to the boundary by the classical Whitney embedding theorem. 

With this embedding we may first prove statement (2) and then prove the rest part of statement (1).

(2). By the argument in (1a), a fibre inclusion $S^{k-1}\hookrightarrow N$ over a point $x\in M$ can be extended to a disk embedding $D^k\stackrel{}{\hookrightarrow}N$. As in the proof of \cite[Lemma 3.1]{HT23}, a key step is to prove that the embedding $D^k\stackrel{}{\hookrightarrow} N$ is local, that is, there is a diffeomorphism 
\[
N\cong N\# S^{n+k-1}
\] 
such that up to isotopy $D^k$ embeds into the $S^{n+k-1}$-factor of $N\# S^{n+k-1}$. The proof is similar, and we include it here for completeness.  

Choose a small closed disk neighborhood $x \in D(x) \subseteq M$ with $(D^n, 0)\cong (D(x), x)$. The restriction of the fibre bundle $p$ (\ref{SNMeq}) on $D(x)$ is trivial, that is, there is a fibrewise diffeomorphism
\[
D^n\times  S^{k-1}\stackrel{\cong}{\longrightarrow} p^{-1}(D(x)),
\]
which restricts to $\{0\}\times S^{k-1}\cong p^{-1}(x)$ on the core spheres. Then the total manifold $N$ can be fibrewisely decomposed as
\[
N\cong p^{-1}(M_0)\cup p^{-1}(D(x))\cong p^{-1}(M_0)\cup (D^n\times  S^{k-1}),
\]
where $M_0=M-\accentset{\circ}{D}(x)$.

As depicted in Figure \ref{torusfigure}, we may shrink the embedded disk $D^k\hookrightarrow N$ such that the image of its boundary sphere $S^{k-1}$ lies on the boundary of the image of the torus embedding  $D^n\times S^{k-1}\stackrel{\cong}{\longrightarrow} p^{-1}(D(x))\subseteq N$. It is a fact that the diffeomorphism
\[
(S^{n-1}\times D^k)\cup_{S^{n-1}\times S^{k-1}}(D^n\times S^{k-1})\cong S^{n+k-1}
\]
restricts to 
\[
(D^{n-1}\times D^k)\cup_{S^{n-1}\times S^{k-1}}(D^n\times S^{k-1})\cong D^{n+k-1}.
\]
Therefore, we can further thicken the modified embedding $D^k\hookrightarrow N$ to an embedding $D^{n-1}\times D^k \hookrightarrow N$ so that its union with the embedding  $D^n\times  S^{k-1}\stackrel{\cong}{\longrightarrow} p^{-1}(D(x))\subseteq N$ gives an embedding $D^{n+k-1}\stackrel{}{\hookrightarrow} N$.

To summarize, there is a series of diffeomorphisms 
\[
\begin{split}
N&\cong (N-\accentset{\circ}{D}^{n+k-1})\cup D^{n+k-1}\\
&\cong N\# S^{n+k-1} \\ 
&\cong N\# (S^{n-1}\times D^k)\cup_{S^{n-1}\times S^{k-1}}(D^n\times S^{k-1}),
\end{split}
\] 
such that the disk embedding of $D^k$ is isotopic to the embedding $D^k\stackrel{}{\hookrightarrow}S^{n-1}\times D^k\subseteq N\# S^{n+k-1}$ into the second factor, and $D^n\times S^{k-1}\subseteq N\# S^{n+k-1}$ is the restriction of the fibre bundle $N\# S^{n+k-1}\stackrel{p}{\longrightarrow} M$ over a small disk of $M$. Therefore, there is a series of diffeomorphisms 
\begin{equation}\label{toruseq}
\begin{split}
p^{-1}(M_0)%&\cong N-(\accentset{\circ}{D}^{n}\times S^{k-1})\\
& \cong N\# S^{n+k-1} -(\accentset{\circ}{D}^{n}\times S^{k-1})\\
& \cong N\# (S^{n-1}\times D^{k}).
\end{split}
\end{equation}
In particular, the canonical inclusion map $N_0\stackrel{}{\hookrightarrow}N\# (S^{n-1}\times D^{k})$ gives a map $\kappa: N_0\stackrel{}{\hookrightarrow}N\# (S^{n-1}\times D^{k})\stackrel{\cong}{\longrightarrow}p^{-1}(M_0)$. 
Recall that the inclusion map $p^{-1}(M_0)\stackrel{i}{\hookrightarrow} N$ factors through an inclusion $p^{-1}(M_0)\stackrel{i_0}{\hookrightarrow} N_0$, where $N_0$ should be viewed as the manifold $N$ with a small open disk $\accentset{\circ}{D}^{n+k-1}\subseteq \accentset{\circ}{D}^{n}\times S^{k-1}$ removed. Then it is clear that image of the composite
\[
i_0\circ \kappa: N_0\stackrel{}{\hookrightarrow}N\# (S^{n-1}\times D^{k})\cong p^{-1}(M_0)\stackrel{i_0}{\hookrightarrow} N_0
\]
can be obtained by pushing the boundary $S^{n+k-1}$ of the codomain $N_0$ inwards along the collar of $S^{n+k-1}\subseteq N_0$, and hence is isotopic to the identity map. 

\begin{figure}[!htb]
\centering
\includegraphics[width=2.7in]{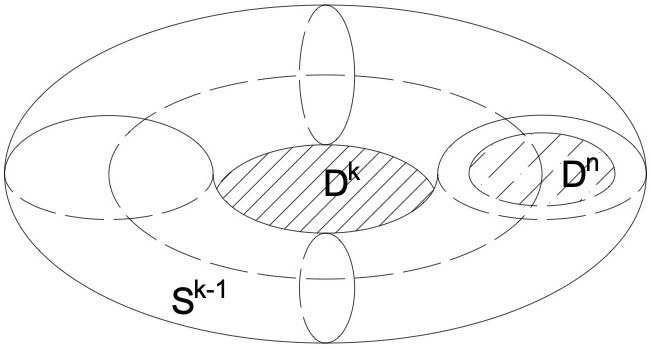}
\caption{The core sphere of $D^n\times S^{k-1}\hookrightarrow N$ bounds a disk;  \cite[Figure1]{HT23}}\label{torusfigure}
\end{figure}

(1b). Let us return to prove statement (1). Recall up to isotopy the embedded disk 
\[
D^k\stackrel{}{\hookrightarrow} N\# S^{n+k-1}\cong N\# (S^{n-1}\times D^k)\cup_{S^{n-1}\times S^{k-1}}(D^n\times S^{k-1}),
\]
in the proof of statement (2), lies in the second factor of $S^{n-1}\times D^k$. Along the identifications (\ref{toruseq}), this implies that the embedding of the disk $D^k$ factors through $p^{-1}(M_0)$ up to isotopy. 
Indeed, we can shrink the embedded disk further so that it contained entirely in the interior of the $(S^{n-1}\times D^{k})$-factor of $N\# (S^{n-1}\times D^{k})\cong p^{-1}(M_0)$, as displayed in Figure \ref{torusfigure}. 
\end{proof}

\begin{remark}\label{krmk}
As in \cite[Remark 3.2]{HT23}, under the condition of Lemma \ref{emblemma} the connecting map of the fibre bundle (\ref{SNMeq}) is null homotopic, and then there is a homotopy equivalence
\[
\Omega M\simeq S^{k-1}\times \Omega N.
\]
In particular, $S^{k-1}$ is an $H$-space, and hence $k=2$, $4$ or $8$ by Adams's solution to the famous Hopf invariant one problem \cite{Ada60}. This justifies the choice of $k$ at the beginning of the subsection. 
\end{remark}

The following theorem is Part (1) of Theorem \ref{geoinertthmintro} and is a geometric version of \cite[Theorem 1.2]{The24b} with partial generalizations. 

\begin{theorem}\label{geoinertthm}
Let $k=2$, $4$, or $8$. Let 
\begin{equation}\label{SNMeqintro}
S^{k-1}\stackrel{j}{\longrightarrow} N\stackrel{}{\longrightarrow} M
\end{equation}
be a fibre bundle of connected oriented closed smooth manifolds such that the fibre inclusion $j$ is null homotopic and ${\rm dim}(M)\geq k+2$. Then the attaching map for the top cell of $N$ is inert if and only if the attaching map for the top cell of $M$ is inert.

Additionally, the assertion holds after localization at any set of primes.
\end{theorem}
\begin{proof}
Let $M_0\stackrel{i_M}{\longrightarrow} M$ and $N_0\stackrel{i_N}{\longrightarrow} N$ be the inclusions of the deleted manifolds. Recall that the canonical inclusion $p^{-1}(M_0)\stackrel{i}{\longrightarrow} N$ factors as
\[
p^{-1}(M_0)\stackrel{i_0}{\longrightarrow} N_0 \stackrel{i_N}{\longrightarrow} N
\]
for an inclusion map $i_0$. 
By Lemma \ref{emblemma} (2), there is a smooth map $N_0\stackrel{\kappa}{\longrightarrow} p^{-1}(M_0)$ such that the composite 
\[
N_0\stackrel{\kappa}{\longrightarrow} p^{-1}(M_0)\stackrel{i_0}{\longrightarrow} N_0 \stackrel{i_N}{\longrightarrow} N
\] 
is homotopic to the inclusion $i_N$. If $\Omega i_N$ has a right homotopy inverse $s$, then $\Omega i\circ \Omega  \kappa\circ s\simeq \Omega i_N\circ \Omega i_0 \circ \Omega \kappa\circ s\simeq \Omega i_N\circ s\simeq 1_{\Omega N}$, that is, $\Omega\kappa\circ s$ is a right homotopy inverse of $\Omega i$. Conversely, if $\Omega i$ has a right homotopy inverse $t$, then $\Omega i_N\circ \Omega i_0 \circ t= \Omega i\circ t\simeq 1_{\Omega N}$, that is, $\Omega i_0\circ t$ is a right homotopy inverse of $\Omega i_N$. Therefore, $\Omega i_N$ has a right homotopy inverse if and only if $\Omega i$ has a right homotopy inverse.

Further, consider the pullback of fibre bundles
\begin{equation}\label{SNM0diag}
\begin{aligned}
\xymatrix{
S^{k-1} \ar@{=}[d] \ar[r]^<<<<{j_0} & p^{-1}(M_0) \ar[r]^<<<{p} \ar[d]^{i} & M_0 \ar[d]^{i_M} \\ 
S^{k-1}\ar[r]^{j}                   &            N               \ar[r]^{p}    &  M ,                
}
\end{aligned}
\end{equation}
where $j_0$ is the induced fibre inclusion, and the right square is also a homotopy pullback. By Lemma \ref{emblemma} (1) $j_0$ is null homotopic, and therefore $\Omega p$ admits a left homotopy inverse. Then applying Theorem \ref{pullbackthm} (3) to the right square of (\ref{SNM0diag}), we see that $\Omega i_M$ has a right homotopy inverse if and only if $\Omega i$ has a right homotopy inverse.

Combining the above together, it follows that $\Omega i_N$ has a right homotopy inverse if and only if $\Omega i$ has a right homotopy inverse, and hence if and only if $\Omega i_M$ has a right homotopy inverse. This proves the integral conclusion in the theorem. 

The same proof can show that the conclusion holds after localization at any set of primes.
\end{proof}

In Examples \ref{BTex1} and \ref{BTex2} we have seen that most highly connected closed manifolds satisfy the inertness property. Here we can apply Theorem \ref{geoinertthm} to give an alternative proof for certain highly connected closed smooth manifolds. More essential examples will be given in Sections \ref{sec: surgery} and \ref{sec: flag}. 

\begin{example}\label{DLex}
Let $M$ be a simply connected $4$-manifold such that $H_2(M;\mathbb{Z})\cong \mathbb{Z}^{\oplus d}$ with $d\geq 2$. By \cite[Corollary 1]{DL05} there is a principal bundle
\[
S^1\stackrel{}{\longrightarrow} \mathop{\#}\limits_{d-1} (S^2\times S^3) \stackrel{}{\longrightarrow} M.
\]
By Corollary \ref{SSsumcor} the attaching map for the top cell of $\mathop{\#}\limits_{d-1} (S^2\times S^3)$ is inert. 
It is clear that the fibre inclusion is null homotopic by connectivity. Therefore, Theorem \ref{geoinertthm} implies that the attaching map for the top cell of $M$ is inert.
\end{example}

\begin{example}\label{BGSex}
Let $M$ be a $3$-connected $8$-manifold such that $H_4(M;\mathbb{Z})\cong \mathbb{Z}^{\oplus d}$ with $d\geq 2$. If the attaching map for the top cell of $M$ is stably trivial, by \cite[Theorem B]{BGS24} there is a principal bundle
\[
S^3\stackrel{}{\longrightarrow} N \stackrel{}{\longrightarrow} M
\]
such that $N$ is homeomorphic to $\mathop{\#}\limits_{d-1} (S^4\times S^7)$. 
By Corollary \ref{SSsumcor} the attaching map for the top cell of $\mathop{\#}\limits_{d-1} (S^4\times S^7)$ is inert, and then so is the attaching map for the top cell of $N$. 
It is clear that the fibre inclusion is null homotopic by connectivity. Therefore, Theorem \ref{geoinertthm} implies that the attaching map for the top cell of $M$ is inert.
\end{example}

\newpage
%----------------------------------------------------------------------------------------------------------------------------------------------------------------------------------------------------------%
\section{A specific surgery}
\label{sec: surgery}
In this section, we explore inertness around a specific surgery. The result is an application of Theorem \ref{geoinertthm}. It illustrates a method of constructing new examples of manifolds with inert top cell attachments via surgery. 

Let $N$ be an $n$-dimensional connected oriented closed smooth manifold. To define a surgery on $N\times S^{k-1}$ with $k\geq 2$, let 
\[
\tau: S^{k-1}\stackrel{}{\longrightarrow} SO(n)
\]
be a map. 
It determines a self diffeomorphism $
f_\tau: D^n \times S^{k-1}\stackrel{}{\longrightarrow}D^n \times S^{k-1}
$ 
by $f_\tau (a, t)=(\tau(t)a, t)$, where the special orthogonal group $SO(n)$ acts canonically on the disk $D^n$ by matrix multiplication. The composite 
\begin{equation}\label{ftaueq}
\bar{f}_\tau: D^n \times S^{k-1}\stackrel{f_\tau}{\longrightarrow}D^n \times S^{k-1}\stackrel{i\times 1_{S^{k-1}}}{\longrightarrow}N\times S^{k-1}
\end{equation}
is called a {\it framed $(k-1)$-embedding}, where $D^n\stackrel{i}{\hookrightarrow} N$ is an embedding of a small disk. A {\it $(k-1)$-surgery} on the product manifold $N\times S^{k-1}$ along the factor $S^{k-1}=\ast\times S^{k-1}$ with {\it the framing} $\tau$ is an operation removing the framed $(k-1)$-embedding $\bar{f}_\tau$ and replacing it with $S^{n-1}\times D^k$, with {\it effect} the $(n+k-1)$-dimensional connected closed smooth manifold
\[
\mathcal{G}^{\tau}(N):=\big((N\times S^{k-1})\backslash f_\tau(\accentset{\circ}{D}^{n}\times  S^{k-1})\big)\cup_{(S^{n-1}\times S^{k-1})} (S^{n-1}\times D^{k}).
\]
Since $N$ is oriented and $\tau$ is valued in $SO(n)$, the effect $\mathcal{G}^{\tau}(N)$ inherits an orientation from that of $N\times S^{k-1}$ through the surgery. When $\tau$ is trivial, we write $\tau=0$ and denote the manifold $\mathcal{G}^{\tau}(N)$ by $\mathcal{G}^{0}(N)$. 

Before exploring the inertness of the top cell attachment of $\mathcal{G}^{\tau}(N)$, we introduce some basic facts about $\mathcal{G}^{\tau}(N)$ and then study its local homotopy type. 
 
By abuse of notation, denote by 
\[
f_\tau: S^{n-1}\times S^{k-1}\stackrel{}{\longrightarrow} S^{n-1}\times S^{k-1}
\]
the restriction of $f_\tau$ (\ref{ftaueq}) on the boundary of $D^n\times S^{k-1}$. Then $f_\tau$ is a diffeomorphism such that $f_\tau(a, t)=(\tau(t)a, t)$. Define the composition and inverse operations 
\[
\tau\circ \tau': S^{k-1}\stackrel{}{\longrightarrow} SO(n) \  \ {\rm and}  \   \ 
 \tau^{-1}: S^{k-1}\stackrel{}{\longrightarrow} SO(n)
\]
by $(\tau\circ \tau')(t)=\tau(t)\cdot \tau'(t)$ and $\tau^{-1}(t)= (\tau(t))^{-1}$, respectively. 
Then $f_{\tau\circ \tau'}=f_{\tau}\circ f_{\tau'}$, and $f_\tau\circ f_{\tau^{-1}}= f_{\tau^{-1}}\circ f_\tau=f_0$ is the identity map. 

The construction of $\mathcal{G}^{\tau}(N)$ implies the following lemma immediately. 
\begin{lemma}\label{gyration-pushout-lemma}
There is a homotopy pushout
\[
  \label{gyrationpo} 
  \diagram 
      S^{n-1}\times S^{k-1}\rto^-{\pi_1}\dto^{(i\times 1_{S^{k-1}})\circ f_\tau} &  S^{n-1}\dto \\ 
      N_0\times S^{k-1}\rto & \mathcal{G}^{\tau}(N), 
  \enddiagram 
\]
where $\pi_1$ is the standard projection, $N_0$ is the manifold $N$ with a small disk removed, and $S^{n-1}\stackrel{i}{\longrightarrow}N_0$ is the inclusion of the boundary. ~$\qqed$
\end{lemma}

To study the local homotopy type of $\mathcal{G}^{\tau}(N)$, we need the following lemma  which is contained in the proof of \cite[Lemma 2.1]{HT23}.  For an integer $k$, let $\mathcal{P}_{k}$ be a set of prime numbers such that
\[\label{pkdefeq}
\mathcal{P}_{k}= \left\{\begin{array}{cc}
\emptyset & k\equiv 3, 5, 6, 7 \ {\rm mod} \ 8, \\
\{2\} & k\equiv 1, 2\ {\rm mod} \ 8, k\neq 1, \\
\{ {\rm prime}~p~|~(p-1)~{\rm divides}~2s\} & k=4s.
\end{array}\right.
\]
Note that in \cite[Lemma 2.1]{HT23} the definition of $\mathcal{P}_k$ looks different for $k=4s$. However, the two definitions are equivalent by a classical fact about Bernoulli numbers; for instance, see \cite[Lemma 7.3]{HT24d} or \cite[Theorem B.3 and B.4]{MS74}.
\begin{lemma}\label{Jlemma}
After localization away from $\mathcal{P}_k$ there is a homotopy commutative diagram 
\[
\begin{aligned}
  \label{gammadgrm} 
  \xymatrix{ 
     S^{n-1}\times S^{k-1}\ar[r]^-{\pi_{1}}\ar[d]^{1\times\tau} & S^{n-1}\ar@{=}[d] \\ 
     S^{n-1}\times SO(n)\ar[r]^-{\theta} & S^{n-1},
     }
     \end{aligned}
\]
where $\theta$ is the standard action of $SO(n)$ on $\mathbb{R}^n$.  $\qqed$
\end{lemma}

The following proposition indicates that there is only one homotopy type for $\mathcal{G}^{\tau}(N)$ with various framings in the local category.

\begin{proposition}\label{localgrationprop}
Let $n\geq k+2\geq 4$. Suppose that $N$ is nilpotent. Then for any framing $\tau$ there is a homotopy equivalence
\[
\mathcal{G}^{\tau}(N)\simeq \mathcal{G}^{0}(N)
\]
after localization away from $\mathcal{P}_k$.
\end{proposition}
\begin{proof}
We work in the homotopy category after localization away from $\mathcal{P}_k$. 
Since 
\[
(\pi_1\circ f_\tau)(a,t)=\pi_1(\tau(t)a, t)=\tau(t)a=(\theta\circ (1\times \tau))(a,t)
\] 
for any $(a, t)\in S^{n-1}\times S^{k-1}$, we have $\pi_1\circ f_\tau= \theta\circ (1\times \tau) \simeq \pi_1$ by Lemma \ref{Jlemma}. Then there is a homotopy commutative diagram 
\begin{equation}\label{ftaudiag}
\diagram
N_0 \times S^{k-1} \ddouble   & S^{n-1}\times S^{k-1} \lto_{(i\times 1_{S^{k-1}})\circ f_\tau}  \rto^<<<{\pi_1}  \dto^{f_\tau} & S^{n-1} \ddouble \\
N_0 \times S^{k-1}  & S^{n-1}\times S^{k-1} \lto_{i\times 1_{S^{k-1}}}  \rto^<<<{\pi_1}  & S^{n-1}. 
\enddiagram
\end{equation}
Lemma \ref{gyration-pushout-lemma} implies that the homotopy pushouts of the first and second rows are the manifolds $\mathcal{G}^{\tau}(N)$ and $\mathcal{G}^{0}(N)$, respectively. 
By the universal property of homotopy pushout, the diagram induces a map of homotopy pushouts  
\[
\varphi: \mathcal{G}^{\tau}(N)\stackrel{}{\longrightarrow}\mathcal{G}^{0}(N).
\]

By a similar argument with the fact $f_{\tau}\circ f_{\tau^{-1}}=1_{S^{n-1}\times S^{k-1}}$,  we can construct a map of homotopy pushouts
\[
\psi: \mathcal{G}^{0}(N)\stackrel{}{\longrightarrow}\mathcal{G}^{\tau}(N).
\]
We want to show that the composites $\psi\circ \varphi$ and  $\varphi\circ \psi$ are homotopic to the identity maps. This technical part is divided into two steps. 

{\it (I). Give explicit constructions of $\varphi$ and $\psi$ by using double mapping cylinders.}

Denote by 
\[
\widetilde{\mathcal{G}}^\tau(N):=(N_0\times S^{k-1})\mathop{\cup}\limits_{(i\times 1_{S^{k-1}})\circ f_\tau} (S^{n-1}\times S^{k-1}\times [0,1])\mathop{\cup}\limits_{\pi_1} S^{n-1}
\]
the double mapping cylinder. It is a homotopy pushout of $(i\times 1_{S^{k-1}})\circ f_\tau$ and $\pi_1$, and then is homotopy equivalent to $\mathcal{G}^{\tau}(N)$ by Lemma \ref{gyration-pushout-lemma}. 

By $\pi_1\circ f_\tau\simeq \pi_1$ from Diagram (\ref{ftaudiag}), there exists a homotopy
\[
H: S^{n-1}\times S^{k-1}\times [\frac{1}{2}, 1] \stackrel{}{\longrightarrow} S^{n-1}
\]
such that $H(\frac{1}{2})=\pi_1\circ f_\tau$ and $H(1)=\pi_1$. Then it is straightforward to check that the maps
\[
\begin{split}
1_{N_0\times S^{k-1}}& :N_0\times S^{k-1}\stackrel{=}{\longrightarrow} N_0\times S^{k-1},   \\
F  & :      S^{n-1}\times S^{k-1}\times [0, \frac{1}{2}]\longrightarrow  S^{n-1}\times S^{k-1}\times [0, 1],  \ \ \    F(a, t, s)= (f_\tau(a, t), 2s)%=(\tau(t)a, t, 2s)         
      \\ 
H& :S^{n-1}\times S^{k-1}\times [\frac{1}{2}, 1] \stackrel{}{\longrightarrow} S^{n-1} \\ 
1_{S^{n-1}}& :S^{n-1} \stackrel{=}{\longrightarrow} S^{n-1}
\end{split}
\]
determine a map of double mapping cylinders 
\[
\widetilde{\varphi}: \widetilde{\mathcal{G}}^\tau(N)
\stackrel{}{\longrightarrow} \widetilde{\mathcal{G}}^0(N).
\]
Comparing the construction with Diagram (\ref{ftaudiag}), we see that $\widetilde{\varphi}$ is an explicit model of $\varphi$. 

Similarly, let $\nu: [\frac{1}{2}, 1]\stackrel{}{\longrightarrow}[\frac{1}{2}, 1]$ be the map defined by $\nu(s)=\frac{3}{2}-s$. Since $f_{\tau}\circ f_{\tau^{-1}}=1_{S^{n-1}\times S^{k-1}}$, the composite 
\[
\overline{H}: S^{n-1}\times S^{k-1}\times [\frac{1}{2}, 1] \stackrel{f_{\tau^{-1}}\times \nu}{\longrightarrow}S^{n-1}\times S^{k-1}\times [\frac{1}{2}, 1] \stackrel{H}{\longrightarrow} S^{n-1}
\]
satisfies that $\overline{H}(\frac{1}{2})=\pi_1\circ f_{\tau^{-1}}$ and $\overline{H}(1)= \pi_1$; that is, $\overline{H}$ is a homotopy from $\pi_1\circ f_{\tau^{-1}}$ to $\pi_1$. 
Then it is straightforward to check that the maps
\[
\begin{split}
1_{N_0\times S^{k-1}}& :N_0\times S^{k-1}\stackrel{=}{\longrightarrow} N_0\times S^{k-1},   \\
\overline{F}    & :       S^{n-1}\times S^{k-1}\times [0, \frac{1}{2}]\longrightarrow  S^{n-1}\times S^{k-1}\times [0, 1],  \ \ \    \overline{F}(a, t, s)= (f_{\tau^{-1}}(a, t), 2s)%=((\tau(t))^{-1}a, t, 2s)
\\ 
\overline{H}& :S^{n-1}\times S^{k-1}\times [\frac{1}{2}, 1] \stackrel{}{\longrightarrow} S^{n-1} \\ 
1_{S^{n-1}}& :S^{n-1} \stackrel{=}{\longrightarrow} S^{n-1}
\end{split}
\]
determine a map of double mapping cylinders 
\[
\widetilde{\psi}: \widetilde{\mathcal{G}}^0(N)
\stackrel{}{\longrightarrow} \widetilde{\mathcal{G}}^\tau (N),
\]
which is an explicit model of $\psi$. 

{\it (II). Show that the composites $ \widetilde{\psi}\circ  \widetilde{\varphi}$ and  $ \widetilde{\varphi}\circ  \widetilde{\psi}$ are homotopic to the identity maps.}

The composite $\widetilde{\mathcal{G}}^\tau(N)
\stackrel{\widetilde{\varphi}}{\longrightarrow} \widetilde{\mathcal{G}}^0(N) \stackrel{\widetilde{\psi}}{\longrightarrow} \widetilde{\mathcal{G}}^\tau (N)$ is determined by its restrictions on the components  
\[
\begin{split}
1_{N_0\times S^{k-1}}& :N_0\times S^{k-1}\stackrel{=}{\longrightarrow} N_0\times S^{k-1},   \\
\overline{F}\circ F& : S^{n-1}\times S^{k-1}\times [0, \frac{1}{4}]\longrightarrow  S^{n-1}\times S^{k-1}\times [0, 1], \ \ \   \\
\overline{H} \circ F   & :    S^{n-1}\times S^{k-1}\times [\frac{1}{4},\frac{1}{2}]\longrightarrow  S^{n-1},  \ \ \         \\ 
H& :S^{n-1}\times S^{k-1}\times [\frac{1}{2}, 1] \stackrel{}{\longrightarrow} S^{n-1} \\ 
1_{S^{n-1}}& : S^{n-1} \stackrel{=}{\longrightarrow} S^{n-1},
\end{split}
\]
where 
\[
\begin{split}
(\overline{F}\circ F)(a, t, s) &= (a, t, 4s), \\
(\overline{H} \circ F)(a,t,s) &=(H\circ (f_{\tau^{-1}}\times \nu)\circ F)(a,t,s)=H(a, t, \frac{3}{2}-2s).
\end{split}
\]
Note that the homotopies $H(-,-, \frac{3}{2}-2s)$ and $H(-,-,s)$ are reverse to each other (up to a scalar). 
Then by Lemma \ref{H+Hv=1lemma} below, the concatenation of the two homotopies is homotopic to the trivial homotopy 
$
H_1=\pi_1: S^{n-1}\times S^{k-1}\times [\frac{1}{4},1]\stackrel{}{\longrightarrow}  S^{n-1}$. 
It is clear that this trivial homotopy can be replaced by the projection
\[
S^{n-1}\times S^{k-1}\times [\frac{1}{4},1]\stackrel{1_{S^{n-1}\times S^{k-1}} \times c}{\longrightarrow}  S^{n-1}\times S^{k-1} \times \{1\}, \ \  \ \ \       c(s)=1, 
\] 
without changing the map between the double mapping cylinders. 
Therefore, the composite $\widetilde{\mathcal{G}}^\tau(N)\stackrel{\widetilde{\varphi}}{\longrightarrow} \widetilde{\mathcal{G}}^0(N) \stackrel{\widetilde{\psi}}{\longrightarrow} \widetilde{\mathcal{G}}^\tau (N)$ is homotopic to a map 
$
\mathcal{I}: \widetilde{\mathcal{G}}^\tau(N)\stackrel{}{\longrightarrow} \widetilde{\mathcal{G}}^\tau(N)
$
that is determined by the maps
\[
\begin{split}
1_{N_0\times S^{k-1}}& :N_0\times S^{k-1}\stackrel{=}{\longrightarrow} N_0\times S^{k-1},   \\
1_{S^{n-1}\times S^{k-1}}\times 4& :  S^{n-1}\times S^{k-1}\times [0, \frac{1}{4}]\longrightarrow  S^{n-1}\times S^{k-1}\times [0, 1], \ \ \  4(s)=4s  \\
1_{S^{n-1}\times S^{k-1}} \times c &:S^{n-1}\times S^{k-1}\times [\frac{1}{4},1]\stackrel{}{\longrightarrow}  S^{n-1}\times S^{k-1} \times \{1\}, \ \  \ \ \       c(s)=1   \\ 
1_{S^{n-1}}& :S^{n-1} \stackrel{=}{\longrightarrow} S^{n-1}. 
\end{split}
\]
Using the liner homotopy
\[
[0, 1]\times I \stackrel{}{\longrightarrow} [0,1],   \ \ \ \ 
(s, d)\longmapsto 
\left\{\begin{array}{cc}
\frac{4s}{3d+1}& 0 \leq s \leq \frac{3d+1}{4} \\
1 &\frac{3d+1}{4}\leq s\leq 1,
\end{array}\right.
\]
we can construct a linear homotopy from $\mathcal{I}$ to the identity map. It follows that the composite $\widetilde{\mathcal{G}}^\tau(N)\stackrel{\widetilde{\varphi}}{\longrightarrow} \widetilde{\mathcal{G}}^0(N) \stackrel{\widetilde{\psi}}{\longrightarrow} \widetilde{\mathcal{G}}^\tau (N)$ is homotopic to the identity map. 

Similarly, the composite $\widetilde{\mathcal{G}}^0(N)\stackrel{\widetilde{\psi}}{\longrightarrow} \widetilde{\mathcal{G}}^\tau(N) \stackrel{\widetilde{\varphi}}{\longrightarrow} \widetilde{\mathcal{G}}^0 (N)$ is homotopic to the identity map. Therefore, $\widetilde{\varphi}$ is a homotopy equivalence, and so is $\mathcal{G}^{\tau}(N)\stackrel{\varphi}{\longrightarrow}\mathcal{G}^{0}(N)$.
\end{proof}

\begin{remark}
The proof of Proposition \ref{localgrationprop} can be considerably simplified provided $k\geq 3$ and $N$ is simply connected. 

Indeed, By Diagram \eqref{ftaudiag} the map of homotopy pushouts $\varphi: \mathcal{G}^{\tau}(N)\stackrel{}{\longrightarrow}\mathcal{G}^{0}(N)$ induces an isomorphism on homology. If $k\geq 3$ and $N$ is simply connected, all the spaces in Diagram \eqref{ftaudiag} are simply connected, and so are $\mathcal{G}^{\tau}(N)$ and $\mathcal{G}^{0}(N)$. It follows that $\varphi$ is a homotopy equivalence by the Whitehead theorem.
\end{remark}

\begin{lemma}\label{H+Hv=1lemma}
Let $H: X\times I\stackrel{}{\longrightarrow} Y$ be a homotopy such that $H_0=H_1=f$. 
If $H_s=H_{1-s}$ for each $0\leq s\leq \frac{1}{2}$, then there is a homotopy
\[
\widetilde{H}: X\times I\times I\stackrel{}{\longrightarrow} Y
\]
such that $\widetilde{H}(-, -, 0)=H$, $\widetilde{H}(-, -, 1)= f\times 1_{I}$ and $\widetilde{H}(-, 0, -)=\widetilde{H}(-, 1, -)=f\times 1_{I}$.
\end{lemma}
\begin{proof}
Define 
\[
\widetilde{H}(x, s, t)=
\left\{\begin{array}{cc}
H(x, s) &  0\leq s\leq \frac{1}{2}(1-t)\\
H(x, \frac{1}{2}(1-t)) &\frac{1}{2}(1-t) \leq  s\leq \frac{1}{2}(1+t)\\
H(x, s) & \frac{1}{2}(1+t)\leq s\leq  1.
\end{array}\right.
\]
As $H_s=H_{1-s}$, the map $\widetilde{H}$ is well-defined. It is straightforward to check it is the required homotopy. 
\end{proof}

We want to prove that the top cell attachment of $\mathcal{G}^{\tau}(N)$ is locally inert when $k=2$ or $4$. 
For this the following lemma is crucial. 
\begin{lemma}\label{bundlegyrationlemma}
Suppose that $N$ is even dimensional. The following hold:
\begin{itemize}
\item[(1).]
if $k=2$, $n\equiv 0~{\rm mod}~4$ and $n\geq 4$, there is a principal bundle
\[
S^1\stackrel{}{\longrightarrow} \mathcal{G}^{0}(N) \stackrel{}{\longrightarrow} N\# \mathbb{C}P^{\frac{n}{2}};
\]
\item[(2).]
if $k=2$, $n\equiv 2~{\rm mod}~4$ and $n\geq 6$, there is a principal bundle
\[
S^1\stackrel{}{\longrightarrow} \mathcal{G}^{\tau_1}(N) \stackrel{}{\longrightarrow} N\# \mathbb{C}P^{\frac{n}{2}}
\]
for some $\tau_1: S^1\stackrel{}{\longrightarrow} SO(n)$;
\item[(3).] 
if $k=4$, $n\equiv 0~{\rm mod}~4$ and $n\geq 8$, there is a principal bundle
\[
S^3\stackrel{}{\longrightarrow} \mathcal{G}^{\tau_3}(N) \stackrel{}{\longrightarrow} N\# \mathbb{H}P^{\frac{n}{4}}
\]
for some $\tau_3: S^3\stackrel{}{\longrightarrow} SO(n)$.
\end{itemize}
Additionally, in all the cases the fibre inclusions are null homotopic. 
\end{lemma}
\begin{proof}
Recall there are the standard principal bundles 
\[
S^1\stackrel{}{\longrightarrow} S^{n+1}\stackrel{}{\longrightarrow} \mathbb{C}P^{\frac{n}{2}}\ \ \  {\rm and} \ \ \ 
S^3\stackrel{}{\longrightarrow} S^{n+3}\stackrel{}{\longrightarrow} \mathbb{H}P^{\frac{n}{4}}
\]
over projective spaces for appropriate values of $n$. Taking the pullback of the bundles along the canonical quotient maps $N\# \mathbb{C}P^{\frac{n}{2}} \stackrel{}{\longrightarrow} \mathbb{C}P^{\frac{n}{2}}$ and $N\# \mathbb{H}P^{\frac{n}{4}} \stackrel{}{\longrightarrow} \mathbb{H}P^{\frac{n}{4}}$, we obtain the principal bundles
\[
S^1\stackrel{}{\longrightarrow} E_1\stackrel{}{\longrightarrow} N\# \mathbb{C}P^{\frac{n}{2}}\ \ \  {\rm and} \ \ \ 
S^3\stackrel{}{\longrightarrow} E_2\stackrel{}{\longrightarrow}N\# \mathbb{H}P^{\frac{n}{4}},
\]
respectively. 

(1). Suppose that $k=2$, $n\equiv 0~{\rm mod}~4$ and $n\geq 4$. 
By \cite[Lemma 3.4]{HT23} or \cite[Theorem B]{Dua22}, the total manifold $E_1$ of the bundle $S^1\stackrel{}{\longrightarrow} E_1\stackrel{}{\longrightarrow} N\# \mathbb{C}P^{\frac{n}{2}}$ is diffeomorphic to $\mathcal{G}^{0}(N)$. This proves statement (1).

(2). Suppose that $k=2$, $n\equiv 2~{\rm mod}~4$ and $n\geq 6$. 
By \cite[Corollary 3.3]{HT23}, the total manifold $E_1$ of the bundle $S^1\stackrel{}{\longrightarrow} E_1\stackrel{}{\longrightarrow} N\# \mathbb{C}P^{\frac{n}{2}}$ is diffeomorphic to $\mathcal{G}^{\tau_1}(N)$ for some framing $S^1\stackrel{\tau_1}{\longrightarrow} SO(n)$. This proves statement (2).

(3). Suppose that $k=4$, $n\equiv 0~{\rm mod}~4$ and $n\geq 8$. 
By \cite[Corollary 3.3]{HT23}, the total manifold $E_2$ of the bundle $S^3\stackrel{}{\longrightarrow} E_2\stackrel{}{\longrightarrow}N\# \mathbb{H}P^{\frac{n}{4}}$ is diffeomorphic to $\mathcal{G}^{\tau_3}(N)$ for some framing $S^3\stackrel{\tau_3}{\longrightarrow} SO(n)$. This proves statement (3).

Finally, the fibre inclusions are null homotopic by \cite[Lemma 3.1]{HT23}.
\end{proof}

We can now prove the main theorem of this section.

\begin{proof}[Proof of Theorem \ref{gyration-inert-thmintro}]
For statement (1), by Lemma \ref{bundlegyrationlemma} (1) and (2), there is a principle bundle
\[
S^1\stackrel{}{\longrightarrow} \mathcal{G}^{\tau_1}(N) \stackrel{}{\longrightarrow} N\# \mathbb{C}P^{\frac{n}{2}}
\]
for some framing $\tau_1: S^1\stackrel{}{\longrightarrow} SO(n)$ and the fibre inclusion is null homotopic. Since the attaching map for the top cell of $N$ is inert, so is the attaching map for the top cell of $ N\# \mathbb{C}P^{\frac{n}{2}}$ by Theorem \ref{exsumthm}, and then the attaching map for the top cell of $\mathcal{G}^{\tau_1}(N)$ is inert by Theorem \ref{geoinertthm}. 
As Proposition \ref{localgrationprop} implies that $\mathcal{G}^{\tau}(N)\simeq \mathcal{G}^{\tau_1}(N)$ after localization away from $\mathcal{P}_2=\{2\}$, the attaching map for the top cell of $\mathcal{G}^{\tau}(N)$ is inert away from $2$. This proves statement (1).

For statement (2), in this case Lemma \ref{bundlegyrationlemma} (1) implies that $\mathcal{G}^{\tau_1}(N)=\mathcal{G}^{0}(N)$, and the above argument already shows that the attaching map for the top cell of $\mathcal{G}^{0}(N)$ is inert without localization. 

Statement (3) can be proved by the similar argument with the fact that $\mathcal{P}_4=\{2, 3\}$. 
\end{proof}

We have seen that the knowledge of \cite[Section 3]{HT23} is crucial for the proof of Theorem \ref{gyration-inert-thmintro}. Additionally, we can apply the knowledge of \cite[Section 3]{HT23} to prove inertness results for another family of manifolds and generalize Theorem \ref{geoinertthm}. 

Let 
\begin{equation}\label{SNMeq2}
S^{k-1}\stackrel{j}{\longrightarrow} N\stackrel{}{\longrightarrow} M
\end{equation}
be a fibre bundle of connected oriented closed smooth manifolds. 
For a given connected sum $M'\#M$, let $p: M'\# M\stackrel{}{\longrightarrow} M$ be the map that pinches the factor $M'$ to a point. Taking the pullback of the spherical fibre bundle (\ref{SNMeq2}) along the map $p$ induces a morphism of fibre bundles
\[
   \diagram 
       S^{k-1}\rto^-{j_{N}}\ddouble & E\rto^-{}\dto & M'\# M\dto^{p} \\ 
       S^{k-1}\rto^-{j} & N\rto^-{} & M
   \enddiagram 
\]
that defines the manifold $E$ and the maps $j_{N}$.

The following proposition is Parts (2) and (3) of Theorem \ref{geoinertthmintro}. 
Note that when $M'$ is a sphere, $E=N$ and Proposition \ref{geoinertprop} reduces to Theorem \ref{geoinertthm}. 

\begin{proposition}\label{geoinertprop}
Let $k=2$, $4$, or $8$. Suppose that the fibre inclusion $j$ of the fibre bundle (\ref{SNMeq2}) is null homotopic and ${\rm dim}(M)\geq k+2$. Then the following hold: 
\begin{itemize}
\item
the attaching map for the top cell of $E$ is inert if and only if the attaching map for the top cell of $M'\# M$ is inert;
\item
if the attaching map for the top cell of $M$ or $M'$ is inert, then the attaching map for the top cell of $E$ is inert.
\end{itemize} 
Additionally, the assertions hold after localization at any set of primes.
\end{proposition}
\begin{proof}
With the assumption that $j$ is null homotopic, \cite[Lemma 3.1]{HT23} and the Whitney embedding theorem imply that the fibre inclusion $j_N$ can be extended to a disk embedding. In particular, $j_N$ is null homotopic. Then applying Theorem \ref{geoinertthm} to the fibre bundle $S^{k-1}\stackrel{j_N}{\longrightarrow} E\stackrel{}{\longrightarrow} M'\# M$, we see that the attaching map for the top cell of $E$ is inert if and only if the attaching map for the top cell of $M'\# M$ is inert. This proves the first statement. 

For the second statement, the attaching map for the top cell of $M$ or $M'$ is inert by assumption. Theorem \ref{exsumthm} implies that the attaching map for the top cell of $M'\# M$ is inert. Then the attaching map for the top cell of $E$ is inert by the first statement.

The local version of the two statements follows by the same argument. 
\end{proof}

For the pullback bundle
\[
S^{k-1}\stackrel{}{\longrightarrow} E\stackrel{}{\longrightarrow} M'\# M,
 \]
by Proposition \ref{geoinertprop}, if the attaching map for the top cell of $M'$ is inert, then the attaching map for the top cell of $E$ is inert. 
For an alternative proof of this fact, recall that Huang-Theriault in \cite[Lemma 3.1]{HT23} shows that there is a diffeomorphism
\[
E\cong \mathcal{G}^{\tau}(M') \# N
\]
for the framing $\tau: S^{k-1}\stackrel{}{\longrightarrow} SO({\rm dim}(M))$ in Lemma \ref{bundlegyrationlemma}. In the proof of Theorem \ref{gyration-inert-thmintro}, we have showed that if the attaching map for the top cell of $M'$ is inert, then the attaching map for the top cell of the manifold $\mathcal{G}^{\tau}(M')$ is inert. 
Therefore, Theorem \ref{exsumthm} implies that the attaching map for the top cell of $ \mathcal{G}^{\tau}(M') \# N\cong E$ is inert. 

\newpage
%----------------------------------------------------------------------------------------------------------------------------------------------------------------------------------------------------------%
\section{Complete flag manifolds}
\label{sec: flag}
In this section, we show that the top cell attachments of complete flag manifolds are inert at large primes. The result, similar to Theorem \ref{gyration-inert-thmintro}, is an application of Theorem \ref{geoinertthm}.

Let $T^\ell=S^1\times \cdots \times S^1$ be the standard torus of rank $\ell$. The following result is a geometric version of \cite[Theorem 8.6]{The24b} with partial generalization. 
\begin{proposition}\label{TNMinertthm}
Let 
\[
T^\ell\stackrel{}{\longrightarrow} N\stackrel{}{\longrightarrow} M
\]
be a principal bundle of connected oriented closed smooth manifolds such that $N$ is simply connected and ${\rm dim}(M)\geq 4$. Then the attaching map for the top cell of $N$ is inert if and only if the attaching map for the top cell of $M$ is inert.
\end{proposition}
\begin{proof}
For any integer $k$ with $2\leq k\leq \ell+1$, let 
\[
i_{k-1}: T^{k-1}\stackrel{}{\longrightarrow}T^{\ell}
\]
be the inclusion into the first $(k-1)$-factors of $T^\ell$, that is, $i_{k-1}(z_1,\ldots, z_{k-1})=(z_1,\ldots, z_{k-1}, 1,\ldots, 1)$. The free action of $T^\ell$ on $N$ induces a free $T^{k-1}$-action on $N$ through $i_{k-1}$, and there is a principal bundle 
\[
T^{k-1}\stackrel{j}{\longrightarrow} N\stackrel{}{\longrightarrow} M_k,
\]
which defines the orbit manifold $M_k$. For convenience, denote $M_1=N$ and $T^0=\{1\}$. Note that $M_{\ell+1}=M$, and each $M_k$ is simply connected as $N$ is. Since $M_k\cong N/T^{k-1}$ and $S^1\cong T^{k}/T^{k-1}$, the canonical fibre bundle $T^{k}/T^{k-1}\stackrel{}{\longrightarrow} N/T^{k-1}\stackrel{}{\longrightarrow} N/T^{k}$ is isomorphic to a principal circle bundle
\[
S^{1}\stackrel{}{\longrightarrow} M_k\stackrel{}{\longrightarrow} M_{k+1}.
\] 
Then for each $1\leq k\leq \ell$, Theorem \ref{geoinertthm} implies that the attaching map for the top cell of $M_{k}$ is inert if and only if the attaching map for the top cell of $M_{k+1}$ is inert. It follows that the attaching map for the top cell of $N=M_{1}$ is inert if and only if the attaching map for the top cell of $M=M_{\ell+1}$ is inert. 
\end{proof}

Let $G$ be a simply connected compact Lie group with a maximal torus $T$. The homogeneous manifold $G/T$ is called the {\it complete flag manifold} of $G$. To study the inertness of the top cell attachment of $G/T$, we may start from the case when $G$ is simple. 

Recall that the rational cohomology of $G$ is an exterior algebra
\[
H^\ast(G;\mathbb{Q})\cong \Lambda (x_{2n_1+1},\ldots, x_{2n_l+1}),
\]
where each generator $x_{2n_i+1}$ is of degree $2n_i+1$ with $n_1\leq n_2\leq \ldots\leq n_l$. 
The integer $l$, equal to the rank of the maximal torus $T$, is called the {\it rank} of $G$, and the index set $\mathfrak{t}(G)=\{n_1, n_2,\ldots, n_l\}$ is called the {\it type} of $G$. The ranks and types of simply connected simple compact Lie groups are well-known and summarized in Table \ref{tabletypelie}.
\begin{table}[H]
{\begin{tabular}{@{}lll|lll@{}} 
\hline
$G$      & Type & Rank  &   $G_2$        &   $1, 5$  & $2$    \\   
$SU(n)$            & $1, 2, \ldots, n-1$   & $n-1$  &   $F_4$        &   $1, 5, 7, 11$     & $4$  \\
$Sp(n)$            & $1, 3, \ldots, 2n-1$  & $n$        &   $E_6$        &   $1, 4, 5, 7, 8, 11$  & $6$       \\ 
$Spin(2n)$            & $1, 3, \ldots, 2n-3, n-1$ & $n$          &   $E_7$        &   $1, 5, 7, 9, 11, 13, 17$  & $7$     \\ 
$Spin(2n+1)$            & $1, 3, \ldots, 2n-1$ & $n$  &    $E_8$        &   $1, 7, 11, 13, 17, 19, 23, 29$   & $8$ \\      
\hline    
\end{tabular}}
\caption{Ranks and types of simply connected simple compact Lie groups}
\label{tabletypelie}
\end{table}
A classical result of Serre \cite{Ser53} shows that a simple group $G$ can be decomposed into products of spheres at large primes. Any such prime $p$ is called a {\it regular} prime of $G$. More precisely, $p$ is regular for 
\begin{equation}\label{regularpeq}
\begin{array}{llll} 
   SU(n) & \quad\mbox{if}\ p\geq n\geq 3 & \quad F_4 & \quad\mbox{if}\ p\geq 12 \\  
   Sp(n) & \quad\mbox{if}\ p\geq 2n\geq 4 & \quad E_6 & \quad\mbox{if}\ p\geq 12 \\ 
   Spin(n) & \quad\mbox{if}\ p\geq n-1\geq 4 & \quad E_7 & \quad\mbox{if}\ p\geq 18 \\
   G_2 & \quad\mbox{if}\ p\geq 6 & \quad E_8 & \quad\mbox{if}\ p\geq 30.  
\end{array} 
\end{equation} 
\begin{theorem}\label{simpleG/Tthm}
Let $(G/T, \mathcal{P})$ be one of the following pairs:  
\[
\begin{split}
&(SU(n)/T^{n-1}, \{p\geq n\geq 3\}),   \ (Sp(n)/T^{n}, \{p\geq 2n\geq 4\}),   \ (Spin(n)/T^{\lfloor \frac{n}{2} \rfloor}, \{p\geq n-1\geq 4\}),   \ \\
&(G_2/T^{2}, \{p\geq 6\}), \  (F_4/T^{4}, \{p\geq 12\}),   \ 
(E_6/T^{6}, \{p\geq 12\}),  \  (E_7/T^{7}, \{p\geq 18\}),  \ (E_8/T^{8}, \{p\geq 30\}).
\end{split}
\]
Then the attaching map for the top cell of $G/T$ is inert after localization at any prime $p\in \mathcal{P}$.
\end{theorem}
\begin{proof}
We work in the homotopy category localized at any $p\in \mathcal{P}$. By Lemma \ref{pdt-inert-lemma}, the attaching map for the top cell of a finite product of spheres is inert. Then as $p$ is a regular prime of $G$, the attaching map for the top cell of $G$ is inert. Therefore, with the canonical principal bundle
\[
T\stackrel{}{\longrightarrow} G\stackrel{}{\longrightarrow}G/T, 
\]
Theorem \ref{TNMinertthm} implies that the attaching map for the top cell of $G/T$ is inert.
\end{proof}

The concrete results in Theorem \ref{simpleG/Tthm} for various simple Lie groups can be organized into a unified statement. For any compact Lie group $G$ denote
\[
\mathfrak{l}(G)={\rm max}\{j\in \mathfrak{t}(G)\}. 
\]
When $G$ is simple as listed in (\ref{regularpeq}), a prime $p$ is regular for $G$ if $p\geq  \mathfrak{l}(G)+1$. Also recall for the classical Lie groups of low ranks, $SU(2)\cong Sp(1)\cong Spin(3)\cong S^3$ and $Spin(4)\cong S^3\times S^3$. 
Then the results in Theorem \ref{simpleG/Tthm} are equivalent to that the attaching map for the top cell of $G/T$ is inert after localization at any prime $p\geq  \mathfrak{l}(G)+1$ provided that $G\not\cong S^3$. 
\begin{theorem}\label{G/Tthm}
Let $G$ be a nontrivial simply connected compact Lie group with a maximal torus $T$ such that $G\not\cong S^3$. Then the attaching map for the top cell of $G/T$ is inert after localization at any prime $p\geq  \mathfrak{l}(G)+1$. 
\end{theorem}
\begin{proof}
It is well-known that there is a diffeomorphism
\[
G/T\cong G_1/T_1 \times \cdots \times G_k/T_k,
\]
where each $G_i$ is a simply connected simple Lie group with $T_i$ a maximal torus. It is clear that $\mathfrak{l}(G)\geq \mathfrak{l}(G_i)$ for any $1\leq i\leq k$. 

If for some $i$ the simple Lie group $G_i$ is not isomorphic to $S^3$, Theorem \ref{simpleG/Tthm} implies that the attaching map for the top cell of $G_i/T_i$ is inert after localization at any prime $p\geq  \mathfrak{l}(G_i)+1$, and then is inert after localization at any prime $p\geq  \mathfrak{l}(G)+1$. By Lemma \ref{pdt-inert-lemma} the theorem follows in this case. 

Otherwise, $G_i\cong S^3$ for each $1\leq i\leq k$. It follows that $k\geq 2$ as $G\not\cong S^3$. In particular, $G/T$ has a factor $S^3/S^1\times S^3/S^1\cong S^2\times S^2$, and then the theorem follows from Lemma \ref{pdt-inert-lemma} in this case. 
\end{proof}

\newpage
%----------------------------------------------------------------------------------------------------------------------------------------------------------------------------------------------------------%
\section{Manifold embeddings}
\label{sec: emb}

In geometry and topology, submanifold embedding is a common context. Let $B\hookrightarrow N$ be a codimension $k$ embedding of connected oriented closed smooth manifolds and ${\rm dim}(N)=n=k+l$. 
Denote by $\nu$ the normal bundle of $B$ in $N$. By the classical tubular neighborhood theorem \cite[Theorem 11.1]{MS74}, there is a closed neighborhood $V$ of $B$ in $N$ and a diffeomorphism between $V$ and the disk bundle $D(\nu)$ of $\nu$. The diffeomorphism restricts to a diffeomorphism on the boundaries between $\partial V$ and $S(\nu)$, the sphere bundle of $\nu$. 
We shall not distinguish between $V$ and $D(\nu)$.
Denote by $N_c$ the closure of the complement of $V$ in $N$. Then $\partial N_c=\partial V$, and there is a pushout
\begin{equation}
\begin{gathered}
\label{Npushoutdiag}
\xymatrix{
\partial V=S(\nu) \ar[r]^{\iota_\nu} \ar[d]^{\iota_c} &
V=D(\nu)\ar[d]^{j_\nu} \\
N_c\ar[r]^{j_c} &
N,
}
\end{gathered}
\end{equation} 
where $\iota_c$ and $\iota_\nu$ are the inclusions of the respective boundaries, $j_c$ and $j_\nu$ are the inclusions of submanifolds. The pushout is clearly a homotopy pushout as well. 
This data of embedding will be used freely in the rest of the section. 

In this section, we study the inertness property around the embedding $B\stackrel{}{\hookrightarrow}N$ with the pushout (\ref{Npushoutdiag}) from three different perspectives. The results are consequences of the criteria for the inertness property around homotopy pushouts in Section \ref{sec: pushout}. Additionally, inspired by the discussions in this section, we will propose a generalization of the inertness problem for manifold embeddings in Subsection \ref{subsec: com-emb-prob}.

%------------------------------------------
\subsection{Interness via embedding}
\label{subsec: compemb}

$\, $

Consider the embedding $B\stackrel{}{\hookrightarrow}N$ with the pushout (\ref{Npushoutdiag}). Let $B_0$ be the manifold $B$ with a small open disk $\accentset{\circ}{D}^{l}$ removed: 
\[
B_0:= B-\accentset{\circ}{D}^{l}.
\]
Then $B_0$ is homotopy equivalent to the $(l-1)$-skeleton of $B$ and there is a cofibration
\[
S^{l-1}\stackrel{h_B}{\longrightarrow} B_0\stackrel{i_B}{\longrightarrow} B, \ \ \ 
\]
where $h_B$ is the inclusion of the boundary of $B_0$ and $i_B$ is the inclusion map. As the disk $D^l$ is contractible, the restriction of the normal bundle $\nu$ over $D^l\subseteq B$ is trivial. In particular, we can remove the restriction $\accentset{\circ}{V}|_{\accentset{\circ}{D}^{l}}=\accentset{\circ}{D}(\nu)|_{\accentset{\circ}{D}^{l}}\cong \accentset{\circ}{D}^{l}\times \accentset{\circ}{D}^{k}\cong \accentset{\circ}{D}^{n}$ from $N$ to obtain the deleted manifold 
\[
N_0:= N- (\accentset{\circ}{V}|_{\accentset{\circ}{D}^{l}})\cong N-\accentset{\circ}{D}^{n}.
\]
Then $N_0$ is homotopy equivalent to the $(n-1)$-skeleton of $N$ and there is a cofibration
\[
S^{n-1}\stackrel{h_N}{\longrightarrow} N_0\stackrel{i_N}{\longrightarrow} N,
\]
where $h_N$ is the inclusion of the boundary of $N_0$ and $i_N$ is the inclusion map. Similarly, let 
\[
V_0:=V-(\accentset{\circ}{V}|_{\accentset{\circ}{D}^{l}})\cong V-\accentset{\circ}{D}^{n}.
\]
From the construction, we see that the inclusion $V\stackrel{j_\nu}{\longrightarrow} N$ restricts to an inclusion $V_0\stackrel{j_{\nu0}}{\longrightarrow} N_0$, and there a commutative square
\begin{equation}\label{VNdiag}
\diagram
V_0 \rto^{j_{\nu0}} \dto^{i_V}  & N_0 \dto^{i_N} \\
V      \rto^{j_{\nu}}                 &  N
\enddiagram
\end{equation}
with $i_V$ the inclusion map. 
\begin{lemma}\label{VNlemma}
The square (\ref{VNdiag}) is a pushout and a homotopy pushout. 
\end{lemma}
\begin{proof}
It is clear that $N$ is the union of $V$ and $N_0$ along their common subcomplex $V_0$. This means that the square (\ref{VNdiag}) is a pushout. It is also a homotopy pushout since $i_V$ and $j_{\nu_0}$ are cofibrations. 
\end{proof}

\begin{theorem}\label{emb-inert-thm}
Let $B\stackrel{}{\hookrightarrow}N$ be the embedding with the pushout (\ref{Npushoutdiag}). Suppose that either of the following holds
\begin{itemize}
\item the embedding $B\stackrel{}{\hookrightarrow} N$ has a right homotopy inverse after looping; 
\item the restricted projection $N_0\stackrel{i_N}{\hookrightarrow} N\stackrel{}{\twoheadrightarrow} N/V$ has a right homotopy inverse after looping.
\end{itemize}
If the attaching map for the top cell of $B$ is inert, then the attaching map for the top cell of $N$ is inert.
\end{theorem}
\begin{proof}
The two hypotheses are respectively equivalent to the following two hypotheses:
\begin{itemize}
\item the map $V\stackrel{j_\nu}{\hookrightarrow} N$ has a right homotopy inverse after looping; 
\item the restriction map $V_0\stackrel{j_{\nu0}}{\longrightarrow} N_0$ of $j_\nu$ is inert.
\end{itemize}
The equivalence for the first hypothesis follows from the fact that $B\simeq V$ through the zero section, while the equivalence for the second hypothesis follows from the pushout diagram (\ref{VNdiag}) by Lemma \ref{VNlemma}. Consider the homotopy pushout square (\ref{VNdiag}). 
Under either of the hypotheses, Theorem \ref{pushoutthm} and Remark \ref{pushoutthm3} imply that if the map $\Omega i_V$ has a right homotopy inverse, then the map $\Omega i_N$ has a right homotopy inverse. Hence, to prove the theorem it suffices to show that $\Omega i_V$ has a right homotopy inverse. 

To this end, we may restrict the normal bundle $\nu$ over $B$ to $B_0$ through the pullback
\[
\diagram
D^k \rto^{} \ddouble  & V' \rto \dto^{i'_V}   & B_0 \dto^{i_B}\\
D^k \rto^{}                & V   \rto^{}               &B 
\enddiagram
\]
where $V'= V|_{B_0}\subseteq V$ with $i'_V$ the inclusion map. In other words, $V'=V-(V|_{\accentset{\circ}{D}^{l}})$. 
Since $V_0=V-(\accentset{\circ}{V}|_{\accentset{\circ}{D}^{l}})$ by construction, it follows that 
\[
V_0=V' \cup \partial V |_{D^l}=V' \cup S(\nu)|_{D^l}.
\]
In particular, the inclusion $V'\stackrel{i'_V}{\hookrightarrow} V$ factors as
\[
V'\stackrel{i'}{\longrightarrow}  V_0\stackrel{i_V}{\longrightarrow} V
\]
for an inclusion $i'$. By assumption the attaching map for the top cell of $B$ is inert, that is, the map $B_0\stackrel{i_B}{\longrightarrow} B$ has a right homotopy inverse after looping. Since the bundle projection $V\stackrel{}{\longrightarrow} B$ and its restriction $V'\stackrel{}{\longrightarrow} B_0$ are compatible homotopy equivalences by the previous diagram, it follows that $V'\stackrel{i'_V}{\hookrightarrow} V$ has a right homotopy inverse after looping. Then the equality $i'_V=i_V\circ i'$ implies that $V_0\stackrel{i_V}{\longrightarrow} V$ has a right homotopy inverse after looping. 

To summarize, we have showed that under either of the hypotheses in the theorem, if the attaching map for the top cell of $B$ is inert, then the map $\Omega i_V$ has a right homotopy inverse, and then the map $\Omega i_N$ has a right homotopy inverse. This means that the attaching map for the top cell of $N$ is inert. 
\end{proof}

%------------------------------------------
\subsection{Inertness via Thom spaces}
\label{subsec: Thom}

$\, $

Recall for a vector bundle $\xi$ over an $l$-dimensional closed manifold $X$
\[
\mathbb{R}^k\stackrel{}{\longrightarrow} E\stackrel{}{\longrightarrow} X,
\]
the {\it Thom space ${\rm Th}(\xi)$} is the quotient complex $D(\xi)/S(\xi)$ of the disk bundle $D(\xi)$ by the sphere bundle $S(\xi)$, and there is a cofibration
\[
S(\xi)\stackrel{\iota_\xi}{\longrightarrow} D(\xi)\stackrel{q_\xi}{\longrightarrow} {\rm Th}(\xi),
\] 
where $q_\xi$ is the quotient map. 
Therefore, the Thom space ${\rm Th}(\xi)$ can be view as a {\it singular manifold} of dimension ${\rm dim}(E)=n=k+l$, in the sense that it is locally homeomorphic to an Euclidean space at every point except at the {\it singular point} $[S(\xi)]=S(\xi)/S(\xi)$. Choosing a small disk $D^{n}\subseteq D(\xi)/S(\xi)-S(\xi)/S(\xi)$, we get a well-defined singular manifold 
\[
{\rm Th}(\xi)_0:=D(\xi)/S(\xi)-\accentset{\circ}{D}^{n},
\]
with boundary $S^{n-1}$. Then there is a cofibration
\[
S^{n-1}\stackrel{h_\xi}{\longrightarrow} {\rm Th}(\xi)_0\stackrel{i_\xi}{\longrightarrow} {\rm Th}(\xi),
\]
where the attaching map $h_\xi$ is the inclusion of the boundary and $i_\xi$ is the inclusion map. Similar to the case of manifolds, this cofibration is the homotopy cofibration for the top cell attachment of the singular manifold ${\rm Th}(\xi)$. 

Let $B\stackrel{}{\hookrightarrow}N$ be the embedding with the pushout (\ref{Npushoutdiag}) at the beginning of the section. There are the homotopy cofibrations
\[
S^{n-1}\stackrel{h_N}{\longrightarrow} N_0\stackrel{i_N}{\longrightarrow} N \ \ \ {\rm and} \ \ \ 
N_c \stackrel{j_c}{\longrightarrow} N\stackrel{q_c}{\longrightarrow} N/N_c, 
\]
where $h_N$ is the inclusion of the boundary of the deleted manifold $N_0$ with $i_N$ the  inclusion map, and the map $q_c$ is the quotient map. 
It is clear that the inclusion $N_c\stackrel{j_c}{\hookrightarrow} N$ in (\ref{Npushoutdiag}) factors as
\begin{equation}\label{NcN0Neq}
N_c\stackrel{\mathfrak{j}_c}{\hookrightarrow} N_0\stackrel{i_N}{\hookrightarrow} N,
\end{equation}
for an inclusion map $N_c\stackrel{\mathfrak{j}_c}{\hookrightarrow} N_0$. In particular, the inclusion $i_N$ induces a map $i_N/N_c: N_0/N_c \stackrel{}{\hookrightarrow} N/N_c$.

\begin{lemma}\label{thomlemma}
There are compatible homotopy equivalences ${\rm Th}(\nu)\simeq N/N_c$ and ${\rm Th}(\nu)_0\simeq N_0/N_c$: 
\[
\diagram
{\rm Th}(\nu)_0  \rto^{\simeq} \dto^{i_\nu}  &   N_0/N_c \dto^{i_N/N_c} \\
{\rm Th}(\nu)     \rto^{\simeq}                     &  N/N_c.     
\enddiagram
\]
\end{lemma}
\begin{proof}
Consider the homotopy cofibration diagram 
\begin{equation}\label{ThNNCdiag}
\diagram
S(\nu) \rto^{\iota_\nu}  \dto^{\iota_c}  & D(\nu) \rto^<<<<{q_\nu} \dto^{j_\nu}  &  {\rm Th}(\nu) \dto^{j_t } \\
N_c \rto^{j_c}  & N \rto^<<<<<{q_c}              &  N/N_c 
\enddiagram
\end{equation}
where the left square is the pushout diagram (\ref{Npushoutdiag}) and $j_t$ is the induced map. Since  the left square is also a homotopy pushout, it follows that the map $j_t: {\rm Th}(\nu)\stackrel{}{\longrightarrow} N/N_c$ is a homotopy equivalence. Restricting the map $j_\nu$ to the deleted manifolds, we obtain a map $j_{\nu 0}: D(\nu)_0\stackrel{}{\longrightarrow} N_0$ where $D(\nu)_0$ is the manifold $D(\nu)$ minus a small open disk in the interior of $D(\nu)$. By abuse of notation, the inclusions $S(\nu)\stackrel{\iota_\nu}{\longrightarrow} D(\nu)$ and $N_c\stackrel{j_c}{\longrightarrow} N$ factor as 
\[
S(\nu)\stackrel{\iota_\nu}{\longrightarrow} D(\nu)_0\stackrel{i_{D(\nu)}}{\longrightarrow} D(\nu), \ \ \ {\rm and} \ \ \ 
N_c\stackrel{j_c}{\longrightarrow} N_0\stackrel{i_N}{\longrightarrow} N,
\]
respectively. 
Therefore, Diagram (\ref{ThNNCdiag}) restricts to a homotopy cofibration diagram 
\begin{equation}\label{ThNNC0diag}
\diagram
S(\nu) \rto^{\iota_\nu}  \dto^{\iota_c}  & D(\nu)_0 \rto^<<<<{q_{\nu_0}} \dto^{j_{\nu0}}  &  {\rm Th}(\nu)_0 \dto^{j_{t0} } \\
N_c \rto^{j_c}  & N_0 \rto^<<<<{q_{c0}}              &  N_0/N_c, 
\enddiagram
\end{equation}
where the maps $q_{\nu_0}$ and $q_{c0}$ are the quotient maps, ${\rm Th}(\nu)_0\cong D(\nu)_0/S(\nu)$ by definition, and $j_{t0}$ is the induced map. It is clear that the left square of Diagram (\ref{ThNNC0diag}) is a pushout and hence a homotopy pushout. Then the induced map $j_{t0}: {\rm Th}(\nu)_0 \stackrel{}{\longrightarrow} N_0/N_c$ is a homotopy equivalence. Since Diagram (\ref{ThNNC0diag}) is the restriction of Diagram (\ref{ThNNCdiag}), the two homotopy equivalences are compatible. 
\end{proof}

\begin{theorem}\label{thom-inert-thm}
Let $B\stackrel{}{\hookrightarrow}N$ be the embedding with the pushout (\ref{Npushoutdiag}). Suppose that the inclusion $N_c\stackrel{\mathfrak{j}_c}{\hookrightarrow} N_0$ is inert. If the attaching map for the top cell of ${\rm Th}(\nu)$ is inert, then the attaching map for the top cell of $N$ is inert.
\end{theorem}
\begin{proof}
By Lemma \ref{thomlemma} and Diagrams (\ref{ThNNCdiag}) and (\ref{ThNNC0diag}), there is a diagram of homotopy cofibrations
\[
\diagram
N_c \ddouble \rto^{\mathfrak{j}_c}  & N_0 \dto^{i_N} \rto^<<<<{q_{c0}} & {\rm Th}(\nu)_0 \dto^{i_\nu} \\
N_c \rto^{j_c}                               & N                    \rto^<<<<{q_c}    &  {\rm Th}(\nu),
\enddiagram
\]
where the left square commutes by (\ref{NcN0Neq}). By assumption $N_c\stackrel{\mathfrak{j}_c}{\hookrightarrow} N_0$ is inert. Applying Theorem \ref{pushoutthm2} to the above diagram, we see that if the map $\Omega i_\nu$ has a right homotopy inverse, then the map $\Omega i_N$ has a right homotopy inverse. This proves the theorem by the definition of inertness. 
\end{proof}

%------------------------------------------
\subsection{Inertness via local embedding}
\label{subsec: localemb}

$\, $

Consider the embedding $B\stackrel{}{\hookrightarrow}N$ with the pushout (\ref{Npushoutdiag}) at the beginning of the section. It is called a {\it local embedding} if it factors through a disk embedding $D^n\stackrel{j_D}{\hookrightarrow} N$, that is, the embedding $B\stackrel{}{\hookrightarrow}N$ is a composite 
\[
B\stackrel{}{\hookrightarrow}D^n\stackrel{j_D}{\hookrightarrow} N.
\] 
Equivalently, the embedding $B\stackrel{}{\hookrightarrow}N$ is local if there exists a diffeomorphism $N\cong N\# S^n$ such that $B$ embeds into the connected summand $S^n$. 
In this case, the tubular neighborhood $V$ of $B$ can be chosen such that $V$ is entirely contained in $D^n$ and the embedding $V\stackrel{j\nu}{\hookrightarrow} N$ factors as
\[
V\stackrel{}{\hookrightarrow} D^n\stackrel{j_D}{\hookrightarrow} N\# S^n\cong N.
\]
Let $N_0$ be the manifold $N$ with a small open disk removed. Since any two disk embeddings are isotopic, there is no preference of the choice of the removed open disk. 
In particular, $N_0$ is homotopy equivalent to the $(n-1)$-skeleton of $N$ and there is a cofibration
\[
S^{n-1}\stackrel{h_N}{\longrightarrow} N_0\stackrel{i_N}{\longrightarrow} N,
\]
where $h_N$ is the inclusion of the boundary of $N_0$ and $i_N$ is the inclusion map. 
\begin{lemma}\label{localemblemma}
There is a sequence of embeddings 
\[
N_0\stackrel{i_{0c}}{\hookrightarrow} N_c \stackrel{i_{c0}}{\hookrightarrow} N_0
\]
such that the composite $i_{c0}\circ i_{0c}$ is isotopic to the identity map.
\end{lemma}
\begin{proof}
Recall that $N_c= N- \accentset{\circ}{V}$ and $N_0=N-\accentset{\circ}{D}^{n}$. We may choose a small disk embedding $D^n\stackrel{}{\hookrightarrow} V$. Then the sequence of embeddings $D^n\stackrel{}{\hookrightarrow} V\stackrel{j_\nu}{\hookrightarrow} N$ gives an embedding of complements $N_c \stackrel{i_{c0}}{\hookrightarrow} N_0$ (denoted by $\mathfrak{j}_c$ in (\ref{NcN0Neq})). 
Since the local embedding $B\hookrightarrow D^n\stackrel{j_D}{\hookrightarrow} N$ can be extended to the embedding $j_\nu: V\hookrightarrow D^n\stackrel{j_D}{\hookrightarrow} N$, there is an embedding of complements $N_0\stackrel{i_{0c}}{\hookrightarrow} N_c$. Further, since the two disk embeddings 
\[
D^n\stackrel{}{\hookrightarrow} V\hookrightarrow D^n\stackrel{j_D}{\hookrightarrow} N \ \ \ ~{\rm and}~\ \ \  D^n\stackrel{j_D}{\hookrightarrow} N
\] 
are isotopic in $N$, the composite $N_0\stackrel{i_{0c}}{\hookrightarrow} N_c \stackrel{i_{c0}}{\hookrightarrow} N_0$ on the complements is isotopic to the identity map. 
\end{proof}

\begin{proposition}\label{localembprop}
Let $B\stackrel{}{\hookrightarrow}N$ be a local embedding. Then the attaching map for the top cell of $N$ is inert if and only if the inclusion $N\backslash B\stackrel{}{\hookrightarrow} N$ of the complement of the embedding has a right homotopy inverse after looping.
\end{proposition}
\begin{proof}
By Lemma \ref{localemblemma}, the composition of the embeddings
\[
N_0\stackrel{i_{0c}}{\hookrightarrow} N_c \stackrel{i_{c0}}{\hookrightarrow} N_0 \stackrel{i_N}{\hookrightarrow} N
\]
is isotopic to $N_0 \stackrel{i_N}{\hookrightarrow} N$. It implies that $\Omega i_N$ has a right homotopy inverse if and only if $\Omega j_c=\Omega (i_N\circ i_{c0})$ has a right homotopy inverse. Also, it is clear that $N_c=N\backslash\accentset{\circ}{V}\simeq N\backslash B$. Then $N_c\stackrel{j_c}{\hookrightarrow} N$ has a right homotopy inverse after looping if and only if $N\backslash B\stackrel{}{\hookrightarrow} N$ has a right homotopy inverse after looping. The proposition follows by combining the above arguments. 
\end{proof}

\newpage
%----------------------------------------------------------------------------------------------------------------------------------------------------------------------------------------------------------%
\section{Open problems}
\label{sec: prob}
In this section, we propose eight open problems based on the materials in this paper, and on the inspiring work \cite{The24b} by Theriault as well. 

%-------------------------------------------
\subsection{From geometry to homotopy}

$\, $

In this paper there are several results which are only proved in geometric context. It is a natural problem to pursuing their counterparts in homotopy context. 

\begin{problem}\label{FEB-prop}
Let 
\[
F\stackrel{}{\longrightarrow} E\stackrel{}{\longrightarrow} B
\]
 be a homotopy fibration of connected Poincar\'{e} duality complexes with a single top cell. 
 If the attaching map for the top cell of $B$ is inert, is the attaching map for the top cell of $E$ is inert?
\end{problem}
When the homotopy fibration is a strict fibration, Problem \ref{FEB-prop} has a positive answer by Theorem \ref{FEBthm}. 

\begin{problem}\label{geoinert-prop}
Let $k=2$, $4$, or $8$. Let 
\[
S^{k-1}\stackrel{j}{\longrightarrow} N\stackrel{}{\longrightarrow} M
\]
be a homotopy fibration of connected Poincar\'{e} duality complexes with a single top cell such that the map $j$ is null homotopic and ${\rm dim}(M)\geq k+2$. 
Is the inertness of the attaching map for the top cell of $N$ equivalent to the inertness of the attaching map for the top cell of $M$?
\end{problem}
When the homotopy fibration is a fibre bundle of connected oriented closed smooth manifolds, Problem \ref{geoinert-prop} has a positive answer by Theorem \ref{geoinertthm}. Additionally, when $k=2$ and the homotopy fibration is principal, Theriault \cite{The24b} proved that if the attaching map for the top cell of $N$ is inert then the attaching map for the top cell of $M$ is inert. These results suggest that a positive answer to Problem \ref{geoinert-prop} may be possible. 

For a homotopy fibration $S^{k-1}\stackrel{j}{\longrightarrow} N\stackrel{}{\longrightarrow} M$, if the map $j$ is null homotopic then the homotopy fibre $S^{k-1}$ has to be an $H$-space. This is the reason for the values of $k$ in Problem \ref{geoinert-prop} and Theorem \ref{geoinertthm}. Nevertheless, it is well-known that odd dimensional spheres are $H$-spaces at any odd prime. Therefore, it is natural to raise a local version of Problem \ref{geoinert-prop}.   
\begin{problem}\label{kodd-inert-prop}
Work in the homotopy category after localization at an odd prime $p$. 
Let 
\[
S^{2k-1}\stackrel{j}{\longrightarrow} N\stackrel{}{\longrightarrow} M
\]
be a homotopy fibration of connected Poincar\'{e} duality complexes with a single top cell such that the map $j$ is null homotopic and ${\rm dim}(M)\geq 2k+2$. 
Is the inertness of the attaching map for the top cell of $N$ equivalent to the inertness of the attaching map for the top cell of $M$?
\end{problem}
Suppose that $M$ is $(2k-1)$-connected and the homotopy fibration is principal, Theriault in \cite[Theorem 9.1]{The24b} proved that after rationalization if the attaching map for the top cell of $N$ is inert then the attaching map for the top cell of $M$ is inert. Further, he applied this special case to reprove a classical result of Halperin and Lemaire \cite{HL87}, which states that the attaching map for the top cell of a Poincar\'{e} duality complex is rationally inert unless its rational cohomology algebra is generated by a single element. It is an interesting problem to looking for a local refinement of the result of Halperin and Lemaire \cite{HL87}.

\begin{problem}\label{HLrefined-prop}
Let $M$ be a simply connected Poincar\'{e} duality complex of dimension $n\geq 4$.
Suppose that the rational cohomology algebra of $M$ is not generated by a single element. Does there exist a linear function $\ell(n)$ depending only on $n$ such that the attaching map for the top cell of $M$ is inert after localization at any prime $p$ with $p>\ell(n)$? If so, find an explicit $\ell(n)$ satisfying the property.  
\end{problem}
Theriault \cite{The24b} has made an important progress on Problem \ref{HLrefined-prop}. In particular, he proved in \cite[Theorems 6.4 and 7.5]{The24b} that if $M$ is $(m-1)$-connected with $m<n$ and $H^m(M;\mathbb{Z})$ has a $\mathbb{Z}$-summand generated by a class $x$ with $x^2=0$, then the attaching map for the top cell of $M$ is inert after localization at any prime $p$ with $p>n/2+1$. This suggests that a positive answer to Problem \ref{HLrefined-prop} may be possible and the function $\ell(n)$ may be close to $n/2+1$.

%-------------------------------------------
\subsection{Surgery}

$\, $

Surgery is a basic operation in geometric topology. A {\it $\mathit{k}$-surgery} on an $n$-dimensional closed smooth manifold $M$ is an operation removing a framed $\mathit{k}$-embedding $f: D^{n-k}\times S^{k}\hookrightarrow  M$ and replacing it with $S^{n-k-1}\times D^{k+1}$, with {\it effect} the $n$-dimensional closed smooth manifold
\[
M^\prime:=(M\backslash f(\accentset{\circ}{D}^{n-k}\times  S^{k}))\cup_{S^{n-k-1}\times S^{k}} (S^{n-k-1}\times D^{k+1}).
\]
It is known that in general the inertness property is not preserved by surgery from the basic example $S^m\times S^{n-m}$. However, we showed in Theorem \ref{gyration-inert-thmintro} that a certain type of surgery preserves the inertness of top cell attachments. It is an interesting problem to looking for other specific surgeries preserving inertness. 

\begin{problem}\label{surgery-prop}
Work in a homotopy category with or without localization at certain primes. Let $M$ be an $n$-dimensional connected oriented closed smooth manifold. Suppose that the attaching map for the top cell of $M$ is inert. For which values of $k$, the attaching map for the top cell of a $k$-surgery effect $M^\prime$ is inert? Moreover, does there exist such $k$ depending only on the dimension and the connectivity of $M$?
\end{problem}
Suppose that $M$ is $(m-1)$-connected. We guess that the possible values of $k$ for Problem \ref{surgery-prop} could be integers less than $m$ or greater than $n-m$, since it is likely that in these cases a $k$-surgery is roughly to produce more cells rather than to kill cells in certain sense. 

%-------------------------------------------
\subsection{Blow ups}

$\, $

We recall the description in~\cite[Section 2]{LS08} and \cite[Section 5]{HT24d} of the blow up construction. 
Consider the embedding $B\stackrel{}{\hookrightarrow}N$ with the pushout (\ref{Npushoutdiag}) and follow the notations and constructions in Section \ref{sec: emb}. 
Suppose that the normal bundle $\nu$ of $B$ is of even dimension and supports a complex structure. Let $P\nu \stackrel{}{\longrightarrow} B$ be the canonical projectivization of the complex normal bundle $\nu$. The canonical complex line bundle of $P\nu$ gives a circle bundle projection $S(\nu)\stackrel{q}{\longrightarrow} P\nu$. 
We may replace the inclusion of the boundary $S(\nu)\stackrel{\iota_{\nu}}{\longrightarrow} D(\nu)$ 
in~(\ref{Npushoutdiag}) with the projection $S(\nu)\stackrel{q}{\longrightarrow} P\nu$ and define the {\it blow up $\widetilde{N}$ of $N$ along $B$} by the pushout  
\begin{equation}
\begin{gathered}
\label{blowuppushoutdiag}
\xymatrix{
\partial N_c=S(\nu) \ar[r]^<<<{q} \ar[d]^{\iota_c} &
P\nu\ar[d]^{} \\
N_c\ar[r]^{} &
\widetilde{N}.
}
\end{gathered}
\end{equation}
In particular, when $B$ is a point the blow up construction is topologically a {\it projectively stabilization} of $N$ following \cite{HT24d}, that is, $\widetilde{N}$ is homeomorphic to the connected sum of $N$ with the complex projective space $\mathbb{C}P^{{\rm dim}(N)/2}$. 

We are interested in the inertness property around blow up constructions.
\begin{problem}\label{blowup-inert-prob}
Consider the blow up $\widetilde{N}$ of $N$ along $B$ defined by the pushout (\ref{blowuppushoutdiag}). 
\begin{itemize}
\item[(1).] If the attaching map for the top cell of $N$ is inert, is the attaching map for the top cell of $\widetilde{N}$ is inert?
\item[(2).] If the attaching map for the top cell of $B$ is inert, is the attaching map for the top cell of $\widetilde{N}$ is inert?
\end{itemize}
\end{problem} 
The two problems are reasonable. Problem \ref{blowup-inert-prob} (1) concerns when a blow up construction preserves the inertness property. When $B$ is a point, $\widetilde{N}\cong N\#\mathbb{C}P^{{\rm dim}(N)/2}$ and then Theorem \ref{exsumthm} provides a positive answer to this problem in this case. In general case Huang-Theriault \cite{HT24d} showed that blow up construction is a type of {\it fibrewise connected sum}. Therefore, a positive answer to Problem \ref{blowup-inert-prob} (1) is about to proving a fibrewise version of Theorem \ref{exsumthm}.  

For Problem \ref{blowup-inert-prob} (2), recall that Theorem \ref{FEBthmintro} shows that for a fibre bundle $F\stackrel{}{\longrightarrow} M\stackrel{}{\longrightarrow} B$, the inertness of the top cell attachment for the base manifold $B$ implies the inertness of the top cell attachment for the total manifold $M$. In particular, if the attaching map for the top cell of $B$ is inert, then the attaching map for the top cell of the projectivization $P\nu$ is inert. Moreover, though the blow up construction $\widetilde{N}$ is not fibered, the piece $\widetilde{N}-\accentset{\circ}{N}_c$ of $\widetilde{N}$ is indeed fibered over $P\nu$ with disk fibre. In other words, part of $\widetilde{N}$ is fibered over a manifold with an inert top cell attachment. Then it is a natural question whether this is enough to prove the inertness of the top cell attachment for $\widetilde{N}$. 

%-------------------------------------------
\subsection{Complement embeddings}
\label{subsec: com-emb-prob}

$\, $

Consider the embedding $B\stackrel{}{\hookrightarrow}N$ with the pushout (\ref{Npushoutdiag}) and follow the notations and constructions in Section \ref{sec: emb}. By \eqref{NcN0Neq} there is a sequence of embeddings
\[
j_c: N_c\stackrel{\mathfrak{j}_c}{\hookrightarrow} N_0\stackrel{i_N}{\hookrightarrow} N,
\]
where $N_c$ is the complement of an open tubular neighborhood of $B$ in $N$, and $N_0$ is the manifold $N$ with a small disk removed. It is clear that $N_c$ is homotopy equivalent to the complement $N\backslash B$. 
These imply the following lemma immediately, which is the sufficiency part of Proposition \ref{localembprop}. 

\begin{lemma}\label{Nclemma}
If the complement embedding $N\backslash B\stackrel{}{\hookrightarrow} N$ has a right homotopy inverse after looping, then the attaching map for the top cell of $N$ is inert.  ~$\qqed$
\end{lemma}

Lemma \ref{Nclemma} suggests studying a refinement of the inertness problem for manifold embeddings. 

\begin{problem}\label{strongint-prob}
For an embedding $B\stackrel{}{\hookrightarrow}N$, whether the complement embedding $N\backslash B\stackrel{}{\hookrightarrow} N$ has a right homotopy inverse after looping?
\end{problem}
When $B$ is a point, the problem reduces to the usual inertness problem for the top cell attachment of $N$. Hence, Problem \ref{strongint-prob} can be viewed as a {\it fibrewise inertness problem} over $B$ in $N$. 
In particular, Proposition \ref{localembprop} implies that the fibrewise inertness problem is equivalent to the usual inertness problem when the embedding $B\stackrel{}{\hookrightarrow}N$ is local. 

The fibrewise inertness problem should be quite difficult in general. Nevertheless, some preliminary discussions can be made by considering Diagram \eqref{Npushoutdiag}. For the normal bundle $\nu$ of $B$ in $N$, the inclusion of the boundary $S(\nu) \stackrel{}{\hookrightarrow} D(\nu)$ is homotopic to the bundle projection $\partial N_c=S(\nu)\stackrel{s_\nu}{\longrightarrow} B$ of the sphere bundle $S(\nu)$. Therefore, Diagram \eqref{Npushoutdiag} implies a homotopy pushout
\[
\diagram
S(\nu)=\partial N_c \dto^{s_\nu}  \rto^<<<{\iota_c}  &  N_c \dto^{j_c} \\
B \rto^{}                               &   N.
\enddiagram
\]
Applying Theorem \ref{pushoutthm} and Remark \ref{pushoutthm3} to the homotopy pushout, we can prove the following proposition immediately.

\begin{proposition}\label{Ncinert-prop}
Let $B\stackrel{}{\hookrightarrow}N$ be the embedding with the pushout (\ref{Npushoutdiag}). Suppose that either of the following holds:
\begin{itemize}
\item the embedding $B\stackrel{}{\hookrightarrow} N$ has a right homotopy inverse after looping;
\item the inclusion of the boundary $\partial N_c\stackrel{\iota_c}{\hookrightarrow} N_c$ is inert.
\end{itemize} 
If the spherical bundle projection $\partial N_c\stackrel{s_\nu}{\longrightarrow} B$ has a homotopy section after looping, then the complement embedding $N_c\stackrel{j_c}{\hookrightarrow} N$ has a right homotopy inverse after looping. $\qqed$
\end{proposition}
Suppose that $B\stackrel{}{\hookrightarrow} N$ is a {\it framed embedding}, that is, the normal bundle $\nu$ of $B$ in $N$ is trivial. Then the spherical bundle $S(\nu)$ is trivial and the bundle projection $\partial N_c\stackrel{s_\nu}{\longrightarrow} B$ has a section. Accordingly, Proposition \ref{Ncinert-prop} can be applied under either condition stated in the proposition. 

%-------------------------------------------
\subsection{Flag manifolds}

$\, $

Let $G$ be a compact connected Lie group. Let $P$ be a parabolic subgroup of $G$. The quotient $G/P$ is called a {\it flag manifold}. In particular, when $P=T$ is a maximal torus of $G$, $G/T$ is called the {\it complete flag manifold} of $G$. By a classical result of Halperin and Lemaire \cite{HL87} it is known that the attaching map for the top cell of a flag manifold is rationally inert. 
For a special case, in Theorem \ref{homthmintro} we showed the inertness property of the top cell attachments for complete flag manifolds at large primes and gave an explicit range of the allowable primes. Therefore, it is a natural guess that similar result could hold for general flag manifolds.
\begin{problem}\label{flag-inert-prob}
Determine for which primes $p$ the attaching map for the top cell of a flag manifold $G/T$ is inert after localization at $p$. 
\end{problem}
One way to study Problem \ref{flag-inert-prob} may be through the intersection theory of flag manifolds. As reviewed in \cite{DZ22} there are fruitful results on Schubert calculus for the intersection theory of flag manifolds. Further, in Section \ref{sec: int} we showed various results for inertness around intersection theory. This illustrates interesting connections between Schubert calculus and the inertness problem for flag manifolds, and it is possible to attack Problem \ref{flag-inert-prob} along the results and ideas of Section \ref{sec: int}.

\newpage
%%%%%%%%%%%%%%%%%%%%%%%%%%%%%%%%%%%%%%%%%%%%%%%%%%%%%%%%%%%%%%%%%%%%%%

%%% The bibliography %%%
\bibliographystyle{amsalpha}

\end{sloppypar}
\end{document}